\documentclass[10pt]{article}
\PassOptionsToPackage{hyphens}{url}
\usepackage[maxbibnames=49,sorting=nyt,arxiv=true]{agn_bib}
\usepackage{verbatim}
\usepackage[multiple]{footmisc}
\usepackage{amsfonts,color}
\usepackage{longtable}
\usepackage{makecell, multirow, tabularx}
\usepackage{amsmath,amssymb}
\usepackage{adjustbox}
\usepackage{epsfig}
\usepackage{lscape}
\usepackage{amssymb}
\usepackage{graphicx}
\usepackage[utf8]{inputenc}	
\usepackage{amsthm,amssymb}
\usepackage{aligned-overset}
\usepackage{amsfonts}
\usepackage[english]{babel}
\usepackage{dsfont}
\usepackage{bbm}
\usepackage[T1]{fontenc}
\usepackage{colonequals}
\usepackage{geometry}
\usepackage{mathtools}
\usepackage{csquotes}
\usepackage{lipsum}
\usepackage{array}
\usepackage{tablefootnote}
\usepackage{wrapfig}
\usepackage[utf8]{inputenc}
\usepackage{cancel}
\usepackage[dvipsnames]{xcolor}
\usepackage{accents}
\usepackage{marvosym}
\usepackage{centernot}
\usepackage{tikz}
\usetikzlibrary{arrows,calc}
\usepackage{nicematrix}
\usepackage[new]{old-arrows}
\usepackage{titling}
\usepackage{multicol}

\usepackage{amsthm}
\newcounter{theoremCounter}
\numberwithin{theoremCounter}{section}
\theoremstyle{plain}
\newtheorem{lem}[theoremCounter]{Lemma}
\newtheorem{prop}[theoremCounter]{Proposition}
\newtheorem{thm}[theoremCounter]{Theorem}
\newtheorem{cor}[theoremCounter]{Corollary}
\newtheorem{conjecture}[theoremCounter]{Conjecture}
\theoremstyle{definition}
\newtheorem{definition}[theoremCounter]{Definition}
\newtheorem{rem}[theoremCounter]{Remark}
\newtheorem{example}[theoremCounter]{Example}

\newcommand{\red}[1]{{\color{red}#1}}

\newcommand{\id}{{\boldsymbol{\mathbbm{1}}}}

\allowdisplaybreaks[1]
\makeindex
\newcommand{\R}{\mathbb{R}}

\newcommand{\norm}[1]{\lVert #1 \rVert}
\newcommand{\matr}[1]{\begin{pmatrix}#1\end{pmatrix}}

\newcommand{\SO}{{\rm SO}}
\newcommand{\OO}{{\rm O}}
\newcommand{\DD}{\mathrm{D}}
\newcommand{\WW}{\mathrm{W}}
\newcommand{\innerproduct}[1]{\langle #1 \rangle}
\newcommand{\iprod}{\innerproduct}

\newcommand*\dif{\mathop{}\!\mathrm{d}}

\newcommand{\notiff}{%
  \mathrel{{\ooalign{\hidewidth$\not\phantom{"}$\hidewidth\cr$\iff$}}}}

\newlength{\dhatheight}

\DeclareMathOperator{\lin}{lin}
\DeclareMathOperator{\diag}{diag}
\DeclareMathOperator{\Sym}{Sym}
\DeclareMathOperator{\Cof}{Cof}
\DeclareMathOperator{\dev}{dev}

\DeclareMathOperator{\sym}{sym}
\DeclareMathOperator{\bfsym}{\textbf{sym}}
\DeclareMathOperator{\tr}{tr}
\DeclareMathOperator{\iso}{iso}
\DeclareMathOperator{\Old}{Old}
\DeclareMathOperator{\ZJ}{ZJ}
\DeclareMathOperator{\GN}{GN}

\DeclareMathOperator{\TR}{TR}
\DeclareMathOperator{\NH}{NH}

\DeclareMathOperator{\vol}{vol}
\def\barr{\begin{array}}
\def\tr{\textnormal{tr}}

\def\sk{\textnormal{skew}}

\def\dd{\displaystyle}

\def\barr{\begin{array}}
\def\earr{\end{array}}
\def\becn{\begin{equation*}}
\def\endec{\end{equation}}
\def\endecn{\end{equation*}}
\def\C{\mathbb{C}}
\def\H{\mathbb{H}}

\makeatletter
\let\@fnsymbol\@arabic
\makeatother
\numberwithin{equation}{section}

\newcommand{\scal}[2]{\left\langle#1,#2\right\rangle}

\global\long\def\bR{\mathbb{R}}

\title{Hypo-elasticity and logarithmic strain}
\makeatletter
\renewcommand*{\@fnsymbol}[1]{\ifcase#1\or1\or2\or3\or4\or5\or6\or7\or$\ast$\else\fi}
\makeatother

\newcommand{\todo}[1][]{\textcolor{red}{TODO\ifx&#1&\else: #1\fi}}

\begin{document}
%%%%%title%%%%%%
\title{Hypo-elasticity, Cauchy-elasticity, corotational stability and monotonicity in the logarithmic strain}
%%%%%authors%%%%%
\author{
Patrizio Neff\thanks{
Patrizio Neff, University of Duisburg-Essen, Head of Chair for Nonlinear Analysis and Modelling, Faculty of Mathematics, Thea-Leymann-Stra{\ss}e 9,
D-45127 Essen, Germany, email: patrizio.neff@uni-due.de}
, \qquad
Sebastian Holthausen\thanks{
Sebastian Holthausen, University of Duisburg-Essen, Chair for Nonlinear Analysis and Modelling,  Faculty of Mathematics, Thea-Leymann-Stra{\ss}e 9,
D-45127 Essen, Germany, email: sebastian.holthausen@uni-due.de
}, \qquad
Marco Valerio d'Agostino\thanks{
Marco Valerio d'Agostino, GEOMAS, INSA-Lyon, Université de Lyon, 20 avenue Albert Einstein, 69621, Villeurbanne cedex, France, email: marco-valerio.dagostino@insa-lyon.fr
}, \\[0.8em]
Davide Bernardini\thanks{
Davide Bernardini, Department of Structural and Geotechnical Engineering, Sapienza University of Rome, Rome, Italy, e-mail: davide.bernardini@uniroma1.it
}, \quad
Adam Sky\thanks{
Adam Sky, Institute of Computational Engineering and Sciences, Department of Engineering, Faculty of Science, Technology and Medicine, University of Luxembourg, 6 Avenue de la Fonte, L-4362 Esch-sur-Alzette, Luxembourg, e-mail: adam.sky@uni.lu
}%$^{\;\,,}$\thanks{corresponding author}
, \quad
Ionel-Dumitrel Ghiba\thanks{
Ionel-Dumitrel Ghiba, Alexandru Ioan Cuza University of Ia\c si, Department of Mathematics,  Blvd. Carol I, no. 11, 700506 Ia\c si, Romania;  Octav Mayer Institute of Mathematics
of the Romanian Academy, Ia\c si Branch, 700505 Ia\c si, email: dumitrel.ghiba@uaic.ro
}, \quad and \quad
Robert J. Martin\thanks{
Robert J. Martin,  Lehrstuhl für Nichtlineare Analysis und Modellierung, Fakultät für Mathematik, Universität Duisburg-Essen, Thea-Leymann Str. 9, 45127 Essen, Germany, email: robert.martin@uni-due.de
}
}
\maketitle
\vspace{-0,6cm}
%%%%%abstract%%%%%
\begin{abstract}
\noindent We combine the rate-formulation for the objective, corotational Zaremba-Jaumann rate
	\begin{align*}
	\frac{\DD^{\ZJ}}{\DD t} [\sigma] = \H^{\ZJ}(\sigma).D, \qquad D = \sym \DD v\,,
	\end{align*}
operating on the Cauchy stress $\sigma$, the Eulerian strain rate $D$ and the spatial velocity $v$ with the novel \enquote{corotational stability postulate} %(CSP), $\left\langle \frac{\DD^{\ZJ}}{\DD t}[\sigma], D \right\rangle > 0 \; \forall \, D \neq 0$,
\begin{equation}\label{eq:abstract_CSP}
	\Bigl\langle \frac{\DD^{\ZJ}}{\DD t}[\sigma], D \Bigr\rangle > 0 \qquad \forall \, D\in\Sym(3)\setminus\{0\} \tag{CSP}
\end{equation}
to show that for a given isotropic Cauchy-elastic constitutive law $B \mapsto \sigma(B)$ in terms of the left Cauchy-Green tensor $B = F F^T$, the induced fourth-order tangent stiffness tensor $\H^{\ZJ}(\sigma)$ is positive definite if and only if for $\widehat{\sigma}(\log B)\colonequals \sigma(B)$, the strong monotonicity condition (\eqref{eq:abstract_tstsmpp}) in the logarithmic strain is satisfied:
	\begin{align}\label{eq:abstract_tstsmpp}
	&\sym \DD_{\log B} \widehat \sigma(\log B) \in \Sym^{++}_4(6) \tag{TSTS-M$^{++}$}\\
	\implies \qquad &\langle \widehat\sigma(\log B_1) - \widehat\sigma(\log B_2), \log B_1 - \log B_2 \rangle > 0 \qquad \forall \, B_1, B_2 \in \Sym^{++}(3)\,,\; B_1 \neq B_2\,.\nonumber
	\end{align}
Thus \eqref{eq:abstract_CSP} implies \eqref{eq:abstract_tstsmpp} and vice-versa, and both imply the invertibility of the hypo-elastic material law between the stress and strain rates given by the tensor $\H^{\ZJ}(\sigma)$, since $\H^{\ZJ}(\sigma)$ is accordingly positive definite. Notably, \eqref{eq:abstract_tstsmpp} is one way to characterize the fundamental notion of \enquote{stress increases with strain}. The same characterization remains true for the corotational Green-Naghdi rate as well as the corotational logarithmic rate, conferring the corotational stability postulate \eqref{eq:abstract_CSP} together with the monotonicity in the logarithmic strain tensor \eqref{eq:abstract_tstsmpp} a far reaching generality. It is conjectured that this characterization of \eqref{eq:abstract_CSP} holds for a large class of reasonable corotational rates. The result for the logarithmic rate is based on a novel chain rule for corotational derivatives of isotropic tensor functions.
\end{abstract}
\bigskip
\textbf{Keywords:} nonlinear elasticity, hyperelasticity, rate-formulation,
Eulerian setting, hypo-elasticity, Cauchy-elasticity, material stability, corotational derivatives, objective derivatives, chain rule, constitutive inequalities, logarithmic strain, stress increases with strain \\
\\
\textbf{Mathscinet} classification:
%35A01 (Existence for PDEs),
15A24, 73G05, 73G99, 74B20
%74H20 (Existence of Solutions of dynamical problems in solid mechanics)
%
\clearpage
\tableofcontents
\section{Introduction}
In the theory of finite elasticity, numerous mathematical models can be used to describe the relation between the stress acting on an elastic body and its deformation. To this day, however, it remains uncertain how to identify models which predict a \emph{physically reasonable} material behaviour.
Although a vast number of constitutive requirements have been suggested in order to ensure mechanically plausible material behaviour of Cauchy-elastic or hyperelastic laws \cite{hill1968constitutivea,hill1968constitutiveb,hill1970constitutive,nefftranslationbecker} -- many of which can be described as a condition of \emph{stress increasing with strain} -- there is no consensus on which requirements should be considered necessary or even desirable. Truesdell considered this to be the \textbf{\enquote{Hauptproblem} of nonlinear elasticity} \cite{truesdell1956ungeloste}.

Recently, constitutive questions in solid mechanics have increasingly been approached from a data-driven perspective \cite{jin2023recent}. In particular, machine learning techniques are being employed not only for finding numerical approximations to given problems \cite{lagaris1998artificial,raissi2019physics,agn_voss2021morrey2,jin2023recent}, but also for developing constitutive laws in nonlinear elasticity and related fields \cite{flaschel2021unsupervised,thakolkaran2022nn,agn_klein2021neural,flaschel2023automated,tacc2024benchmarking,linka2023new,klein2024}. While this approach can indeed provide new material models which satisfy known constitutive requirements \cite{agn_klein2021neural,klein2024}, it cannot give a meaningful answer to the main question posed by Truesdell's Hauptproblem: which requirements \emph{should} be imposed on the relation between stress and strain?

In this contribution, we therefore continue the classical investigation of nonlinear elasticity theory from an analytical point of view. In particular, we present
%This contribution is
a variation on a theme initially introduced by Hill \cite{hill1968constitutivea,hill1968constitutiveb, hill1978} following first steps by Truesdell \cite{truesdell55hypo} and Noll \cite{Noll55}.

Truesdell and Noll rediscovered and formalized, in the mid fifties of the last century, the work of Hencky \cite{biezeno1928, Hencky1928, Hencky1929} and Oldroyd \cite{oldroyd1950} on the formulation of rate-type equations encoding the constitutive law in isotropic nonlinear elasticity involving necessarily {\bf objective time derivatives} \cite{Marsden83} of the Cauchy stress tensor $\frac{\DD^{\sharp}}{\DD t}[\sigma]$ such that the constitutive law is expressed as
	\begin{align}\label{eq1}
	\frac{\DD^{\sharp}}{\DD t}[\sigma] =\H^*(\sigma).\, D, \qquad \textnormal{with} \quad D = \sym \DD v \quad \textnormal{the Eulerian stretching}.
	\end{align}
Here, $\H^*(\sigma)$ is a fourth-order tangent stiffness tensor mapping symmetric tensors to symmetric tensors (minor symmetry) \cite{ericksen1958hypo, xiao97_2}, which does not necessarily have to be self-adjoint (major symmetric), and \eqref{eq1} needs to be integrated along the loading path. Such models are known as hypo-elastic\footnote
{
Today, hypo-elastic models appear prominently in rate-formulations of finite strain elasto-plasticity, splitting the stretch rate $D = D_e + D_p$ additively into elastic and plastic parts and setting $\frac{\DD^{\sharp}}{\DD t}[\sigma] = \H^*(\sigma).(D - D_p)$, see e.g.~\cite{baghani2014, bruhns1999, eshraghi2013, peshkov2019, volokh2013} and references therein. Otherwise, current FEM-software such as Abaqus\texttrademark\ and LS-DyNA require the input of the constitutive law for nonlinear elasticity in the rate format with $\H^*(\sigma)$ for some specified objective rate and some spatial stress tensor.
}.

In contrast, in nonlinear \emph{Cauchy-elasticity}, the stress-strain relation is defined in absolute terms via a stress response function
\[
	\sigma\colon\Sym^{++}(3)\to\Sym(3)\,,\quad B\mapsto\sigma(B)
	\,,
\]
mapping the left Cauchy-Green $B=FF^T$ corresponding to a deformation gradient $F$ to the Cauchy stress tensor $\sigma(B)$. While hypo-elasticity can be applied to more general problems in continuum mechanics, it was first shown by Noll\footnote%
{%
Noll \cite{Noll55} considered only the Zaremba-Jaumann rate $\frac{\DD^{\ZJ}}{\DD t}$.%
} \cite{Noll55}, to the surprise of Truesdell \cite{truesdellremarks}, that every Cauchy-elastic material is hypo-elastic if the mapping $B\mapsto\sigma(B)$ is isotropic and invertible\footnote%
{%
It is quite remarkable that Truesdell and Noll \cite{Truesdell65} largely ignored this invertibility assumption for the Cauchy stress $\sigma$ in their subsequent development for nonlinear elasticity.%
}. Under these conditions, hypo-elasticity can be considered an equivalent formulation of nonlinear elasticity, more specifically a \emph{rate-formulation} in which stress increments are functionally related to strain increments in order to describe the constitutive law \cite{ericksen1958hypo}. In that case, for any given objective time derivative $\frac{\DD^{\sharp}}{\DD t}$, there exists a naturally \textbf{induced tangent stiffness tensor} $\H^{\sharp}(\sigma)$ with
\begin{align}
	\H^{\sharp}(\sigma).D = \frac{\DD^{\sharp}}{\DD t}[\sigma]
	\,.
\end{align}
However, even for perfect nonlinear elasticity, there is a range of limitations and problems for such a rate-formulation:
\begin{itemize}
	\item[1)] There are infinitely many different objective time derivatives which can be seen as covariant derivatives \cite{kolev2024objective} (not all of them are Lie-derivatives \cite{federico2022, kolev2024objective, Marsden83}).
	\item[2)] One can choose different spatial stress tensors to encode the same constitutive law.
	\item[3)] The choice of $\H^*(\sigma)$ is open but a priori restricted to isotropy \cite{Noll55}.
	\item[4)] The rate \eqref{eq1} might not be integrable towards a Cauchy-elastic or hyperelastic model \cite{bernstein1958}.
\end{itemize}
We will address these issues here as they present themselves.

At the same time, in nonlinear Cauchy-elasticity, it was and is still unclear\footnote
{%
%the final solution of the Hauptproblem came in 1977 with the work J. M. Ball [1]. He showed that the triplet of weakened convexity conditions, viz., polyconvexity, Morrey’s quasiconvexity, and the rank 1 convexity(the Legendre-Hadamard condition, the nonstrict version of the strong ellipticity) is exactly what is missing.
{\v{S}}ilhav{\'y} \cite{Silhavy2019}
argues that \enquote{the final solution of the Hauptproblem came in 1977 with the work of J. M. Ball [who] showed that the triplet of weakened convexity conditions, viz., polyconvexity, Morrey’s quasiconvexity, and the rank 1 convexity [...] is exactly what is missing}.
%considers Truesdell's \enquote{Hauptproblem} to be solved with the constitutive requirement of poly, quasi and rank-one convexity in hyperelasticity.
%However, this is just a mathematical side of the discussion adapted to the direct methods of the calculus of variations.%
However, while these generalized convexity properties indeed provide a satisfying answer to central existence problems in nonlinear elasticity from the point of view of the direct methods of the calculus of variations, they do not fully adress the mechanical questions raised by Truesdell.
}
as to what constitutive assumption should reasonably be placed on the nonlinear elasticity law $B \mapsto \sigma(B)$ in general to ensure \emph{physically reasonable response}, loosely connected to the idea that \emph{stress increases with strain},\footnote
{
But which stress? Which strain? What does \enquote{increase} mean? We will arrive at a tentative answer at the end of this paper.
}
which Truesdell called the \enquote{Hauptproblem} of nonlinear elasticity \cite{truesdell1956ungeloste}.

In an attempt to answer this question, Hill modified \eqref{eq1} into a formulation for the spatial Kirchhoff stress tensor $\tau = J \, \sigma$ and required
	\begin{align}\label{eq2}
	\langle \frac{\DD^{\sharp}}{\DD t}[\tau] ,D\rangle>0 \qquad \textnormal{for all}\quad D \in \Sym(3) \! \setminus \! \{0\}
	\end{align}
for a specific subclass of objective derivatives, among them the Zaremba-Jaumann rate (cf.~\cite{jaumann1905, jaumann1911geschlossenes, zaremba1903forme}). He continued to examine the consequences of imposing \eqref{eq2} for the constitutive law $B\mapsto \tau(B)$ and obtained with his complicated method of Lagrangian axes for the Zaremba-Jaumann rate the result that (here stated for hyperelasticity) $\tau(B)=\widehat{\tau}(\log B)$ satisfies the condition $\DD_{\log B} \widehat \tau(\log B) \in \Sym^{++}_4(6)$ if and only if \eqref{eq2} is satisfied. Leblond \cite{leblond1992constitutive} took up Hill's development and proved a similar result for the Cauchy stress $\sigma$ in the hyperelastic setting.

In \cite{tobedone} the authors generalize Leblond's results to the Cauchy-elastic case and show that for isotropic nonlinear elasticity, the relation involving the logarithmic strain tensor \cite{Neff_Osterbrink_Martin_Hencky13, NeffGhibaLankeit, NeffGhibaPoly}
	\begin{equation}
	\label{eqCPSdef}
	\begin{alignedat}{2}
	   \forall \,D\in\Sym(3) \! \setminus \! \{0\}:
	   ~
	   \scal{\frac{\DD^{\ZJ}}{\DD t}[\sigma]}{D} > 0
	   \quad
	   &\iff
	   \quad
	   \log B
	   \mapsto
	   \widehat\sigma(\log B)
	   \;\textrm{is strongly Hilbert-monotone\footnotemark} \\
		&\iff \quad \sym \DD_{\log B} \widehat \sigma(\log B) \in \Sym^{++}_4(6)
	\end{alignedat}
	\end{equation}
\footnotetext
{
See Appendix \ref{appendixnotation}
}
holds for the corotational Zaremba-Jaumann objective derivative of the Cauchy stress $\sigma$, which is given by
	\begin{align}
	\frac{\DD^{\ZJ}}{\DD t}[\sigma] = \frac{\DD}{\DD t}[\sigma] - W \, \sigma + \sigma \, W, \qquad W = \sk(\dot F \, F^{-1})\,.
	\end{align}
Here, $B = F \, F^T$ is the left Cauchy-Green tensor and $F = \DD \varphi$ is the Fréchet derivative of the deformation $\varphi$. Relation \eqref{eqCPSdef} implies the \textbf{T}rue-\textbf{S}tress \textbf{T}rue-\textbf{S}train strict Hilbert-\textbf{M}onotonicity (TSTS-M$^+$)
\begin{equation}
	\scal{\widehat\sigma(\log B_1)-\widehat\sigma(\log B_2)}{\log B_1-\log B_2} > 0
	\qquad
	\forall \, B_1\neq B_2\in\Sym^{++}(3),
\end{equation}
setting throughout
\begin{equation}
	\widehat\sigma(\log V)
	=
	\sigma(V)
	=
	\sigma(B)
	= \widehat \sigma(\log B) \qquad 
	\textrm{for $V = \sqrt{B}$ \quad by abuse of notation}.
\end{equation}
One goal of this paper is to show that the same characterization can be obtained when using the corotational Green-Naghdi rate as well as the corotational logarithmic rate. Hence, TSTS-M$^+$ is one way to constitutively characterize the notion that \enquote{stress increases with strain}. \\
\\
Because of its key role, we will refer to the requirement
	\begin{align}
	\langle \frac{\DD^{\circ}}{\DD t}[\sigma],D\rangle>0 \quad \forall \, D \in \Sym(3) \! \setminus \! \{0\} \qquad \textnormal{as the {\bf corotational stability postulate (CSP)},}
	\end{align}
where $\frac{\DD^{\circ}}{\DD t}$ denotes an arbitrary corotational derivative of the Cauchy stress $\sigma$, e.g.~the Zaremba-Jaumann derivative $\frac{\DD^{\ZJ}}{\DD t}$. \\
\\
Furthermore, it is important to note a complete \textbf{paradigm shift}: instead of concentrating on the hypo-elastic formulation \eqref{eq1} per se, we presently only consider the hypo-elastic rate-formulation to elucidate constitutive issues for a given \emph{Cauchy-elastic} law. Indeed, \\
	\begin{center}
	\fbox{
	\begin{minipage}[h!]{0.85\linewidth}
		\centering
		\textbf{our aim is to understand what kind of given isotropic constitutive Cauchy-elastic laws $B \mapsto \sigma(B)$ satisfy $\langle \frac{\DD^{\ZJ}}{\DD t}[\sigma],D \rangle = \langle \H^{\ZJ}(\sigma).D,D \rangle > 0 \quad \forall \, D \in \Sym(3) \! \setminus \! \{0\}$.}
	\end{minipage}}
	\end{center}
\bigskip
In order to properly answer our guiding question, also in the context of more general corotational rates, we, nevertheless, need to dive deep into the hypo-elasticity framework \eqref{eq1}. In this way we shift attention from $\frac{\DD^{\ZJ}}{\DD t}[\sigma]$ towards $\H^{\ZJ}(\sigma)$. Therein, following the paradigm shift, we are only considering those fourth-order stiffness tensors $\H^*(\sigma)$ that are {\bf induced} by a given invertible Cauchy-elastic law, thus circumventing any integrability issue \cite{ericksen1958hypo} and already answering one of the problems alluded to above. Moreover, we formally motivate \textbf{CSP} for the Cauchy stress $\sigma$ from stability requirements in isotropic linear elasticity. For the Zaremba-Jaumann and Green-Naghdi rate, we provide a semi-explicit representation for the induced tangent stiffness tensors $\H^{\rm ZJ}(\sigma)$ and $\H^{\rm GN}(\sigma)$, respectively. More precisely, we obtain in Section \ref{sechypoelastic} the formulas (for an invertible law $B \mapsto \sigma(B) = \widehat \sigma(\log B)$)
	\begin{align}\label{eq6}
	\H^{\rm ZJ}(\sigma).D = \DD_B\sigma(B).[B \, D + D \, B] \qquad \textnormal{and} \qquad \H^{\rm GN}(\sigma).D =\DD_B\sigma(B).[2 \, V D \, V]
	\end{align}
together with (see Proposition \ref{propinvert1})
	\begin{align}
	\label{eq1.14a}
	\left.\begin{array}{r}
	\det\H^{\rm ZJ}(\sigma)\neq 0\\
	\det \H^{\rm GN}(\sigma)\neq 0
	\end{array}\right\} \quad \iff \quad \det \DD _B\sigma(B) \neq 0 \quad \iff \quad \det \DD_{\log B} \widehat \sigma(\log B) \neq 0,
	\end{align}
where the determinant of a symmetric tensor $\C \colon \Sym(3) \to \Sym(3)$ is defined in the Appendix \ref{appendixnotation}. With the representation \eqref{eq6}, which is true for every invertible law $B\mapsto \sigma(B)$ (this invertibility\footnote
{
Invertibility of the Cauchy stress $\sigma$ is independent of the used strain measure. In \cite{richter1948isotrope} Richter writes: \enquote{\textsl{In Verallgemeinerung des Hookeschen Gesetzes nennt man ein Material rein elastisch, wenn die [Cauchy] Spannungen  in umkehrbar eindeutiger Weise von den Dehnungen abhängen},} which translates to \enquote{in generalization of Hooke’s law, a material is called purely elastic, if the [Cauchy] stresses depend in a uniquely reversible way on the stretches.}
}
was already required by Richter \cite{richtertranslation, richter1948isotrope, richter1949hauptaufsatze, Richter50, Richter52} for idealized perfect isotropic nonlinear elasticity), we can rewrite the characterization of \textbf{CSP} for the Zaremba-Jaumann rate as
	\begin{align}
	\langle \frac{\DD^{\rm ZJ}}{\DD t}[\sigma],D\rangle&=\langle \H^{\ZJ}(\sigma). D,D \rangle = \langle \DD_B\sigma(B).\, [BD+DB],D\rangle=\langle \DD_{\log B}\widehat{\sigma}(\log B).\DD_B \log B. [D \, B + B \, D], D \rangle >0 \notag \\
	\label{eq8}
	&\iff \qquad \langle \DD_{\log B} \widehat{\sigma}(\log B).D,D\rangle=\langle \sym \DD_{\log B}\widehat{\sigma}(\log B).D,D\rangle>0.
	\end{align}
The latter equivalence can be viewed as a pure statement of matrix analysis for isotropic tensor functions and we provide a suitable context for this observation in Sections \ref{sec5}--\ref{sec6}. From the matrix analysis viewpoint it furthermore transpires that the equivalence \eqref{eq8} also holds true for the Green-Naghdi and the logarithmic rate and thus possibly for a large class of reasonable objective corotational rates. For the Green-Naghdi rate and an arbitrary corotational rate we formulate corresponding conjectures for any isotropic tensor function $\sigma$. If the latter conjecture is true, this would confer to the equivalence CSP $\iff$ TSTS-M$^{++}$ a far reaching universality. In fact, the result for the Green-Naghdi rate will be shown in an upcoming contribution, using entirely different methods than those described here. The result for the logarithmic rate will be given in Section \ref{corotationalstablog} following still other lines of thought owing to the exceptional structure of the logarithmic rate \cite{xiao98_1}. Thus, the equivalence of CSP with TSTS-M$^{++}$ is already not confined to the Zaremba-Jaumann rate. The problem with more general corotational rates resides in not immediately having a useful representation like the ones presented in \eqref{eq6}.

In Section \ref{appendchainrule} we recall some properties of corotational rates in general and prove a new universal chain rule for corotational rates, which allows to extend our findings to the logarithmic rate. 

With this chain rule at hand, we are able to find a general representation of $\H^{\circ}(\sigma)$ in the spirit of \eqref{eq6} for a large class of reasonable corotational rates defined by their spins and with the help of \cite{neffkoro2024} we ascertain \eqref{eq1.14a} for this class of spins. We propose a further subclass of these spins defining the \enquote{positive} corotational rates which merits future investigation based on geometrical arguments in the hope of extending our characterization of CSP.

In the Appendix, we introduce the notation and gather necessary requisites from linear algebra and matrix analysis, among them the Daleckii-Krein formula for derivatives of primary matrix functions. Moreover, we give examples that monotonicity of $\widehat{\sigma}$ in $\log B$ is different from monotonicity of $\sigma$ in $B$. We also provide an example of a hyperelastic formulation that globally satisfies CSP (and therefore TSTS-M$^{++}$). On the other hand, we show that a polyconvex slightly compressible Neo-Hooke model gives an invertible Cauchy stress relation for which $\H^{\ZJ}(\sigma)$ is not positive definite, such that CSP is not satisfied. In addition, we recall the implications of CSP for purely volumetric energy functions and the known problem of the Zaremba-Jaumann rate in zero grade hypo-elasticity. 

It is our belief and motivation that the corotational stability postulate CSP is equivalent to TSTS-M$^{++}$ for a large class of corotational rates and that it can be used as a cornerstone for a novel local existence result in finite strain Cauchy-elasticity with pre-stress, cf.~Blesgen et al.~\cite{blesgen2024}. As a byproduct of our development we appreciate again the pivotal role that the logarithmic strain tensor $\log V$ should play in isotropic nonlinear elasticity \cite{Becker1893,  Hencky1928, martin2018history, NeffGhibaLankeit, NeffGhibaPoly, nefftranslationbecker}.

\section{Introduction to three-dimensional hypo-elasticity}
A hypo-elastic material, in the sense of Truesdell \cite{truesdellremarks} and Noll \cite{Noll55}, obeys a constitutive law of the form\footnote
{
Rate-formulations in terms of the Kirchhoff stress $\tau = J \, \sigma$ will not be considered here. See e.g.~Korobeynikov \cite{korobeynikov2020}, Bellini and Federico \cite{bellini2015}, and Federico et al.~\cite{federico2024} for results in this direction. 
}
	\begin{align}
	\label{eqthedoublestr}
	\frac{\DD^{\sharp}}{\DD t}[\sigma] = \H^{*}(\sigma) . D \quad \iff \quad \underbrace{D}_{\textnormal{the stretching}} = [\H^{*}(\sigma)]^{-1} . \underbrace{\frac{\DD^{\sharp}}{\DD t}[\sigma]}_{\textnormal{the stressing}} = \mathbb{S}^{*}(\sigma) . \frac{\DD^{\sharp}}{\DD t}[\sigma],
	\end{align}
where
	\begin{itemize}
	\item $\frac{\DD^{\sharp}}{\DD t}[\sigma]$ is an appropriate \textbf{objective} rate of the Cauchy stress tensor $\sigma$,
	\item $\H^{*}(\sigma)$ is a constitutive \textbf{fourth-order} tangent stiffness \textbf{tensor},
	\item $\mathbb{S}^{*}(\sigma) = [\H^{*}(\sigma)]^{-1}$ is a constitutive \textbf{fourth-order} tangent compliance \textbf{tensor} and
	\item $D = \sym \, \DD v$ is the \textbf{Eulerian strain rate tensor}, measuring the spatial rate of deformation \newline or \enquote{stretching}, where $v$ describes the spatial velocity at a point $\xi$ in the current configuration.
	\end{itemize}
In this formulation it can be neatly seen that the elastic stretching $D$ depends solely on the current stress level $\sigma$ together with the rate of stress $\frac{\DD^{\sharp}}{\DD t}[\sigma]$ (\enquote{the stressing} in the parlance of Romano\footnote
{
The work of Romano et al.~\cite{romano2011, romano2017, romano2014} is particularly intriguing. They use a novel differential-geometric framework for setting up rate-formulations of nonlinear elasticity using suitable Lie-derivatives $\mathcal{L}_{v_\varphi}$. Their final result is, however, substantially departing from classical concepts like hyperelasticity or Cauchy-elasticity, which both need a reference configuration. In our rate-formulation, to the contrary, we remain in the classical context. While Romano et al.~use a different rate formulation, they arrive conceptually at constitutive equations of the type
	\begin{align*}
	\mathcal{L}_{v_\varphi}(e) = \mathbb{S}(\sigma). \mathcal{L}_{v_\varphi}(\sigma), \qquad \textnormal{where $e$ is some spatial strain measure,}
	\end{align*}
and assume that the tangent compliance tensor $\mathbb{S}$ is invertible. In \cite[eq.~8]{romano2017}, it is assumed further that $\mathbb{S}$ is positive definite as well as major symmetric, and even that $\mathbb{S}(\sigma) = \DD_{\sigma}^2 \Xi(\sigma)$ for a convex complementary energy potential $\Xi\colon \Sym(3) \to \R$.
}
et al.~\cite{romano2011}) as seen by the objective derivative. \\
\\
Throughout this work, without loss of generality, we will start directly with a Cauchy-elastic stress response instead of a hyperelastic energy potential\footnote
{
Truesdell \cite[p.~88]{truesdell1955} writes: \enquote{[...]~if the strain energy in the elastic case has an especially simple form, no particularly simple form of the resulting hypo-elastic coefficients [of $\H^{\sharp}(\sigma)$] is to be expected.}
}
(cf.~Ogden \cite[sec.~4.2]{Ogden83}), since this gives us a better take on how to arrive at the corresponding hypo-elastic formulation and we can a priori determine and evaluate constitutive requirements on the Cauchy stress without being concerned by explicit calculations about integrability of the Cauchy stress-strain law towards hyperelasticity. \\
\\
As for objective time derivatives $\frac{\DD^{\sharp}}{\DD t}[\sigma]$ of the Cauchy stress tensor $\sigma$, there are infinitely many different possible choices, see e.g.~\cite{kolev2024objective, neffkoro2024}. However, some are more reasonable than others. Paraphrasing\footnote
{
Statements are taken from a discussion on imechanica.org (\url{https://imechanica.org/node/1646}), regarding the \enquote{correct} choice of a time derivative.
}
Andrew Norris (cf. \cite{Norris2008} for an example of his work), the most reasonable time derivatives are those that are \textbf{corotational} and \textbf{objective}, since \textsl{corotational} \enquote{just means that the rate is taken with respect to a frame that is rotating relative to where you are sitting}, while \textsl{objective} \enquote{means that this extra rate of rotation, or spin, should be defined only by the underlying spins in the problem and the velocity gradient $L$} (see e.g.~also the related works by Xiao et al. \cite{xiao97, xiao98_2}, Aubram \cite{Aubram2017}, Federico \cite{bellini2015, palizi2020consistent}, Fiala \cite{fiala2009, fiala2016, fiala2020objective}, Govindjee \cite{govindjee1997}, Korobeynikov et al.~\cite{korobeynikov2018, korobeynikov2023, korobeynikov2023book, korobeynikov2024}, Pinsky et al.~\cite{pinsky1983} and Zohdi \cite{zohdi2006}. From now on we denote by $\frac{\DD^{\circ}}{\DD t}$ a general corotational derivative. \\
\\
For this reason, here we restrict our attention primarily to two specific objective and corotational derivatives, namely the \textbf{Zaremba-Jaumann} derivative, using the spin given by the vorticity\footnote
{
It can be shown that $W = \dot{R}R^T + R \, \sk(\dot{U}U^{-1}) \, R^T$ cf. Gurtin et al. \cite{Gurtin2010}, Nasser et al. \cite{mehrabadi1987} and the books by Ogden \cite[p.126]{Ogden83} and Truesdell \cite[p.21]{truesdell1966}.
}
tensor $W = \sk \, L, \; L = \dot F \, F^{-1}$,
	\begin{equation}
	\label{ZJrate01}
	\boxed{
		\begin{aligned}
		\frac{\DD^{\ZJ}}{\DD t}[\sigma] \colonequals \frac{\DD}{\DD t}[\sigma] + \sigma \, W - W \, \sigma &= Q^W \, \frac{\DD}{\DD t}[(Q^W)^T \, \sigma \, Q^W] \, (Q^W)^T  \\
		&\textnormal{for} \ Q^W(t) \in \OO(3)\quad \textnormal{such that} \;  W =\dot{Q}^W \, (Q^W)^T\in \mathfrak{so}(3)\,,
		\end{aligned}
	}
	\end{equation}
and the \textbf{Green-Naghdi derivative} or \textbf{polar rate} (cf. \cite{bellini2015, naghdi1961}), using the polar spin $\Omega = \dot R \, R^T$ related to the polar decomposition $F = R \, U$:
	\begin{align}
	\label{GNrate01}
	\boxed{\frac{\DD^{\GN}}{\DD t}[\sigma] \colonequals \frac{\DD}{\DD t}[\sigma] + \sigma \, \Omega - \Omega \, \sigma = R \, \frac{\DD}{\DD t}[R^T \, \sigma \, R] \, R^T \qquad \textnormal{for}\  R(t) \in \OO(3)\quad\qquad \textnormal{with} \ \Omega\colonequals \dot{R} \, R^T \in \mathfrak{so}(3).}
	\end{align}
Here, $\frac{\DD}{\DD t}[\sigma] = \dot \sigma$ denotes the material or substantial time-derivative. For more information on the notation, we refer to the Appendix \ref{appendixnotation}. \\
\\
The crucial property that characterizes \textbf{any objective derivative}\cite{kolev2024objective} is the so called \textbf{frame-indifference}, which means that the time derivatives do not depend on a Euclidean transformation of the observer, in the sense that (cf.~Lemma \ref{frameindZJ})
	\begin{align}
	\label{eq0013.122}
	\frac{\DD^{\sharp}}{\DD t}[Q(t) \, \sigma(t) \, Q^T(t)] = Q(t) \, \frac{\DD^{\sharp}}{\DD t} [\sigma] \, Q^T(t), \qquad \forall \, Q(t) \in \OO(3).
	\end{align}
Note that the material derivative $\frac{\DD}{\DD t}[\sigma]$ does not satisfy this transformation law. \\
\\
An important consequence of considering corotational derivatives is that any isotropic scalar function $I(\sigma(t))$ of the stress $\sigma(t)$ is stationary if the corotational rate $\frac{\DD^{\circ}}{\DD t}[\sigma]$ is zero, where we call $I(\sigma(t))$ an isotropic scalar function if
	\begin{align}
	I(\sigma(t)) = I(Q^T(t) \, \sigma(t) \, Q(t)) \qquad \forall \, Q(t) \in \OO(3).
	\end{align}
This property is crucial when developing plasticity theories (cf. Prager \cite{prager1961}). In this sense, only corotational derivatives leave physical properties of the Cauchy stress tensor invariant and are therefore the only ones considered here, answering partly the first issue raised in the introduction.
\begin{lem}
Let $\frac{\DD^{\circ}}{\DD t}$ be any corotational derivative of $\sigma$, i.e.\ of the form
	\begin{align}
	\label{eqlemma3.1a}
	\hspace*{-8pt} \frac{\DD^{\circ}}{\DD t}[\sigma] = Q^{\circ} \, \frac{\DD}{\DD t}[(Q^{\circ})^T \, \sigma \, Q^{\circ}] \, (Q^{\circ})^T \quad \iff \quad (Q^{\circ})^T \, \frac{\DD^{\circ}}{\DD t}[\sigma] \, Q^{\circ} = \frac{\DD}{\DD t}[(Q^{\circ})^T \, \sigma \, Q^{\circ}] \quad \textnormal{with} \quad \dot{Q}^{\circ} \, (Q^{\circ})^T \equalscolon \Omega^{\circ}  ,
	\end{align}
with a corresponding orthogonal frame $t\mapsto Q^{\circ}(t) \in \OO(3)$. If $I(\sigma(t))$ is any isotropic scalar invariant of $\sigma$, then
	\begin{align}
	\label{eqlemma3.1}
	\frac{\DD^{\circ}}{\DD t}[\sigma] = 0 \qquad \implies \qquad \frac{\DD}{\DD t}I(\sigma(t)) = 0.
	\end{align}
\end{lem}
\begin{proof}
In order to see \eqref{eqlemma3.1}, from the representation of $\frac{\DD^{\circ}}{\DD t}[\sigma]$ in \eqref{eqlemma3.1a} we have
	\begin{align}
	\frac{\DD^{\circ}}{\DD t}[\sigma] = 0 \qquad \iff \qquad \frac{\DD}{\DD t}[(Q^{\circ})^T \, \sigma \, Q^{\circ}] = 0.
	\end{align}
Using the standard chain rule and the isotropy of $I(\sigma(t))$ we conclude
	\begin{align}
	\frac{\DD}{\DD t}I(\sigma(t)) &= \frac{\DD}{\DD t}I((Q^{\circ})^T \, \sigma \, Q^{\circ}) = \langle \DD I((Q^{\circ})^T \, \sigma \, Q^{\circ}) ,\left[\frac{\DD}{\DD t}[(Q^{\circ})^T \, \sigma \, Q^{\circ}]\right] \rangle = 0. \qedhere
	\end{align}
\end{proof}
\begin{rem}[Corotational derivatives as special Lie-derivatives]
The corotational rate may be interpreted as the Lie derivative with respect to spatial rotation defined by $Q^{\circ}(t)$. Thus, let $\phi$ define the (rotation) mapping $\phi(x) = Q^{\circ}x$, then the corotational rate is the \enquote{Lie-type} derivative $\frac{\DD^{\circ}}{\DD t}[\sigma] = \phi_*[\frac{\DD}{\DD t} \phi^*(\sigma)] = Q^{\circ} \, \frac{\DD}{\DD t}[(Q^{\circ})^T \, \sigma \, Q^{\circ}] \, (Q^{\circ})^T$ (cf.~\cite[p.8]{Norris2008}), where $\phi_*$ is the push-forward and $\phi^*$ the pull-back (see also \cite{holzapfel2000}).
\end{rem}
Having discussed which objective derivatives appear to be a naturally sound choice\footnote
{
In Section \ref{appendchainrule} we will show a universal chain rule for corotational derivatives. This possibility is absent for general objective derivatives. The additional chain rule structure gives further support for using only corotational derivatives.
}
-- namely the corotational derivatives $\frac{\DD^{\circ}}{\DD t}$ -- the next important question regards a sound choice for the tangent stiffness tensor $\H^*(\sigma)$ appearing in \eqref{eqthedoublestr}. Considering a rigid rotation $F(t) \to Q(t) \, F(t)$, we have the transformations
	\begin{align}
	D(t) = \sym L(t) \to Q(t) \, D \, Q^T(t), \qquad \sigma \to Q(t) \, \sigma \, Q^T(t), \qquad \frac{\DD^{\circ}}{\DD t}[Q(t) \, \sigma \, Q^T(t)] = Q(t) \, \frac{\DD^{\circ}}{\DD t}[\sigma] \, Q^T(t),
	\end{align}
such that the hypo-elastic model \eqref{eqthedoublestr} is frame-indifferent if and only if
	\begin{align}
	\H^{*}(Q \, \sigma \, Q^T) . (Q \, D \, Q^T) = Q \, (\H^{*}(\sigma) . D) \, Q^T
	\end{align}
is satisfied for any orthogonal tensor $Q \in \OO(3)$. Thus $\H^{*}(\sigma) . D$ must be an isotropic tensor function of $\sigma$ and $D$. It is a well-known fact that any hyperelastic or Cauchy-elastic model in which $\sigma\colon \Sym^{++}(3) \to \Sym(3), \break B \mapsto \sigma(B)$ is bijective can be written in the format\footnote
{
The constitutive law of every isotropic hyperelastic or isotropic Cauchy-elastic model can always be written in rate-format as
	\begin{align}
	\frac{\DD^{\ZJ}}{\DD t}[\sigma(B)] = \H^{\ZJ}(B).D
	\end{align}
with an isotropic fourth order tensor $\H^{\ZJ}(B)$. In order to arrive rewrite this expression as $\H^{\ZJ}(\sigma)$, one needs the invertibility of $B \mapsto \sigma(B)$. The generality claimed in \eqref{eq2.12} will be shown in Section \ref{app13}.
}
	\begin{align}
	\label{eq2.12}
	\frac{\DD^{\circ}}{\DD t}[\sigma] = \H^*(\sigma).D
	\end{align}
(cf. Truesdell \cite{truesdellremarks} and Noll \cite{Noll55}), albeit with an expression for $\H^{*}(\sigma)$ that is not easily manageable. The general format of such a mapping is given as (see already Noll \cite{Noll55})
	\begin{equation}
	\label{eqrepresent001}
	\begin{alignedat}{2}
	\H^{*}(\sigma) . D = \; &b_1 \, D + b_2(\sigma \, D + D \, \sigma) + b_3(\sigma^2 \, D + D \, \sigma^2) + [b_4 \, \tr D + b_5 \, \tr(D \, \sigma) + b_6 \, \tr(D \, \sigma^2)] \, \id \\
	&+[b_7 \, \tr D + b_8 \, \tr(D \, \sigma) + b_9 \, \tr(D \, \sigma^2)] \, \sigma + [b_{10} \, \tr D + b_{11} \, \tr(D \, \sigma) + b_{12} \, \tr (D \, \sigma^2)] \, \sigma^2,
	\end{alignedat}
	\end{equation}
where the coefficients $b_i$ are scalar valued functions of the principal invariants of $\sigma$ (in \cite{bernstein2007} only a subclass is considered). Note, however, that in this generality, the latter formulation is rather useless. Therefore, in the absence of other evidence zero-grade hypo-elasticity\footnote
{
Simo and Pister \cite{simopister} write concerning the hypo-elastic formulation: \enquote{[L]ack of experimental evidence supporting a particular form for [the stiffness tensor $\H^{*}(\sigma)$] often leads to the choice of the constant isotropic elasticity tensor $[\H^{*}(\sigma) = \C^{\iso}]$ of the linearized theory.}
} \cite{kim2016}
	\begin{align}
	\H^{*}(\sigma).D = b_1 D + b_4(\tr D) \, \id = 2\mu \, D + \lambda \, \tr(D) \, \id = \C^{\iso} . D, \qquad \mu > 0, \quad 2 \, \mu + 3 \, \lambda > 0,
	\end{align}
has often been considered and is still sometimes considered as the definition of hypo-elasticity \cite{LIN2003443}. For more information about the motivation behind considering zero-grade hypo-elasticity we refer to the Appendix \ref{appendixzerograde}. \\
\\
Even though zero-grade hypo-elasticity might appear to be an easy and sound choice\footnote
{
In fact, discarded by Truesdell \cite[p.405]{Truesdell65}.
}
-- as it provides a major and minor symmetric and positive definite constant tangent stiffness tensor $\H^* = \C^{\iso}$ and is motivated by linear elasticity -- it is easy to demonstrate by simple shear considerations, that zero-grade hypo-elasticity together with the Zaremba-Jaumann rate is a poor choice from a physical point of view (see the Appendix \ref{appendixsimpleshear}). \\
\\
Hence the question remains open which choice for the tangent stiffness tensor $\H^*$ is a good choice. The next more or less \enquote{obvious} attempt is to use the tangent stiffness tensor $\H^{\circ}$ that is \textbf{induced} by the chosen objective, corotational time derivative $\frac{\DD^{\circ}}{\DD t}$ for a \textbf{given invertible} and \textbf{isotropic} Cauchy-elastic stress-stretch law $B \mapsto \sigma(B)$, i.e.~to set
	\begin{align}
	\label{eqdefZJ}
	\boxed{\H^{\ZJ}(\sigma).D \colonequals \frac{\DD^{\ZJ}}{\DD t}[\sigma],\qquad \H^{\GN}(\sigma).D \colonequals \frac{\DD^{\GN}}{\DD t}[\sigma] \qquad \text{and} \qquad \H^{\circ}(\sigma).D \colonequals \frac{\DD^{\circ}}{\DD t}[\sigma].}
	\end{align}
Note that any definition of a tangent stiffness tensor $\H^* = \H^{\circ}$ is usually expressed in $B = F \, F^T$, i.e.\ $\H^{\circ} = \H^{\circ}(B)$, hence we generally require the invertibility of the stress-stretch law $B \mapsto \sigma(B)$ with inverse mapping $\sigma \mapsto \mathcal{F}^{-1}(\sigma)$ so that we can write $\H^{\circ}(\sigma)\colonequals \H^{\circ}(\mathcal{F}^{-1}(\sigma))$. Furthermore, $\H^{\circ}(\sigma)$ is now dependent only on stresses but not on strains or their rates and for each different objective rate it will be a different stiffness tensor. In general, this induced tangent stiffness tensor has minor symmetry $\H^{\circ}\colon \Sym(3) \to \Sym(3)$ but \textbf{no} major symmetry: $\H^{\circ} \notin \Sym_4(6)$, i.e.~$\langle \H^\circ.D_1, D_2 \rangle \neq \langle \H^{\circ}.D_2, D_1 \rangle$ and the corresponding induced tangent compliance tensor $\mathbb{S}^{\circ}(\sigma) = [\H^{\circ}(\sigma)]^{-1}$ shares the same properties as $\H^{\circ}$. \\
\\
Let us consider two examples of simple shear for this new choice of $\H^*(\sigma) = \H^{\circ}(\sigma)$.
\begin{example}
The isotropic, physically linear Cauchy-elastic law $\sigma\colon \Sym^{++}(3) \to \Sym(3)$
	\begin{align}
	\label{eqlinearlaw1}
	\sigma(B) = \mu \, (B - \id), \qquad \nu = 0
	\end{align}
 is bijective viewed as a mapping $\sigma\colon \Sym^{++}(3) \to \textnormal{range}(\sigma)$ and strictly monotone in $B$. Here, $\nu = \frac{\lambda}{2 \, (\lambda + \mu)}$ is the Poisson's ratio. The Cauchy stress $\sigma(B)$ has the Zaremba-Jaumann derivative (see Section \ref{sechypoelastic} for a general formula)
	\begin{align}
	\frac{\DD^{\ZJ}}{\DD t}[\sigma] = \mu \, (D \, B + B \, D) = 2 \, \mu \, D + D \, \sigma + \sigma \, D \equalscolon \H^{\ZJ}(\sigma).D \, ,
	\end{align}
which, for this \enquote{simple} law, can be calculated directly from the definition of the Zaremba-Jaumann rate, i.e.
	\begin{align}
	\frac{\DD}{\DD t}[\sigma] &= \mu \, \frac{\DD}{\DD t}[B- \id] = \mu \, (L \, B + B \, L^T) = \mu \, (L \, (B-\id) + (B-\id) \, L^T + L + L^T) \notag \\
	&= (L \, \sigma + \sigma \, L^T) + 2\mu \, D = (D+W) \, \sigma + \sigma \, (D+W)^T + 2 \, \mu \, D = D \, \sigma + \sigma \, D + 2 \, \mu \, D + W\, \sigma - \sigma \, W \notag \\
	\label{eqsigma1}
	\iff \frac{\DD^{\ZJ}}{\DD t}[\sigma] &= \frac{\DD}{\DD t}[\sigma] + \sigma \, W - W \, \sigma = 2 \, \mu \, D + D \, \sigma + \sigma \, D = \H^{\ZJ}(\sigma) . D \, .
	\end{align}
Regarding the representation formula \eqref{eqrepresent001}, here we have $b_1 = 2 \, \mu, \; b_2 = 1$ and for the remaining $b_i, i=3,...,12$ it is $b_i = 0$. Thus, we may speak of a hypo-elastic formulation of grade 1. By chance, the induced tangent stiffness tensor has minor \textbf{and} major symmetry but the constitutive law \eqref{eqlinearlaw1} is not hyperelastic. For the planar case with
	\begin{align}
	F(t) =
	\begin{pmatrix}
	1 & \gamma \, t \\
	0 & 1
	\end{pmatrix},
	\qquad
	B = F \, F^T =
	\begin{pmatrix}
	1 & \gamma \, t \\
	0 & 1
	\end{pmatrix}
	\begin{pmatrix}
	1 & 0 \\
	\gamma \, t & 1
	\end{pmatrix}
	=
	\begin{pmatrix}
	1 + \gamma^2 \, t^2 & \gamma \, t \\
	\gamma \, t & 1
	\end{pmatrix}
	\end{align}
as well as
\begin{align}
	L(t) = \dot{F} \, F^{-1} =
	\begin{pmatrix}
	0 & \gamma \\
	0 & 0
	\end{pmatrix},
	\qquad
	D(t) = \sym L(t) = \frac12
	\begin{pmatrix}
	0 & \gamma \\
	\gamma & 0
	\end{pmatrix}
	\end{align}
we see that the rate-formulation
	\begin{align}
	\frac{\DD^{\circ}}{\DD t}[\sigma] = \frac{\DD^{\ZJ}}{\DD t}[\sigma] = \H^{\ZJ}(\sigma).D = \H^*(\sigma).D
	\end{align}
is trivially fulfilled under stress-free initial conditions, and we conclude from
	\begin{align}
	\sigma = \mu \, (B - \id) = \mu \, 
	\begin{pmatrix}
	\gamma^2 \, t^2 & \gamma \, t \\
	\gamma \, t & 0
	\end{pmatrix},
	\end{align}
that the shear stress $\sigma_{12}(B(t)) = \mu \, \gamma \, t$ is linear increasing in the amount of shear and therefore physically reasonable when using the Zaremba-Jaumann rate $\H^{\ZJ}(\sigma)$. Of course, $\H^{\ZJ}(\sigma)$ is now not constant in the Cauchy stress $\sigma$.
\end{example}
\begin{example}
Consider the physically non-linear Cauchy-elastic law $\sigma\colon \Sym^{++}(3) \to \Sym(3)$,
	\begin{align}
	\label{eqintro003}
	\sigma(B)= \frac{\mu}{2}(B-B^{-1})+\frac{\lambda}{2} (\log \det B) \, \id,
	\end{align}
with the Zaremba-Jaumann derivative \cite{zaremba1903forme,jaumann1911geschlossenes} (see Section \ref{sechypoelastic} for a general formula)
	\begin{align}
	\label{eqintro004}
	\frac{\DD^{\ZJ}}{\DD t}[\sigma] = \frac{\mu}{2} \, \{B \, D + D \, B + B^{-1} \, D + D \, B^{-1}\} + \lambda \, \tr(D) \, \id \equalscolon\H^{\ZJ}(B).D \, ,
	\end{align}
which again happens to be minor and major symmetric but is also not hyperelastic. Since the constitutive law \eqref{eqintro003} is now invertible \cite{ghiba2024}, we can write $\H^{\ZJ}(\sigma) \colonequals \H^{\ZJ}(\mathcal{F}^{-1}(\sigma))$. Similarly to the previous example we can determine the shear stress $\sigma_{12}(B(t))$. With $F(t), L(t)$ and $D(t)$ as in the previous example,
	\begin{align}
	B = F \, F^T =
	\begin{pmatrix}
	1 & \gamma \, t \\
	0 & 1
	\end{pmatrix}
	\begin{pmatrix}
	1 & 0 \\
	\gamma \, t & 1
	\end{pmatrix}
	=
	\begin{pmatrix}
	1 + \gamma^2 \, t^2 & \gamma \, t \\
	\gamma \, t & 1
	\end{pmatrix},
	\qquad
	B^{-1} =
	\begin{pmatrix}
	1 & -\gamma \, t \\
	-\gamma \, t & 1 + \gamma^2 \, t^2
	\end{pmatrix}
	\end{align}
such that the rate equation \eqref{eqintro004} is satisfied. Since $\tr \log B = \log \det B = 0$, we can calculate the shear stress $\sigma_{12}$ directly from \eqref{eqintro003}:
	\begin{align}
	\sigma_{12}(B(t)) = \frac{\mu}{2} \, (\gamma \, t + \gamma \, t) = \mu \, \gamma \, t \, ,
	\end{align}
which is, once again, linear and increasing in the amount of shear $\gamma \, t$ and thus physically reasonable. $\hfill{\square}$
\end{example}
Considering the previous two examples, the tangent stiffness tensor induced by a suitable Cauchy-elastic constitutive law appears to be a physically reasonable choice. Since its derivation can be traced back to the derivation of the corresponding corotational, objective time derivative $\frac{\DD^{\circ}}{\DD t}[\sigma]$, we will prove next a concise formula for the calculation of this induced tangent stiffness tensor. Proceeding in this fashion we obviate any integrability issue \cite{ericksen1958hypo, bernstein1958} that arises if $\H^*(\sigma)$ is taken in arbitrary form according to \eqref{eqthedoublestr}.
%
%Already Noll \cite{Noll55} realized\footnote
%{
%Noll \cite{Noll55} considered only the Zaremba-Jaumann rate.
%},
%to the surprise of Truesdell \cite{truesdellremarks}, that every isotropic, hyperelastic or isotropic, Cauchy-elastic constitutive relation can be written in rate-format as a hypo-elastic formulation, provided that the mapping $\sigma\colon \Sym^{++}(3) \to \Sym(3), \; B \mapsto \sigma(B)$ is invertible\footnote
%{
%It is quite remarkable that Truesdell and Noll \cite{Truesdell65} largely ignored this invertibility assumption for the Cauchy stress $\sigma$ in their subsequent development for nonlinear elasticity.
%}.
%In that case, there exists a naturally \textbf{induced tangent stiffness tensor} 
%	\begin{align}
%	\H^{\circ}(\sigma).D = \frac{\DD^{\circ}}{\DD t}[\sigma],
%	\end{align}
%appears to be a physically reasonable choice. Since its derivation can be traced back to the derivation of the corresponding corotational, objective time derivative $\frac{\DD^{\circ}}{\DD t}[\sigma]$, we will prove next a concise formula for the calculation of this induced tangent stiffness tensor. Proceeding in this fashion we obviate any integrability issue \cite{ericksen1958hypo, bernstein1958} that arises if $\H^*(\sigma)$ is taken in arbitrary form according to \eqref{eqthedoublestr}.
%
\subsection{Formulas for the corotational Zaremba-Jaumann and Green-Naghdi rates} \label{sechypoelastic}
The following calculations are motivated by the elegant early exposition in \cite{fei1994}. Recall that the constitutive equation of nonlinear isotropic Cauchy-elasticity can be expressed as
	\begin{align}
	\label{eqratetype1}
	\sigma\colon \Sym^{++}(3) \subset \Sym(3) \to \Sym(3), \quad B \mapsto \sigma(B), \quad (B = \mathcal{F}^{-1}(\sigma) \in \Sym^{++}(3) \; \textnormal{if $\sigma(B)$ is invertible})
	\end{align}
with $B = F\, F^T$ and an isotropic tensor function $\sigma(B)$, i.e.~$\sigma(B)$ satisfies\footnote
{
As shown already by Richter in 1948 \cite{richter1948isotrope, richter1949hauptaufsatze, Richter50, Richter52} and later by Rivlin and Ericksen \cite{rivlin1955} every such isotropic tensor function $B \mapsto \sigma(B)$ can e.g.~be expressed as
	\begin{align}
	\sigma(B) = \beta_0 \, \id + \beta_1 \, B + \beta_{-1} \, B^{-1}\,,
	\end{align}
where the $\beta_i$ are scalar functions of the principal invariants of $B$ \cite{mihai2017}. However, this representation of the elastic law will not be useful for our purposes.
}
	\begin{align}
	\label{eqratetype02}
	\sigma\colon \Sym^{++}(3) \subset \Sym(3) \to \Sym(3), \qquad Q \, \sigma(B) \, Q^T = \sigma(Q \, B \, Q^T) \qquad \forall \, Q \in \OO(3).
	\end{align}
Recalling the identities $L = \dot{F} \, F^{-1}$ and $L = D + W$ with $D = \sym L, \; W = \sk \, L$, we calculate the material time derivative (cf.~\cite[p.37]{Noll55}) 
	\begin{equation}
	\label{eqratetype2}
	\begin{alignedat}{2}
	\frac{\DD}{\DD t}[\sigma] &= \DD_B\sigma(B).\dot{B} = \DD_B\sigma(B).[\dot{F} \, F^T + F \, \dot{F}^T] = \DD_B\sigma(B).[L \, F \, F^T + F \, (L \, F)^T] \\
	&= \DD_B\sigma(B).[L \, B + B \, L^T] =\DD_B\sigma(B).[D \, B + B \, D] + \DD_B\sigma(B).[W \, B - B \, W] \, ,
	\end{alignedat}
	\end{equation}
where $B \mapsto \DD_B\sigma(B)$ is an isotropic fourth order tensor function satisfying
	\begin{align}
	\label{eqratetype01}
	Q \, [\DD_B\sigma(B).H] \, Q^T = \DD_B\sigma(Q \, B \, Q^T).[Q \, H \, Q^T].
	\end{align}
Taking time derivatives on both sides of \eqref{eqratetype02} leads to
	\begin{align}
	\frac{\DD}{\DD t}[Q \, \sigma(B) \, Q^T] = \dot{Q} \, \sigma(B) \, Q^T + Q \, (\DD_B\sigma(B).\dot{B}) \, Q^T + Q \, \sigma(B) \, \dot{Q}^T
	\end{align}
for the left hand side, as well as
	\begin{equation}
	\begin{alignedat}{2}
	\frac{\DD}{\DD t}[\sigma(Q \, B \, Q^T)] &= \DD_B\sigma(Q \, B \, Q^T).\frac{\DD}{\DD t}[Q \, B \, Q^T] = \DD_B\sigma(Q \, B \, Q^T).[\dot{Q} \, B \, Q^T + Q \, \dot{B} \, Q^T + Q \, B \, \dot{Q}^T] \\
	&= \DD_B\sigma(Q \, B \, Q^T).[\dot{Q} \, B \, Q^T + Q \, B \, \dot{Q}^T] + \DD_B\sigma(Q \, B \, Q^T).[Q \, \dot{B} \, Q^T].
	\end{alignedat}
	\end{equation}
Using \eqref{eqratetype01} with $H = \dot{B}$ shows that 
	\begin{align}
	Q \, [\DD_B\sigma(B).\dot{B}] \, Q^T = \DD_B\sigma(Q \, B \, Q^T).[Q \, \dot{B} \, Q^T],
	\end{align}
leading to the identity
	\begin{align}
	\label{eqratetype3}
	\dot{Q} \, \sigma(B) \, Q^T + Q \, \sigma(B) \, \dot{Q}^T = \DD_B\sigma(Q \, B \, Q^T).[\dot{Q} \, B \, Q^T + Q \, B \, \dot{Q}^T].
	\end{align}
A suitable choice of the arbitrary matrix $Q(t) \in \OO(3)$ in combination with \eqref{eqratetype2} now lead to a formula for the Zaremba-Jaumann and the Green-Naghdi time derivative, respectively.
\subsubsection{Formula for the corotational, objective Zaremba-Jaumann time derivative}
For the Zaremba-Jaumann derivative we choose $Q(0) = \id, \; \dot{Q}(0) = W$ in \eqref{eqratetype3} to obtain
	\begin{equation}
	\label{eqratetype4}
	\boxed{W \, \sigma(B) - \sigma(B) \, W = \DD_B\sigma(B).[W \, B - B \, W].}
	\end{equation}
Note that this formula was discovered by Noll in the year 1955 \cite[p. 37]{Noll55} in the same fashion. Inserting \eqref{eqratetype4} in \eqref{eqratetype2} yields
	\begin{align}
	\label{eqratetype5}
	\boxed{\frac{\DD^{\ZJ}}{\DD t}[\sigma] = \frac{\DD}{\DD t}[\sigma] - W \, \sigma + \sigma \, W = \DD_B\sigma(B).[D \, B + B \, D] = \H^{\ZJ}(B).D}
	\end{align}
for the Zaremba-Jaumann derivative. If $B \mapsto \sigma(B)$ is invertible, this can be rewritten as \cite[p.~37, eq.~(15.12)]{Noll55}
	\begin{align}
	\frac{\DD^{\ZJ}}{\DD t}[\sigma] = \DD_B\sigma(\mathcal{F}^{-1}(\sigma)).[D \, \mathcal{F}^{-1}(\sigma) + \mathcal{F}^{-1}(\sigma) \, D] \equalscolon \H^{\ZJ}(\sigma).D
	\end{align}
and determines the induced tangent stiffness tensor $\H^{\ZJ}(\sigma)$ since $B \mapsto \sigma(B)$, and thus its inverse $\sigma\mapsto\mathcal{F}^{-1}(\sigma)$, is assumed given.
\subsubsection{Formula for the corotational, objective Green-Naghdi time derivative}
For the derivation of the corresponding formula for the Green-Naghdi derivative\footnote
{
Related to this general rule, in Dienes \cite[p.~13]{dienes1987discussion} it is proven that $\frac{\DD^{\GN}}{\DD t}[B] = 2 \, V \, D \, V$.
}
(the polar rate) $\frac{\DD^{\GN}}{\DD t}$, we first recall the following identity (cf. \cite{dienes1979})
	\begin{lem}
	If $F = V \, R$ is the left polar decomposition of $F$ and $\Omega = \dot{R} \, R^T$, then
	\begin{align}
	\label{eqdienes}
	(D+W-\Omega)\, B + B \, (D-W+\Omega) = L \, B + B \, L^T - \Omega \, B + B \, \Omega = 2 \, V \, D \, V .
	\end{align}
	\end{lem}
	\begin{proof}
	We begin by pointing out the additional identities
		\begin{align}
		V = V^T \quad \implies \quad \dot{V} = \dot{V}^T \qquad \textnormal{and} \qquad R \, R^T = \id \quad \implies \quad \Omega^T = - \Omega,
		\end{align}
	where the last implication holds since $\Omega + \Omega^T = \dot{R} \, R^T + R \, \dot{R}^T = \frac{\DD}{\DD t}[R \, R^T] = \frac{\DD}{\DD t}[\id] = 0$.
	Now we have
		\begin{align}
		\label{eqdienes00}
		L = \dot{F} \, F^{-1} = \frac{\DD}{\DD t}[V \, R] \, (V \, R)^{-1} = (\dot{V} \, R + V \, \dot{R})\, R^{-1} \, V^{-1} = \dot{V} \, V^{-1} + V \, \dot{R} \, R^{-1} \, V^{-1}
		\end{align}
	which we premultiply and postmultiply by $V$ to obtain (with $R^T = R^{-1}$ and $B = V^2$)
		\begin{align}
		\label{eqdienes01}
		V \, L \, V = V \, \dot{V} + V^2 \, \dot{R} \, R^{-1} = V \, \dot{V} + B \, \Omega.
		\end{align}
	In the next step we transpose \eqref{eqdienes01} to obtain (with $B = B^T$)
		\begin{align}
		\label{eqdienes02}
		V \, L^T \, V = \dot{V}^T \, V^T + \Omega^T \, B^T = \dot{V} \, V - \Omega \, B.
		\end{align}
	Adding \eqref{eqdienes01} and \eqref{eqdienes02} then yields (with $L + L^T = 2 \, D$)
		\begin{align}
		2 \, V \, D \, V = V \, \dot{V} + \dot{V} \, V + B \, \Omega - \Omega \, B.
		\end{align}
	Hence, it remains to show that
		$
		L \, B + B \, L^T = (L\, B) + (L \, B)^T = V \, \dot{V} + \dot{V} \, V.
		$
	But with \eqref{eqdienes00} and $B = V^2$ it follows that
		$
		L \, B = (\dot{V} \, V^{-1} + V \, \Omega \, V^{-1}) \, V^2 = \dot{V} \, V + V \, \Omega \, V,
		$
	which implies
		\begin{equation}
		\begin{alignedat}{2}
		\hspace*{-10pt} (L \, B) + (L \, B)^T &= (\dot{V} \, V + V \, \Omega \, V) + (\dot{V} \, V)^T + (V \, \Omega \, V)^T = \dot{V} \, V + V^T \, \dot{V}^T + V \, \Omega \, V - V \, \Omega \, V = \dot{V} \, V + V \, \dot{V},
		\end{alignedat}
		\end{equation}
	and the claim follows. 
	\end{proof}
\noindent Knowing that \eqref{eqdienes} holds, we can substitute
	\begin{align}
	L \, B + B \, L^T = 2 \, V \, D \, V + \Omega \, B - B \, \Omega
	\end{align}
into \eqref{eqratetype2} to obtain
	\begin{equation}
	\label{eqratetype6}
	\begin{alignedat}{2}
	\frac{\DD}{\DD t}[\sigma] &= \DD_B\sigma(B).[\dot B] = \DD_B\sigma(B).[L \, B + B \, L^T] = \DD_B\sigma(B).[2 \, V \, D \, V - B \, \Omega + \Omega \, B] \\
	&= 2 \, \DD_B\sigma(B).[V \, D \, V] + \DD_B\sigma(B).[\Omega \, B - B \, \Omega].
	\end{alignedat}
	\end{equation}
Next we choose $Q(0) = \id$ and $\dot{Q}(0) = \Omega$ in \eqref{eqratetype3} and use $R \, R^T = \id$, which implies $\Omega^T = - \Omega$. This leads to
	\begin{equation}
	\boxed{\Omega \, \sigma - \sigma \, \Omega = \DD_B\sigma(B).[\Omega \, B - B \, \Omega],}
	\end{equation}
which can be inserted in \eqref{eqratetype6} to obtain the formula
	\begin{align}
	\label{eqrateGN01}
	\boxed{\frac{\DD^{\GN}}{\DD t}[\sigma] = \frac{\DD}{\DD t}[\sigma] + \sigma \, \Omega - \Omega \, \sigma = 2 \, \DD_B\sigma(B).[V \, D \, V] = \H^{\GN}(V).D \, .}
	\end{align}
Since $V^2 = B = \mathcal{F}^{-1}(\sigma)$ in the case of invertible $B \mapsto \sigma(B)$, the latter can be rewritten as
	\begin{align}
	\frac{\DD^{\GN}}{\DD t}[\sigma] = 2 \, \DD_B\sigma(\mathcal{F}^{-1}(\sigma)).[\sqrt{\mathcal{F}^{-1}(\sigma)} \, D \sqrt{\mathcal{F}^{-1}(\sigma)}] \equalscolon \H^{\textnormal{GN}}(\sigma).D \, ,
	\end{align}
determining the induced tangent stiffness tensor $\H^{\GN}(\sigma)$ for the Green-Naghdi rate.
%
%
%
%
%
\begin{comment}
%
%
\red{
\begin{example}[$\sigma(B) = \frac12 \, (B - \id)$]
For this example we simply obtain
	\begin{align}
	\frac{\DD^{\GN}}{\DD t}[\frac12 \, (B - \id)] = V \, D \, V
	\end{align}
(calculus checked independently by S.~N.~Korobeynikov in August 2024). This result should be contrasted with an involved formula from \cite[p.~187, eq.~(100)]{asghari2008} [Naghdabadi, On the objective corotational rates of Eulerian strain measures, Journal of Elasticity, 2008]
	\begin{align*}
	\frac{\DD^{\GN}}{\DD t}\left[\frac12 \, (B - \id)\right] = V^2 D_d + ((-II/2) \, \id + I \, V - V^2)(D - D_d) + (D-D_d)((-II/2) \, \id + I \, V - V^2),
	\end{align*}
where $D_d$ is given by
	\begin{align}
	D_d = [I \, V^2 - 2 \, II \, V + 3 \, III \, \id]^{-1} \, [V^2 \, D \, V + V \, D \, V^2 - I \, V \, D \, V + III \, \id]
	\end{align}
and $I, II, III$ are the principal invariants of $V = \sqrt{B}$.
\end{example}
}
%
%
\end{comment}
%
%
%
%
%
\begin{rem}
Note again that whenever we speak about the induced fourth order tangent stiffness tensor $\H^{\circ}(\sigma)$, we need to require invertibility.\footnote
{
Saccomandi and Rajagopal put forward \cite{saccomandi2016novel}: ``There are several shortcomings with respect to the manner in which constitutive relations are usually specified currently, both from a philosophical standpoint and more pragmatic considerations. From the philosophical standpoint, expressing the stress in terms of kinematical variables turns causality on its head, as forces and stresses are the causes, and the kinematics is the effect. It makes much more sense to describe kinematics in terms of the stresses and/or their derivatives.'' Not much experience is available in directly prescribing the inverted law $B=\mathcal{F}^{-1}(\sigma)$ apart from those restrictions that isotropy dictates, see \cite{rajagopal2003implicit, rajagopal2007elasticity, sfyris2015treatment}.
 }
of $B \mapsto \sigma(B)$ (since otherwise we would only have a tensor $\H^{\circ}(B)$). Thus, whenever the tensor $\H^{\circ}(\sigma)$ appears, we silently assume that the underlying constitutive law $B \mapsto \sigma(B)$ is invertible\footnote
{
Otherwise, the stiffness tensor $\H^{\circ}(B)$ does depend explicitly on $B = F F^T$ instead of only the Cauchy stress $\sigma$. The Cauchy stress tensor $\sigma$ itself is independent of the adopted reference configuration, while $B$ involves the computation of the deformation gradient $F = \DD \varphi$, defined with respect to the adopted reference configuration which is arbitrary. The independence of the reference configuration gives additional motivation for insisting on the invertibility of $B \mapsto \sigma(B)$, at least as regards the induced tangent stiffness tensor $\H^{\circ}(\sigma)$. Of course, assuming $\sigma = \sigma(B)$ will always involve the reference configuration.
}
\end{rem}
\begin{rem}[\textbf{One formula to rule them all...}] \label{theoneremark}
Consider an \textbf{arbitrary corotational derivative} (not necessarily objective) with spin tensor $\Omega^{\circ} \in \mathfrak{so}(3)$ for an isotropic function $\sigma = \sigma(B)$, i.e.
	\begin{align}
	\label{eq:general_corotational_rate}
	\frac{\DD^{\circ}}{\DD t}[\sigma] = \frac{\DD}{\DD t}[\sigma] - \Omega^{\circ} \, \sigma + \sigma \, \Omega^{\circ}.
	\end{align}
Then the equation \eqref{eqratetype3} can be applied to any choice of $Q(t)$ with $Q(0) = \id, \; \dot Q(0) = \Omega^{\circ}$ to obtain the general relation (remaining true for isotropic tensor functions $\overline \sigma \colon \Sym^{++}(3) \to \R^{3 \times 3}$)
	\begin{equation}
	\label{eqgeneralomsig}
	\boxed{\Omega^{\circ} \, \sigma(B) - \sigma(B) \, \Omega^{\circ} = \DD_B\sigma(B).[\Omega^{\circ} \, B - B \, \Omega^{\circ}],}
	\end{equation}
alternatively expressed via the Lie-bracket $[A,B] = A \, B - B \, A$ as
	\begin{equation}
	\boxed{[\Omega^{\circ}, \sigma(B)] = \DD_B\sigma(B).[\Omega^{\circ}, B].}
	\end{equation}
This formula represents the crucial identity used in Section \ref{chainruleallg} to prove a chain rule for arbitrary corotational derivatives.
\end{rem}
\begin{rem}
Due to its importance for our exposition, a slightly different approach to also obtain \eqref{eqgeneralomsig} is given by the following idea thanks to a discussion with Sergey N.~Korobeynikov: \\
\\
We begin again with the relations \eqref{eqratetype01} and \eqref{eqratetype3}, i.e.
	\begin{align}
	\label{eqkoro1}
	Q \, [\DD_B\sigma(B).H] \, Q^T = \DD_B\sigma(Q \, B \, Q^T).[Q \, H \, Q^T]
	\end{align}
and
	\begin{align}
	\label{eqkoro2}
	\dot{Q} \, \sigma(B) \, Q^T + Q \, \sigma(B) \, \dot{Q}^T = \DD_B\sigma(Q \, B \, Q^T).[\dot{Q} \, B \, Q^T + Q \, B \, \dot{Q}^T].
	\end{align}
Next, we define $H$ as
	\begin{align}
	Q \, H \, Q^T = \dot Q \, B \, Q^T + Q \, B \, \dot Q^T \qquad \implies \qquad H = Q^T \, \dot Q \, B + B \, \dot Q^T \, Q
	\end{align}
and insert this in \eqref{eqkoro1} which afterwards can be inserted into \eqref{eqkoro2} to obtain
	\begin{align}
	\dot{Q} \, \sigma(B) \, Q^T + Q \, \sigma(B) \, \dot{Q}^T = Q \, \DD_B \sigma(B).[Q^T \, \dot Q \, B + B \, \dot Q^T \, Q] \, Q^T \, .
	\end{align}
Multiplying this equation by $Q^T$ from the left and $Q$ from the right results in
	\begin{align}
	Q^T \, \dot{Q} \, \sigma(B) + \sigma(B) \, \dot Q^T \, Q = \DD_B \sigma(B).[Q^T \, \dot Q \, B + B \, \dot Q^T \, Q]
	\end{align}
so that, by defining the skew-symmetric spin tensor $\Omega^{\circ}\colonequals Q^T \, \dot Q$, we finally obtain \eqref{eqgeneralomsig}.
\end{rem}
\pagebreak
\subsection{Positive definiteness of the induced tangent stiffness tensor $\H^\circ(\sigma)$: the corotational stability postulate}
	\begin{wrapfigure}[29]{r}{0.55\textwidth}
		\begin{center}
		\begin{minipage}[h!]{\linewidth}
			\arraycolsep1pt
			\centering
			\includegraphics[scale=0.33]{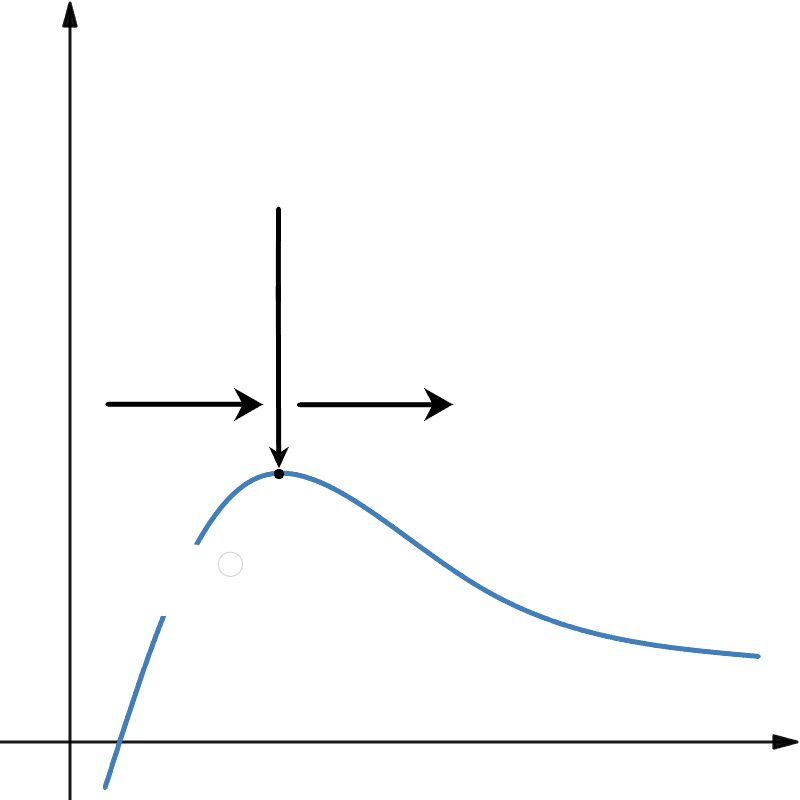}
			\put(-10,0){\footnotesize{$\lambda$}}
			\put(-165,2){\footnotesize $1$}
			\put(-175,102){\footnotesize{stiffening}}
			\put(-127,102){\footnotesize{softening}}
			\put(-170,57){\footnotesize{corotational}}
			\put(-170,47){\footnotesize{stability}}
			\put(-160,32){\footnotesize{$\langle \frac{\DD^{\ZJ}}{\DD t} [\sigma] , D \rangle > 0$}}
			\put(-60,75){\footnotesize{corotational}}
			\put(-60,65){\footnotesize{instability}}
			\put(-60,50){\footnotesize{$\langle \frac{\DD^{\ZJ}}{\DD t} [\sigma] , D \rangle < 0$}}
			\put(-175,165){\footnotesize{``two intrinsic}}
			\put(-175,155){\footnotesize{failure conditions'' :}}
			\put(-125,125){\footnotesize{$\frac{\DD^{\ZJ}}{\DD t} [\sigma] = 0, \; \textnormal{but} \; D \neq 0$}}
			\put(-100,155){\footnotesize{$\left\{ \begin{array}{ll} \det \H^{\ZJ}(\sigma) = 0 \; \iff \; \det \DD_B\sigma(B) = 0 \\ \textnormal{or} \; \sym \H^{\ZJ}(\sigma) \notin \Sym^{++}_4(6) \end{array} \right.$}}
			\put(-175,190){\footnotesize{$\sigma(\lambda)$}}
		\end{minipage}
		\end{center}
		\caption{Illustration of our interpretation of Drucker's stability inequality for isotropic nonlinear elasticity as \break \textbf{``corotational stability postulate''}. In the $\sigma$-softening regime we expect convergence problems and non-well-posedness.}
		\label{figDrucker}
	\end{wrapfigure}
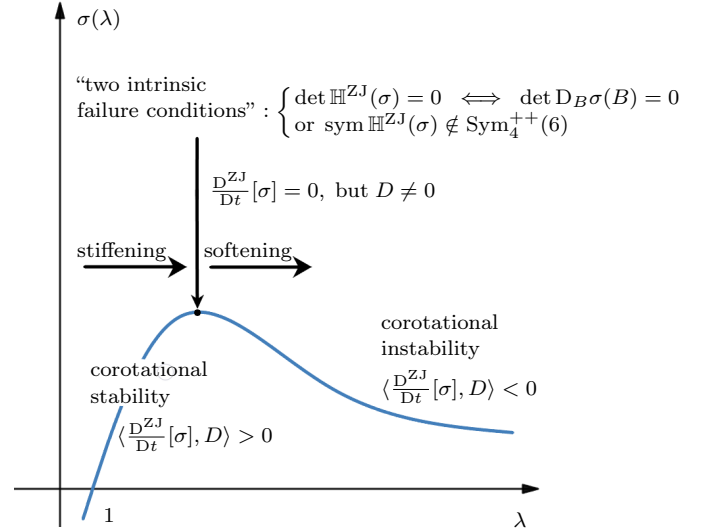

To motivate the meaningfulness of the analysis of the induced tangent stiffness tensor $\H^{\circ}(\sigma)$, we take a look at isotropic linear elasticity, i.e.~we consider the linear elastic law
	\begin{align}
	\varepsilon = \sym \, \DD u \mapsto \sigma_{\lin}(\varepsilon) = \C^{\iso}.\varepsilon.
	\end{align}
Here, stability is expressed as the \textbf{positive definiteness} of $\C^{\iso} \in \Sym^{++}_4(6)$:
	\begin{equation}
	\label{eq2.60}
	\begin{alignedat}{2}
	&\langle \C^{\iso}.\varepsilon, \varepsilon \rangle \ge c^+ \, \norm{\varepsilon}^2 \\
	\iff \qquad &\mu>0, \quad 2 \, \mu + 3 \, \lambda > 0.
	\end{alignedat}
	\end{equation}
This is almost equivalent to the \textbf{invertibility} of \break $\varepsilon \mapsto \sigma_{\lin}(\varepsilon)$, for which we note that
	\begin{align}
	\det \C^{\rm iso}\neq 0\iff \mu\neq 0, \ \ 2\mu+3\lambda\neq 0.
	\end{align}
The linear constitutive law can also be written in rate-format as
	\begin{align}
	\dot \sigma = \C^{\iso}.\dot \varepsilon\,.
	\end{align}
Then the positive definiteness of $\C^{\iso}$ is equivalently described through 
	\begin{align}
	\label{eqlinear01}
	\langle \dot \sigma, \dot \varepsilon \rangle > 0 \qquad \forall \, \dot \varepsilon \in \Sym(3), \; \dot \varepsilon \neq 0.
	\end{align}
We refer to \eqref{eq2.60} and \eqref{eqlinear01} as {\bf linear stability postulate}. In isotropic linear elasticity, the linear stability postulate is equivalent to the convexity of the elastic energy $\frac12 \, \langle \C^{\iso}.\varepsilon, \varepsilon \rangle$ in the displacement gradient $\DD u$ which is the same as positivity of the second order work (see also the Appendix \ref{appsecondorderwork}). \\
\\
Next, consider an isotropic tensor function $\varepsilon \mapsto \sigma(\varepsilon)$, not necessarily linear in $\varepsilon$, and impose
	\begin{equation}
	\label{eqdrucker}
	\begin{alignedat}{2}
	\langle \dot \sigma, \dot \varepsilon \rangle > 0
	\quad\iff&\quad
	\langle \frac{\dif}{\dif t}[\sigma(\varepsilon(t))], \frac{\dif}{\dif t}[\varepsilon(t)] \rangle > 0 \qquad \forall \, \frac{\dif}{\dif t}[\varepsilon(t)] \neq 0
	 \\
	\iff& \quad \langle \DD \sigma(\varepsilon(t)). \frac{\dif}{\dif t}[\varepsilon(t)], \frac{\dif}{\dif t}[\varepsilon(t)] \rangle > 0 
	\iff \quad \sym \DD \sigma(\varepsilon(t)) \in \Sym^{++}_4(6) \\
	\iff& \quad \langle \dif \sigma(\varepsilon), \dif \varepsilon \rangle > 0 \qquad \textnormal{classical ``\textbf{Drucker-stability}'' for linearized kinematics} \\
	\implies& \quad \langle \sigma(\varepsilon_1) - \sigma(\varepsilon_2), \varepsilon_1 - \varepsilon_2 \rangle > 0 \qquad \implies \qquad  \varepsilon \mapsto \sigma(\varepsilon) \quad \textnormal{is invertible.}
	\end{alignedat}
	\end{equation}
Now, the rate-form stability requirement \eqref{eqlinear01} can naturally be generalized to the hypo-elastic setting by replacing the material time derivative of $\sigma$ by the corotational Zaremba-Jaumann rate and the time derivative of $\varepsilon$ by the Eulerian strain rate, i.e.,
	\begin{align}
	\frac{\dif}{\dif t}[\sigma] \quad \textnormal{by} \quad \frac{\DD^{\ZJ}}{\DD t}[\sigma] \qquad \textnormal{and} \qquad \frac{\dif}{\dif t}[\varepsilon] \quad \textnormal{by} \quad D = \sym \, \DD v,
	\end{align}
so that \eqref{eqlinear01} turns into
	\begin{align}
	\label{eqlinear02}
	\langle \frac{\DD^{\ZJ}}{\DD t}[\sigma], D \rangle > 0 \qquad \forall \, D \in \Sym(3) \! \setminus \! \{0\}.
	\end{align}
In analogy to \eqref{eqlinear01} we call \eqref{eqlinear02} the \textbf{``corotational stability postulate'' CSP}. It is evident that \eqref{eqlinear02} is frame-indifferent, while the ``second order work'' $\langle \frac{\dif}{\dif t}[\sigma],D \rangle$ is not, cf.~Appendix \ref{appsecondorderwork}. Also note that CSP is in general not related to a convexity type condition (e.g.~rank-one convexity) in $F = \DD \varphi$. \\
\\
Since by \eqref{eqdefZJ} we have the relation $\frac{\DD^{\ZJ}}{\DD t}[\sigma] = \H^{\ZJ}(\sigma).D$, the requirement \eqref{eqlinear02} turns further into the positive definiteness condition for the induced tangent stiffness tensor $\H^{\ZJ}(\sigma)$
	\begin{equation}
	\label{eqlinstability}
	\begin{alignedat}{2}
	\langle \H^{\ZJ}(\sigma).D,D \rangle > 0 \quad \forall \, D \in \Sym(3) \! \setminus \! \{0\} \qquad &\iff \qquad \langle \sym \, \H^{\ZJ}(\sigma).D, D \rangle > 0 \quad \forall \, D \in \Sym(3) \! \setminus \! \{0\} \\
	&\iff \qquad \sym \H^{\ZJ}(\sigma) \in \Sym^{++}_4(6).
	\end{alignedat}
	\end{equation}
Note that since $\H^{\ZJ}(\sigma)$ is not necessarily major symmetric (i.e.\ self adjoint), we obtain only the positive definiteness of the symmetric part of $\H^{\ZJ}(\sigma)$, i.e.~$\sym \, \H^{\ZJ}(\sigma) \in \Sym^{++}_4(6)$. 
\subsection{Invertibility of the induced tangent stiffness tensor $\H^{\circ}(\sigma)$}
In Section \ref{sechypoelastic} we have derived the formulas for the induced tangent stiffness tensors $\H^{\ZJ}(\sigma)$ and $\H^{\GN}(\sigma)$ under the assumption that $B \mapsto \sigma(B)$ is invertible, i.e.~there is a function $\mathcal{F}^{-1} = \mathcal{F}^{-1}(\sigma)$ with $B = \mathcal{F}^{-1}(\sigma)$ for every $B \in \Sym^{++}(3)$. Under this assumption it holds
	\begin{align}
	\label{eqraterepresnt1}
	\H^{\rm ZJ}(\sigma).D = \DD_B\sigma(B).[B \, D + D \, B] \qquad \textnormal{and} \qquad \H^{\rm GN}(\sigma).D =\DD_B\sigma(B).[2 \, V \, D \, V].
	\end{align}
Let us first check, as a minimal requirement usually made (cf.~e.g.~Ericksen \cite{ericksen1958hypo} or Romano \cite{romano2011}), under which condition on the given constitutive law the induced tangent stiffness tensors $\H^{\ZJ}(\sigma)$ and $\H^{\GN}(\sigma)$ will be invertible, considered as $6 \times 6$ matrices so that one could also speak about the induced compliance tensor $\mathbb{S} = \H^{-1}$. In this respect we can formulate the
\begin{prop}[Invertibility of $\H^{\ZJ}(\sigma)$ and $\H^{\GN}(\sigma)$] \label{propinvert1}
The induced tangent stiffness tensors $\H^{\ZJ}(\sigma)$ and $\H^{\GN}(\sigma)$ are invertible if and only if $\DD_B \sigma(B)$ is invertible, i.e.
	\begin{align}
	\left.\begin{array}{r}
	\det \H^{\rm ZJ}(\sigma) \neq 0\\
	\det \H^{\rm GN}(\sigma) \neq 0
	\end{array}\right\} \quad \iff \quad \det \DD _B\sigma(B) \neq 0 \qquad \iff \qquad \det \DD_{\log B} \widehat \sigma(\log B) \neq 0.
	\end{align}
\end{prop}
\begin{proof}
It is well-known (cf.~Sidoroff \cite{sidoroff1974restrictions} or Scheidler \cite{scheidler1994}) that the linear mappings
	\begin{align}
	D \mapsto B \, D + D \, B \qquad \textnormal{and} \qquad D \mapsto V \, D \, V
	\end{align}
are invertible for $V^2 = B \in \Sym^{++}(3)$. Thus, from the representation of $\H^{\ZJ}(\sigma)$ and $\H^{\GN}(\sigma)$ in \eqref{eqraterepresnt1} it is clear that $\H^{\ZJ}(\sigma)$ and $\H^{\GN}(\sigma)$ are invertible if and only if $\DD_B \sigma(B)$ is invertible. Finally, since $B \mapsto \log B$ is invertible and $\det \DD_B \log B > 0$ (cf.~Appendix \ref{lemhilbertmon1}), the same holds true if and only if $\DD_{\log B} \widehat \sigma(\log B)$ is invertible, by the standard chain rule.
\end{proof}
The different conditions $\det \H^{\ZJ}(\sigma)>0$ and $\sym \H^{\ZJ}(\sigma) \in \Sym^{++}_4(6)$ can be used as novel stability criteria, see Figure \ref{figDrucker} and compare with \cite{jog2013conditions}. In general, we know that
	\begin{align}
	\sym \H^{\ZJ}(\sigma) \in \Sym^{++}_4(6) \quad &\implies \quad \langle (\sym \H^{\ZJ}(\sigma)). D, D \rangle = \langle \H^{\ZJ}(\sigma).D, D \rangle > 0 \quad \forall D \in \Sym(3) \! \setminus \! \{0\} \\
	&\implies \quad \H^{\ZJ}(\sigma).D \in \Sym(3) \! \setminus \! \{0\} \quad\forall D \in \Sym(3) \! \setminus \! \{0\}\quad \implies \quad \det \H^{\ZJ}(\sigma) \neq 0. \notag
	\end{align}
In this respect it is useful to recall the Bendixson-inequality \cite{bendixson1902} $\lambda_{\min}(\sym \H^{\ZJ}) \le \textnormal{Re}(\lambda_i(\H^{\ZJ}))$ showing that the positive definiteness of $\sym \H^{\ZJ}$ implies $ \det \H^{\ZJ}(\sigma) > 0$ in general. A detailed exposition of these stability requirements for nonlinear elasticity will be pursued in a future contribution.

\begin{rem}
It is useful to remember that the global invertibility of $B \mapsto \sigma(B), \; \Sym^{++}(3) \to \Sym(3)$ does not already imply that $\DD_B \sigma(B)$ is invertible everywhere, as the simple one-dimensional example
$f\colon \R \to \R, \; f(t) = t^3$ demonstrates:
%$f\colon \R \to \R, \; f(t) = \log(t)^3$ shows:
$f$ is strictly monotone throughout and globally invertible, but $\DD f(0) = 0$.
\end{rem}
\section{Corotational stability results and conjectures} \label{sec5}
We recall the main result from \cite{tobedone} as it is the starting point for the upcoming considerations.
	\begin{thm} \label{thm4.3}
	Let $\sigma\colon \Sym^{++}(3) \to \Sym(3), \; B \mapsto \sigma(B)$ be the constitutive expression of the Cauchy stress tensor for isotropic nonlinear elasticity. Assume that $\sigma \in C^1(\Sym^{++}(3), \Sym(3))$. Then the following equivalence holds:
	\begin{equation}
		\begin{alignedat}{2}
		&0 \;<\; \langle \frac{\DD^{\ZJ}}{\DD t}[\sigma], D \rangle &&= \langle \DD_B \sigma(B) . [B \, D + D \, B] , D \rangle \\&&&= \langle \DD_{\log B} \widehat \sigma(\log B). \DD_B \log B. [B \, D + D \, B] \\
		&&&= \mathrlap{\langle \H^{\ZJ}(\sigma).D, D \rangle} \hphantom{\langle \DD_{\log B} \widehat \sigma(\log B). \DD_B \log B. [B \, D + D \, B]} \qquad \forall \, D \in \Sym(3) \setminus \{0\} \\
		\iff \qquad &\mathrlap{B \mapsto \sigma(B) = \widehat{\sigma}(\log B) \quad \textnormal{is strongly monotone in $\log B$ (TSTS-M$^{++}$)}} \\
		\iff \qquad &\mathrlap{\sym \DD_{\log B} \widehat \sigma(\log B) \in \Sym^{++}_4(6)\,,}
		\end{alignedat}
	\end{equation}
	which implies
		\begin{equation}
		\label{eqappendix11}
		\langle \widehat{\sigma}(\log B_1) - \widehat{\sigma}(\log B_2), \log B_1 - \log B_2 \rangle > 0 \qquad \quad \forall \, B_1 \neq B_2, \quad B_1, B_2 \in \Sym^{++}(3)\,.
		\end{equation}
	\end{thm}
\begin{rem}
The novel point of departure is to view the CSP characterization now from the angle of properties of the induced tangent stiffness tensor $\H^{\ZJ}(\sigma)\colonequals \DD_B \sigma(B). [B \, D + D \, B] = \DD_{\log B} \widehat \sigma(\log B). \DD_B \log B.[B \, D + D \, B]$, giving directly algebraic conditions on the Cauchy stress $\sigma$.
\end{rem}
	\begin{rem}
	\label{rem4.4}
	The strong Hilbert-monotonicity reads
		\begin{align}
		\label{eqappendix22}
		\sym \DD_{\log B} \widehat{\sigma}(\log B) \in \Sym^{++}_4(6) \qquad \iff \qquad \sym \DD_{\log V} \widehat{\sigma}(\log V) \in \Sym^{++}_4(6) \, ,
		\end{align}
	which can be equivalently expressed in principal Cauchy stresses (see the Appendix, Remark \ref{remA5})
		\begin{align}
		\forall \, (\lambda_1, \lambda_2, \lambda_3) \in \R^3_+: \qquad \Lambda_{ij} = \sym \, \frac{\partial \widehat{\sigma}_i (\log \lambda_1, \log \lambda_2, \log \lambda_3)}{\partial \log \lambda_j} \, \in \Sym^{++}(3).
		\end{align}
		For the proof we refer to \cite{tobedone}.
	\end{rem}
It is important to point out that strict Hilbert monotonicity of $\sigma(V) = \widehat \sigma(\log V)$ in $V$ and in $\log V$ (in $B$ and in $\log B$), respectively, are two completely unrelated characteristics of a constitutive Cauchy law $V \mapsto \sigma(V)$. In fact, neither implies the other, as shown explicitly in Appendix~\ref{appmono001}.
\subsection{Some instructive examples}
Recall from formula \eqref{eqratetype5} that
	\begin{align}
	\frac{\DD^{\ZJ}}{\DD t}[\sigma] = \frac{\DD}{\DD t}[\sigma] - W \, \sigma + \sigma \, W = \DD_B\sigma(B).[D \, B + B \, D] = \H^{\ZJ}(\sigma).D
	\end{align}
and from \eqref{eqrateGN01} that
	\begin{align}
	\frac{\DD^{\GN}}{\DD t}[\sigma] = \frac{\DD}{\DD t}[\sigma] + \sigma \, \Omega - \Omega \, \sigma = 2 \, \DD_B\sigma(B).[V \, D \, V] = \H^{\GN}(\sigma).D
	\end{align}
for the Zaremba-Jaumann and the Green-Naghdi derivative, respectively.
\begin{example} \label{example5.1}
We reconsider the stress response $\sigma(B) = \mu \, (B - \id)$, which is a primary matrix function of $B$, for $\mu > 0$, so that $\DD_B\sigma(B) . H = \mu \, H$. Then, denoting by $\lambda_{\min}$ the smallest eigenvalue of its argument\footnote
{
Note that the relation $\lambda_{\min}(B) = \lambda_{\min}^2(V)$ holds.
},
we obtain\footnote{Note carefully (cf.~Remark~\ref{remark:one_dimensional_examples}) that the mapping $B\mapsto\sigma(B)=\mu\,(B-\id)$ is not invertible in this case and that $\frac{\DD^{\circ}}{\DD t}[\sigma]$ according to \eqref{eq:non_invertible_ZJ} and \eqref{eq:non_invertible_ZJ} only corresponds to the Cauchy-elastic constitutive law if $\sigma$ is in the range of the stress response mapping.}
	\begin{align}
	\label{eq:non_invertible_ZJ}
	\langle \H^{\ZJ}(\sigma).D, D \rangle = \mu \, \langle B \, D + D \, B , D \rangle = 2 \, \mu \, \langle D \, B , D \rangle \geq 2\,\mu\, \lambda_{\rm min}(B)\, \|D\|^2> 0
	\end{align}
as well as
	\begin{align}
	\label{eq:non_invertible_GN}
	\langle \H^{\GN}(\sigma).D,D \rangle = 2 \, \mu \, \langle [V \, D \, V], D \rangle = 2 \, \mu \, \langle V \, D \, V , D \rangle \ge 2 \, \mu \, \lambda_{\min}^2(V) \, \norm{D}^2 > 0\,,
	\end{align}
showing (directly) that for both rates the corotational stability postulate (CSP) is satisfied, which is equivalent to strong monotonicity of $\widehat \sigma(\log B)$ as shown in Theorem \ref{thm4.3}. Independently, we see that $\sigma(B) = \widehat \sigma(\log B)$ is also strictly Hilbert-monotone in $\log B$, by observing that
	\begin{equation}
	\label{eqmonotonesig}
	\begin{alignedat}{2}
	\langle \widehat{\sigma}(\log B_1) - \widehat{\sigma}(\log B_2)&, \log B_1 - \log B_2 \rangle \\
	&= \langle \sigma(B_1) - \sigma(B_2) , \log B_1 - \log B_2 \rangle = \mu \, \langle B_1 - \id - (B_2 - \id), \log B_1 - \log B_2 \rangle \\
	&= \mu \, \langle B_1 - B_2, \log B_1 - \log B_2 \rangle = \mu \, \langle \log B_1 - \log B_2, B_1 - B_2 \rangle > 0,
	\end{alignedat}
	\end{equation}
since $B \mapsto \log B$ is a strongly monotone primary matrix function, as shown in Lemma \ref{lemhilbertmon1} in the Appendix.
\end{example}
\begin{example}
We next show positive definiteness of the induced tangent stiffness tensors $\H^{\ZJ}(\sigma)$ and $\H^{\GN}(\sigma)$ for the constitutive Cauchy-elastic law (a non-primary isotropic matrix function)
	\begin{align}
	\label{eqthesecondlaw1}
	\sigma(B) = \frac{\mu}{2} (B - B^{-1}) + \frac{\lambda}{2} \, \tr(\log B) \, \id.
	\end{align}
Therefore, we first calculate $\DD_B \sigma(B).H$, by making use of the equality $\DD_B [B^{-1}] = - B^{-1} \, H \, B$, which leads to
	\begin{equation}
	\DD_B\sigma(B).H = \frac{\mu}{2} (H + B^{-1} \, H \, B^{-1}) + \frac{\lambda}{2} \langle B^{-1} , H \rangle \, \id.
	\end{equation}
Inserting the increment $H = [B \, D + D \, B]$ (not to be confounded with $\H^{\ZJ}(\sigma)$), we obtain
	\begin{align}
	\H^{\ZJ}(\sigma).D = \DD_B\sigma(B).[B \, D + D \, B] = \frac{\mu}{2} \, (B \, D + D \, B + B^{-1} \, D + D \, B^{-1})
+ \lambda \, \tr(D) \, \id
	\end{align}
and similarly, for $H = [V \, D \, V]$ with $\langle B^{-1},VDV\rangle=\langle VB^{-1}V,D\rangle=\tr(D)$,
	\begin{align}
	\H^{\GN}(\sigma).D = 2 \, \DD_B\sigma(B).[V \, D \, V] = \mu \, (V \, D \, V + V^{-1} \, D \, V^{-1}) + \lambda \, \tr(D) \, \id.
	\end{align}
From this we find
	\begin{equation}
	\label{eqthelatter01}
	\begin{alignedat}{2}
	\langle \H^{\ZJ}(\sigma) . D , D \rangle &= \frac{\mu}{2} \, (2 \, \langle B \, D , D \rangle + 2 \, \langle B^{-1} \, D , D \rangle) + \lambda \, \langle \tr(D) \, \id , D \rangle \\
	&= \mu \, (\langle B \, D , D \rangle + \langle B^{-1} \, D , D \rangle) + \lambda \, \tr^2(D) \\
	&\ge \mu \, \left(\lambda_{\min}(B) + \frac{1}{\lambda_{\max}(B)}\right) \, \norm{D}^2 + \lambda \, \tr^2(D) \ge c^+(\mu, \lambda) \, \norm{D}^2
	\end{alignedat}
	\end{equation}
as well as
	\begin{equation}
	\begin{alignedat}{2}
	\langle \H^{\GN}(\sigma).D,D \rangle &= \mu \, \langle V \, D \, V , D \rangle + \mu \, \langle V^{-1} \, D \, V^{-1} , D \rangle + \lambda \, \langle \id , D \rangle \, \tr(D) \\
	&\ge \mu \, (\lambda_{\min}^2(V) \, \norm{D}^2 + \lambda_{\min}^2(V^{-1}) \, \norm{D}^2) + \lambda \, \tr^2(D)\,,
	\end{alignedat}
	\end{equation}
which are both non-negative if we assume $\mu, \lambda \ge 0$, thus showing that CSP is fulfilled and $\widehat \sigma(\log B)$ is strongly monotone in $\log B$ according to Theorem \ref{thm4.3}. \\
\\
Dividing \eqref{eqthesecondlaw1} into three parts
	\begin{align}
	\label{eqsigma2}
	\sigma = \frac{\mu}{2} (B - B^{-1}) + \frac{\lambda}{2} \, \tr(\log B) \, \id = \frac12 \left\{\, \smash{\underbrace{\mu \, (B-\id)}_{\equalscolon \sigma_1} + \underbrace{\mu \, (\id-B^{-1})}_{\equalscolon \sigma_2} + \underbrace{\lambda \, \tr(\log B) \, \id}_{\equalscolon\sigma_3}} \,\right\}\,,
	\vphantom{\underbrace{\mu \, (\id-B^{-1})}_{\equalscolon \sigma_2} + \underbrace{\lambda \, \tr(\log B) \, \id}_{\equalscolon\sigma_3}}
	\end{align}
we can easily check (directly) that $\sigma$ in \eqref{eqsigma2} is monotone as a function of $\log B$, where we suppose that $\mu, \, \lambda > 0$. The term $B - \id$ is monotone in $\log B$ as shown in Example \ref{example5.1}. Furthermore, it is easy to see that $\tr(\log B) \, \id$ is monotone in $\log B$. It remains to check the term $\sigma = \frac{1}{\mu} \sigma_2 = \id - B^{-1}$. We find
	\begin{equation}
	\begin{alignedat}{2}
	\langle \sigma(B_1) - \sigma(B_2)&, \log B_1 - \log B_2 \rangle \\
	&= \langle \id - B_1^{-1} - (\id - B_2^{-1}), \log B_1 - \log B_2 \rangle = \langle -B_1^{-1} - (-B_2^{-1}), \log B_1 - \log B_2 \rangle \\
	&= \langle -B_1^{-1} + B_2^{-1} , - \log B_1^{-1} - (-\log B_2^{-1}) \rangle = \langle B_2^{-1} - B_1^{-1} , \log B_2^{-1} - \log B_1^{-1} \rangle \\
	&= \langle \log X - \log Y, X - Y \rangle \,>\, 0\,, \quad X=B_2^{-1}, \quad Y = B_1^{-1},
	\end{alignedat}
	\end{equation}
since $\log$ is strongly monotone in its argument, showing independently that $\sigma$ given by \eqref{eqsigma2} is strictly Hilbert-monotone in $\log B$.
\end{example}

\begin{rem}
\label{remark:one_dimensional_examples}
Let us consider, for the previous two examples, the corresponding one-dimensional Cauchy stresses. On the one hand we have $\sigma_a$, shown in Figure~\ref{xfig90}, corresponding to the three-dimensional Cauchy-elastic (but not hyperelastic) constitutive law $\sigma_{\alpha}$, with
	\begin{align}
	\sigma_a(\lambda) \colonequals \frac{1}{2} (\lambda^2-1), \quad \qquad \sigma_{\alpha}(B) = \mu \, (B-\id) = 2\mu \, \sym \DD u + \textnormal{h.o.t.} = 2\mu \, \varepsilon + \textnormal{h.o.t.},
	\end{align}
and on the other hand the one-dimensional Cauchy stress $\sigma_b$ (cf.~Figure \ref{xfig91}) with the corresponding three-dimensional counterpart $\sigma_{\beta}$,
	\begin{align}
	\sigma_b(\lambda) \colonequals \frac{1}{5} (\lambda^2 - \lambda^{-2} + \log \lambda)\,, \quad \qquad \sigma_{\beta}(B) = \frac{\mu}{2} (B - B^{-1}) + \frac{\lambda}{2} \, \tr(\log B) \, \id\,,
	\end{align}
respectively. In the one-dimensional setting, both give rise to monotone functions $\lambda \mapsto \sigma(\lambda)$.\footnote
{
Both constitutive laws are isotropic and objective, therefore they can be used, in principle, for small strain and large rotations. However, only $\sigma_b$ is suitable for large strains, since $B \mapsto \sigma_b(B)$ is invertible while $\sigma_a$ does not respond properly for extreme stretches $B \to 0$.
}
	\begin{figure}[h!]
		\begin{center}
		\begin{minipage}[h!]{0.4\linewidth}
		\includegraphics[scale=0.25]{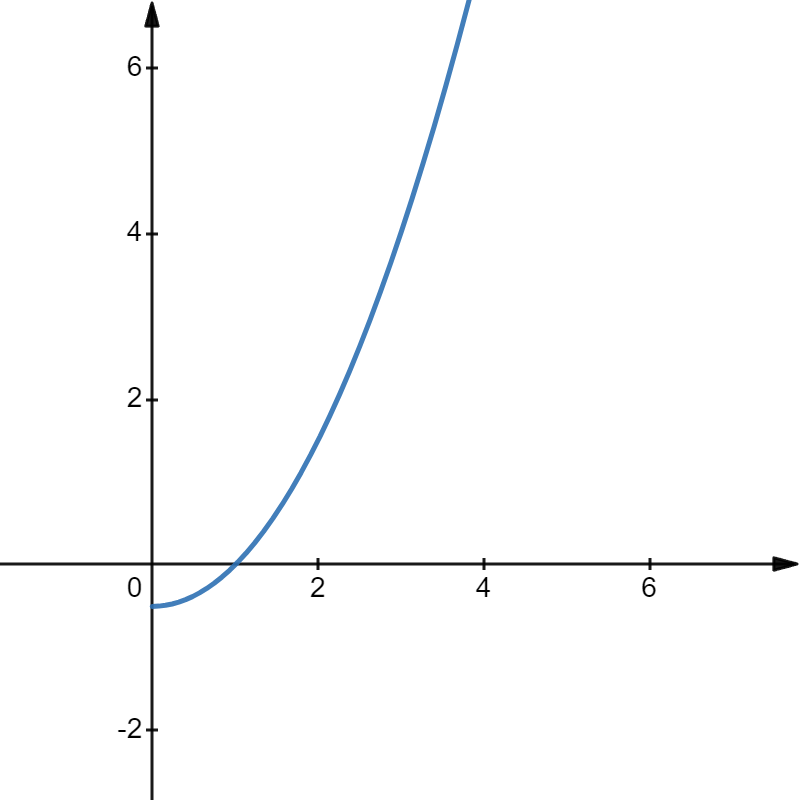}
		\put(-5,45){\footnotesize{$\lambda$}}
		\put(-155,190){\footnotesize{$\sigma(\lambda)$}}
		\caption{Picture of the monotone Cauchy stress $\sigma_a(\lambda) = \frac12(\lambda^2-1)$.}
		\label{xfig90}
		\end{minipage}
		\qquad \qquad
		\begin{minipage}[h!]{0.4\linewidth}
		\includegraphics[scale=0.25]{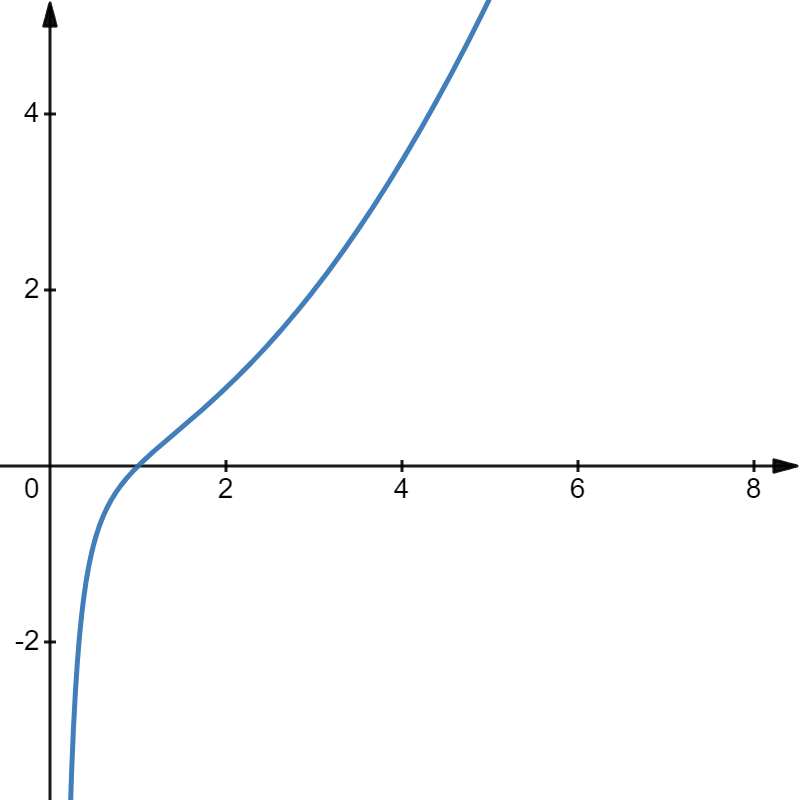}
		\put(-5,70){\footnotesize{$\lambda$}}
		\put(-180,190){\footnotesize{$\sigma(\lambda)$}}
		\caption{Monotone and bijective Cauchy stress $\sigma_b(\lambda) = \frac15(\lambda^2 - \lambda^{-2} + \log \lambda)$. Idealized nonlinear elastic response: ``\textbf{stress increases with strain}''.}
		\label{xfig91}
		\end{minipage}
		\end{center}
	\end{figure}
Figure~\ref{xfig90} demonstrates that for the choice $\sigma = \mu \, (B-\id)$, the stress response is not surjective; in particular, there does not exist any $B\in\Sym(3)^{++}$ such that $\mu\,(B-\id)=\sigma$ if $\sigma_i\leq-\mu$ for any eigenvalue $\sigma_i$ of $\sigma$. Therefore, the relation
	\begin{align}
	\label{eq:rate_formulation_non_invertible_incomplete}
	\H^{\ZJ}(\sigma).D = \mu \, (D \, B + B \, D)
	\end{align}
is only well-defined for $\sigma$ in the range of the stress response mapping, since otherwise, the left Cauchy-Green tensor $B$ on the right-hand side is not well defined.
However, we may still talk about \textbf{conditional positive definiteness} of $\H^{\ZJ}(\sigma)$, i.e.\ $\H^{\ZJ}(\sigma)$ is positive definite whenever there is a function $\mathcal{F}^{-1}\colon \Sym(3) \to \Sym^{++}(3)$ such that
	\begin{align}
	\langle \H^{\ZJ}(\sigma).D,D \rangle = \langle \H^{\ZJ}(\mathcal{F}^{-1}(\sigma)).D,D \rangle = \langle \H^{\ZJ}(B).D,D \rangle \qquad \textnormal{for} \qquad B = \mathcal{F}^{-1}(\sigma) \in \Sym^{++}(3)\,.
	\end{align}
Note also that the relation
	\begin{align}
	\H^{\ZJ}(\sigma).D = 2 \, \mu \, D + D \, \sigma + \sigma \, D\,,
	\end{align}
which is equivalent to \eqref{eq:rate_formulation_non_invertible_incomplete} for $\sigma = \mu \, (B-\id)$ if such a $B\in\Sym^{++}(3)$ exists, properly defines a stiffness tensor for any $\sigma\in\Sym(3)$ and thus (globally) constitutes a hypoelastic law.

%is positive definite for all $\sigma \in \Sym(3)$ since on the one hand $\mathcal{F}^{-1}(\sigma)$ does not exist for $\sigma < -\frac12$ and on the other hand $\mathcal{F}^{-1}(\sigma)$ is not necessarily in $\Sym^{++}(3)$, since, as in this one-dimensional example, $\mathcal{F}^{-1}(-\frac12) = 0$. However, to speak about positive definiteness of $\H^{\ZJ}(\sigma)$ we need the invertibility of $B \mapsto \sigma(B)$ since for a Cauchy-elastic law we always consider $\langle \H^{\ZJ}(B).D,D \rangle > 0$ and use the bijectivity of $B \mapsto \sigma(B)$ or $\sigma \mapsto \mathcal{F}^{-1}(\sigma)$ respectively, to equivalently write $\langle \H^{\ZJ}(\sigma).D,D \rangle > 0$.
%Nonetheless, even in the case of non-bijective $\sigma(B)$, we may still talk about \textbf{conditional positive definiteness} of $\H^{\ZJ}(\sigma)$, i.e. $\H^{\ZJ}(\sigma)$ is positive definite whenever there is a function $\mathcal{F}^{-1}\colon \Sym(3) \to \Sym^{++}(3)$, such that
%	\begin{align}
%	\langle \H^{\ZJ}(\sigma).D,D \rangle = \langle \H^{\ZJ}(\mathcal{F}^{-1}(\sigma)).D,D \rangle = \langle \H^{\ZJ}(B).D,D \rangle \qquad \textnormal{for} \qquad B = \mathcal{F}^{-1}(\sigma) \in \Sym^{++}(3).
%	\end{align}
An example for an \emph{invertible} Cauchy stress response $B \mapsto \sigma(B)$ that generates an induced fourth order tangent stiffness tensor $\H^{\ZJ}(\sigma)$ that is \emph{not positive definite} throughout is given by
	\begin{equation}
	\begin{alignedat}{2}
	\WW_{\NH}(F) &= \frac{\mu}{2} \, \left(\frac{\norm{F}^2}{(\det F)^{\frac23}} - 3\right) + \kappa \, \mathrm{e}^{(\log \det F)^2}, \\
	\sigma_{\NH}(B) &= \mu \, (\det B)^{-\frac56} \, \dev_3 B + \kappa \, (\det B)^{-\frac12} \, (\log \det B) \, \mathrm{e}^{\frac14 \, (\log \det B)^2} \, \id\,,
	\end{alignedat}
	\end{equation}
as discussed in Appendix~\ref{appendixneohooke}.
\end{rem}
\subsection{Conjectures - a far reaching generality for the CSP}
In the previous examples, we observe that whenever $\log B \mapsto \widehat \sigma(\log B)$ is monotone, the equivalence \break CSP $\iff$ TSTS-M$^{++}$ holds not only for the corotational Zaremba-Jaumann rate, but also for the corotational Green-Naghdi rate. The authors have not yet found any isotropic Cauchy-elastic law for which this observation is not true, giving reason to propose the following conjecture (compare Theorem \ref{thm4.3}).
	\begin{conjecture}[Corotational stability for the Green-Naghdi rate $\frac{\DD^{\GN}}{\DD t}$]~\\
	Let $\sigma\colon \Sym^{++}(3) \to \Sym(3), \; B \mapsto \sigma(B)$ be the constitutive expression of the Cauchy stress tensor for isotropic nonlinear elasticity. Assume that $\sigma \in C^1(\Sym^{++}(3), \Sym(3))$. Then the following equivalence holds:
	\begin{equation}
		\begin{alignedat}{2}
		&0 \;<\; \langle \frac{\DD^{\GN}}{\DD t}[\sigma], D \rangle &&= \langle \DD_B \sigma(B) . [2 \, V D \, V] , D \rangle \\&&&= \langle \DD_{\log B} \widehat \sigma(\log B). \DD_B \log B. [2 \, V \, D \, V] \\
		&&&= \mathrlap{\langle \H^{\GN}(\sigma).D, D \rangle} \hphantom{\langle \DD_{\log B} \widehat \sigma(\log B). \DD_B \log B. [2 \, V \, D \, V]} \qquad \forall \, D \in \Sym(3) \setminus \{0\} \\
		\iff \qquad &\mathrlap{B \mapsto \sigma(B) = \widehat{\sigma}(\log B) \quad \textnormal{is strongly monotone in $\log B$ (TSTS-M$^{++}$)}} \\
		\iff \qquad &\mathrlap{\sym \DD_{\log B} \widehat \sigma(\log B) \in \Sym^{++}_4(6)\,,}
		\end{alignedat}
	\end{equation}
	implying
		\begin{equation}
		\langle \widehat{\sigma}(\log B_1) - \widehat{\sigma}(\log B_2), \log B_1 - \log B_2 \rangle > 0 \qquad \quad \forall \, B_1 \neq B_2, \quad B_1, B_2 \in \Sym^{++}(3).
		\end{equation}
	\end{conjecture}
%
%This conjecture will be shown to be true in an upcoming note \cite{Martin2024}.
In fact, we suspect that the equivalence between CSP and TSTS-M$^{++}$ may even hold for a larger class of objective corotational rates $\frac{\DD^{\circ}}{\DD t}$, the \textbf{positive corotational rates} (cf.~Section \ref{app13}, Definitions \ref{appmaterialspins} and \ref{appendixpcd}). For an in-depth analysis of positive corotational rates we refer to the development in \cite{neffkoro2024}.
	\begin{conjecture}[Corotational stability for positive corotational rates $\frac{\DD^{\circ}}{\DD t}$] \label{conjCSPgeneral}
	~\\Let $\sigma\colon \Sym^{++}(3) \to \Sym(3), \; B \mapsto \sigma(B)$ be the constitutive expression of the Cauchy stress tensor for isotropic nonlinear elasticity. Assume that $\sigma \in C^1(\Sym^{++}(3), \Sym(3))$. Let $\frac{\DD^{\circ}}{\DD t}$ belong to the subclass of positive, objective and corotational derivatives, i.e.~for the fourth order stiffness tensor $\mathbb{A}^{\circ}(B)$ in 
	\begin{equation}
	\frac{\DD^{\circ}}{\DD t}[B] = \mathbb{A}^{\circ}(B).D
	\end{equation}
we require $\mathbb{A}^{\circ}(B) \in \Sym^{++}_4(6)$. Then the following equivalence holds:
	\begin{equation}
		\begin{alignedat}{2}
		&0 \;<\; \langle \frac{\DD^{\circ}}{\DD t}[\sigma], D \rangle &&= \langle \DD_B \sigma(B) . \left[\frac{\DD^{\circ}}{\DD t}[B]\right] , D \rangle \\&&&= \langle \DD_{\log B} \widehat \sigma(\log B). \DD_B \log B. \left[\frac{\DD^{\circ}}{\DD t}[B]\right] \\
		&&&= \mathrlap{\langle \H^{\circ}(\sigma).D, D \rangle} \hphantom{\langle \DD_{\log B} \widehat \sigma(\log B). \DD_B \log B. \left[\frac{\DD^{\circ}}{\DD t}[B]\right]} \qquad \forall \, D \in \Sym(3) \setminus \{0\} \\[.7em]
		\iff \qquad &\mathrlap{B \mapsto \sigma(B) = \widehat{\sigma}(\log B) \quad \textnormal{is strongly monotone in $\log B$ (TSTS-M$^{++}$)}} \\
		\iff \qquad &\mathrlap{\sym \DD_{\log B} \widehat \sigma(\log B) \in \Sym^{++}_4(6)\,,}
		\end{alignedat}
	\end{equation}
	implying
		\begin{equation}
		\langle \widehat{\sigma}(\log B_1) - \widehat{\sigma}(\log B_2), \log B_1 - \log B_2 \rangle > 0 \qquad \quad \forall \, B_1 \neq B_2, \quad B_1, B_2 \in \Sym^{++}(3).
		\end{equation}
	\end{conjecture}
\section{The CSP from a matrix-analysis viewpoint}
It is clear that showing the far reaching generalizations in Conjecture \ref{conjCSPgeneral} cannot be undertaken by mimicking the technical calculations in principal Lagrangean axis given in \cite{tobedone}. To this aim, we need a new perspective, as introduced in the following section.
\subsection{Chain rule formula for arbitrary corotational rates} \label{appendchainrule}
In order to further investigate the relation between CSP and TSTS-M$^{++}$ for the logarithmic rate, we derive here an important and useful chain rule formula for an arbitrary corotational rate $\frac{\DD^{\circ}}{\DD t}$.
\subsubsection{Chain rule formula for primary matrix functions}
In \cite[p.~19, Theorem 2]{xiao98_1}, \cite[p.~7]{fiala2020objective}, \cite[p.~1066, Theorem 2.3]{korobeynikov2018} and \cite[p.~10, Lemma 1]{Norris2008}, a chain rule like formula for corotational rates acting on \emph{primary matrix functions} (cf.~Appendix~\ref{appendixnotation})
is supplied. More precisely, they prove the following.
\begin{prop} \label{propallgchain01}
Let $\sigma = \sigma(B) \in C^1(\Sym^{++}(3), \Sym(3))$ be a primary matrix function (e.g.~$B \mapsto \log B$) and consider an arbitrary corotational rate
	\begin{align}
	\frac{\DD^{\circ}}{\DD t}[\sigma(B)] = \frac{\DD}{\DD t}[\sigma(B)] + \sigma(B) \, \Omega^{\circ} - \Omega^{\circ} \, \sigma(B)
	\end{align}
defined by any spin tensor $\Omega^{\circ} \in \mathfrak{so}(3)$. Then $\sigma(B)$ obeys the chain rule
	\begin{align}
	\frac{\DD^{\circ}}{\DD t}[\sigma(B)] = \DD_B \sigma(B).\!\left[\frac{\DD^{\circ}}{\DD t}[B]\right].
	\end{align}
\end{prop}
\noindent Proposition \ref{propallgchain01} constitutes a partial chain rule result. Indeed, although being true for \textbf{any corotational rate}, it is \textbf{restricted to isotropic matrix functions that are primary matrix functions}. Nevertheless, we can e.g.~already calculate
	\begin{align}
	\frac{\DD^{\circ}}{\DD t}[B] = \frac{\DD^{\circ}}{\DD t}[V^2] = \DD_V[V^2].\frac{\DD^{\circ}}{\DD t}[V] = V \, \frac{\DD^{\circ}}{\DD t}[V] + \frac{\DD^{\circ}}{\DD t}[V] \, V
	\end{align}
so that $\frac{\DD^{\circ}}{\DD t}[V]$ is uniquely determined by $\frac{\DD^{\circ}}{\DD t}[B]$ since $V \in \Sym^{++}(3)$ (cf.~\cite[eq.~3.5b]{Norris2008}) and $X \mapsto X \, V + V \, X$ is invertible for $\frac{\DD^{\circ}}{\DD t}[V] = X \in \Sym(3)$. \\
\\
Next, we generalize Proposition \ref{propallgchain01} by first \textbf{directly} proving a chain rule formula for the corotational Zaremba-Jaumann rate and the corotational Green-Naghdi rate for \textbf{all isotropic tensor functions} $\sigma(B)$, which will then be generalized to a \textbf{chain rule} for an \textbf{arbitrary corotational rate} $\frac{\DD^{\circ}}{\DD t}$ and all \textbf{isotropic tensor functions} $\sigma(B)$. 
%\red{
%$\mathbb{A}^{\circ}(B).D = V \, [\overline{\mathbb{A}}^{\circ}(V).D] + [\overline{\mathbb{A}}^{\circ}(V).D] \, V$. Gilt dann $\mathbb{A}^{\circ}(B) \in \Sym^{++}_4(6) \; \iff \; \overline{\mathbb{A}}^{\circ}(V) \in \Sym^{++}_4(6)$? Norris-Operator dafür?
%}
%
\subsubsection{Chain rule formulas for the corotational Zaremba-Jaumann and Green-Naghdi rates}
We have seen in Section \ref{sechypoelastic} that the relations
	\begin{align}
	\label{eqchainrule01}
	\frac{\DD^{\ZJ}}{\DD t}[\sigma] = \DD_B \sigma(B).[D \, B + B \, D] \qquad \textnormal{and} \qquad \frac{\DD^{\GN}}{\DD t}[\sigma] = \DD_B \sigma(B).[2 \, V \, D \, V]
	\end{align}
hold. For the choice $\sigma(B) = B$ this yields\footnote
{
Recall that \eqref{eqZJB} can also directly be calculated. Surely, we have
	\begin{align*}
	\frac{\DD^{\ZJ}}{\DD t}[B] = \frac{\DD}{\DD t} B + B \, W - W \, B = L \, B + B \, L^T + B \, W - W \, B = (D + W) \, B + B \, (D - W) + B \, W - W \, B = D \, B + B \, D.
	\end{align*}
}
	\begin{align}
	\label{eqZJB}
	\frac{\DD^{\ZJ}}{\DD t}[B] = \DD_B B.[D \, B + B \, D] = \textnormal{id}.[D \, B + B \, D] = D \, B + B \, D
	\end{align}
and
	\begin{align}
	\frac{\DD^{\GN}}{\DD t}[B] = \DD_B B.[2 \, V \, D \, V] = \textnormal{id}.[2 \, V \, D \, V] = 2 \, V \, D \, V,
	\end{align}
implying that \eqref{eqchainrule01} can formally be rewritten as
	\begin{align}
	\frac{\DD^{\ZJ}}{\DD t}[\sigma] = \DD_B \sigma(B).\!\left[\frac{\DD^{\ZJ}}{\DD t}[B]\right] \qquad \textnormal{and} \qquad \frac{\DD^{\GN}}{\DD t}[\sigma] = \DD_B \sigma(B).\!\left[\frac{\DD^{\GN}}{\DD t}[B]\right].
	\end{align}
If we assume that $\sigma(B) = \widehat{\sigma}(\log B)$, application of the standard chain rule implies
	\begin{align}
	\label{eqchainZJ01}
	\frac{\DD^{\ZJ}}{\DD t}[\widehat{\sigma}(\log B)] = \DD_B \{\widehat{\sigma}(\log B)\}.\!\left[\frac{\DD^{\ZJ}}{\DD t}[B]\right] = \DD_{\log B} \widehat{\sigma} . \DD_B \log B .\!\left[\frac{\DD^{\ZJ}}{\DD t}[B]\right]
	\end{align}
and
	\begin{align}
	\label{eqchainGN01}
	\frac{\DD^{\GN}}{\DD t}[\widehat{\sigma}(\log B)] = \DD_B \{\widehat{\sigma}(\log B)\}.\!\left[\frac{\DD^{\GN}}{\DD t}[B]\right] = \DD_{\log B} \widehat{\sigma} . \DD_B \log B .\!\left[\frac{\DD^{\GN}}{\DD t}[B]\right].
	\end{align}
Additionally, applying formula \eqref{eqchainrule01} to the (primary matrix) function $B \mapsto \log B$ yields
	\begin{align}
	\frac{\DD^{\ZJ}}{\DD t}[\log B] = \DD_B \log B . [B \, D + D \, B]=\DD_B \log B.\! \left[\frac{\DD^{\ZJ}}{\DD t}[B]\right],
	\end{align}
so that by using the standard chain rule again we can also write
	\begin{equation}
	\begin{alignedat}{2}
	\H^{\ZJ}(\sigma).D &= \frac{\DD^{\ZJ}}{\DD t}[\sigma] = \DD_B \sigma(B) . [B \, D + D \, B] \\
	&= \DD_{\log B} \widehat{\sigma}(\log B) . \DD_B \log B . [B \, D + D \, B] = \DD_{\log B} \widehat{\sigma}(\log B).\!\left[\frac{\DD^{\ZJ}}{\DD t}[\log B]\right],
	\end{alignedat}
	\end{equation}
where obviously the same holds true for the Green-Naghdi rate. \\
\\
The transformation $\sigma(B) = \widehat \sigma(\log B)$ can be replaced by an arbitrary transformation of the type $\sigma(B) = \widetilde \sigma(g(B))$ with a differentiable  mapping $B \mapsto g(B)$ (neither $\widetilde \sigma$ nor $g$ need be monotone), yielding
	\begin{align}
	\frac{\DD^{\ZJ}}{\DD t}[\widetilde{\sigma}(g(B))] = \DD_B \{\widetilde{\sigma}(g(B))\}.\!\left[\frac{\DD^{\ZJ}}{\DD t}[B]\right] = \DD_{g(B)} \widetilde{\sigma} . \DD_B g(B) .\!\left[\frac{\DD^{\ZJ}}{\DD t}[B]\right] = \DD_{g(B)} \widetilde \sigma .\!\left[ \frac{\DD^{\ZJ}}{\DD t}[g(B)] \right]
	\end{align}
and
	\begin{align}
	\frac{\DD^{\GN}}{\DD t}[\widetilde{\sigma}(g(B))] = \DD_B \{\widetilde{\sigma}(g(B))\}.\!\left[\frac{\DD^{\GN}}{\DD t}[B]\right] = \DD_{g(B)} \widetilde{\sigma} . \DD_B g(B) .\!\left[\frac{\DD^{\GN}}{\DD t}[B]\right] = \DD_{g(B)} \widetilde \sigma .\!\left[ \frac{\DD^{\GN}}{\DD t}[g(B)] \right]
	\end{align}
instead. \\
\\
Let us independently check the validity of the equation \eqref{eqchainZJ01} for the Zaremba-Jaumann and \eqref{eqchainGN01} for the Green-Naghdi derivative by a direct calculation. To this end, let $R \in \OO(3)$ be the orthogonal rotation tensor with $\dot{R} \, R^T = \Omega$ and let $Q^W = Q = Q(t) \in \OO(3)$ be a matrix with $\dot{Q} \, Q^T = W$. Then we can express the corotational rates as
	\begin{align}
	\frac{\DD^{\ZJ}}{\DD t}[\sigma] = Q \, \frac{\DD}{\DD t} \, [Q^T \, \sigma \, Q] \, Q^T \qquad \textnormal{and} \qquad \frac{\DD^{\GN}}{\DD t}[\sigma] = R \, \frac{\DD}{\DD t}\, [R^T \, \sigma \, R] \, R^T \, .
	\end{align}
Since the upcoming calculation is the same for either rate (Zaremba-Jaumann with rotation matrix $Q$ and Green-Naghdi with matrix $R$), we will only carry it out for the Zaremba-Jaumann derivative. Writing $\widehat \sigma = \widehat \sigma(\log B)$ and recalling $\frac{\DD}{\DD t} B = L \, B + B \, L^T$ we obtain together with the chain rule for the material derivative $\frac{\DD}{\DD t}[\widehat \sigma(\log B)]$
	\begin{equation}
	\label{thelasteq}
	\begin{alignedat}{2}
	\frac{\DD^{\ZJ}}{\DD t}[\widehat \sigma] &= Q \, \dot{Q}^T \, \widehat \sigma \, Q \, Q^T + Q \, Q^T \, \widehat \sigma \, \dot{Q} \, Q^T + \frac{\DD}{\DD t}[\widehat \sigma(\log B)] \\
	&= Q \, \dot{Q}^T \, \widehat \sigma \, Q \, Q^T + Q \, Q^T \, \widehat \sigma \, \dot{Q} \, Q^T + \DD_{\log B} \widehat \sigma . \DD_B \log B.\!\left[\frac{\DD}{\DD t} B\right] \\
	&= W^T \, \widehat \sigma + \widehat \sigma \, W + \underbrace{\DD_{\log B} \widehat \sigma . \DD_B \log B}_{\DD_B [\widehat \sigma(\log B)]}.[L \, B + B \, L^T] \\
	&= \widehat \sigma \, W - W \, \widehat \sigma + \underbrace{\DD_B \widehat \sigma.[W \, B - B \, W]}_{\overset{\eqref{eqratetype4}}{=} W \, \widehat \sigma - \widehat \sigma W} + \DD_{\log B} \widehat \sigma . \DD_B \log B.[D \, B + B \, D] \\
	&= \DD_{\log B} \widehat \sigma . \DD_B \log B.[D \, B + B \, D] = \DD_{\log B} \widehat \sigma . \DD_B \log B.\!\left[\frac{\DD^{\ZJ}}{\DD t}[B]\right].
	\end{alignedat}
	\end{equation}
\subsubsection{Chain rule like formula for an arbitrary corotational rate} \label{chainruleallg}
We now generalize both Proposition~\ref{propallgchain01} and the above observations on the Zaremba-Jaumann and Green-Naghdi rates to a chain rule formula for arbitrary corotational rates and isotropic (not necessarily primary) matrix functions.
\begin{prop} \label{apppropa20}
Let $\frac{\DD^{\circ}}{\DD t}$ be an arbitrary corotational rate with spin tensor $\Omega^{\circ} \in \mathfrak{so}(3)$ and an isotropic, differentiable function $\sigma = \sigma(B) = \widehat \sigma(\log B)$. Then the chain rule
	\begin{align}
	\frac{\DD^{\circ}}{\DD t}[\widehat \sigma] = \DD_{\log B} \widehat \sigma(\log B) . \frac{\DD^{\circ}}{\DD t}[\log B]
	\end{align}
holds
\end{prop}
\begin{proof}
Recall the formula
	\begin{equation}
	\label{eqboxedrateeq}
	\boxed{\Omega^{\circ} \, \sigma(B) - \sigma(B) \, \Omega^{\circ} = \DD_B\sigma(B).[\Omega^{\circ} \, B - B \, \Omega^{\circ}],}
	\end{equation}
derived in Remark \ref{theoneremark}, alternatively expressed as
	\begin{equation}
	\boxed{[\Omega^{\circ}, \sigma(B)] = \DD_B \sigma(B).[\Omega^{\circ},B]}
	\end{equation}
for an arbitrary corotational derivative
	\begin{align}
	\frac{\DD^{\circ}}{\DD t}[\sigma] = \frac{\DD}{\DD t}[\sigma] - \Omega^{\circ} \, \sigma + \sigma \, \Omega^{\circ} = \frac{\DD}{\DD t}[\sigma] + [\sigma, \Omega^{\circ}]
	\end{align}
with spin tensor $\Omega^{\circ} \in \mathfrak{so}(3)$ and an isotropic, differentiable function $\sigma = \sigma(B) = \widehat \sigma(\log B)$. Repeating the calculations of \eqref{thelasteq} shows (with $\dot Q \, Q^T = \Omega^{\circ}$ and $\widehat \sigma = \widehat \sigma(\log B)$) that
	\begin{equation}
	\label{eqallglong1}
	\begin{alignedat}{2}
	\frac{\DD^{\circ}}{\DD t}[\widehat \sigma] &= Q \, \frac{\DD}{\DD t}[Q^T \, \widehat \sigma \, Q] \, Q^T = Q \, \dot{Q}^T \, \widehat \sigma \, Q \, Q^T + Q \, Q^T \, \widehat \sigma \, \dot{Q} \, Q^T + \frac{\DD}{\DD t}[\widehat \sigma(\log B)] \\
	&= Q \, \dot{Q}^T \, \widehat \sigma \, Q \, Q^T + Q \, Q^T \, \widehat \sigma \, \dot{Q} \, Q^T + \DD_{\log B} \widehat \sigma . \DD_B \log B.\!\left[\frac{\DD}{\DD t} B\right] \\
	&= (\Omega^{\circ})^T \, \widehat \sigma + \widehat \sigma \, \Omega^{\circ} + \underbrace{\DD_{\log B} \widehat \sigma . \DD_B \log B}_{\DD_B \sigma(B)}.\!\left[\frac{\DD^{\circ}}{\DD t}[B] + \Omega^{\circ} \, B - B \, \Omega^{\circ} \right] \\
	&= \underbrace{\sigma \, \Omega^{\circ} - \Omega^{\circ} \, \sigma + \underbrace{\DD_B \sigma.[\Omega^{\circ} \, B - B \, \Omega^{\circ}]}_{\overset{\eqref{eqboxedrateeq}}{=} \Omega^{\circ} \, \sigma - \sigma \, \Omega^{\circ}}}_{= \, 0} + \DD_{\log B} \widehat \sigma . \DD_B \log B.\!\left[\frac{\DD^{\circ}}{\DD t} B \right] \\
	&= \DD_{\log B} \widehat \sigma . \DD_B \log B.\!\left[\frac{\DD^{\circ}}{\DD t}[B]\right] = \DD_{\log B} \widehat \sigma(\log B) . \frac{\DD^{\circ}}{\DD t}[\log B],
	\end{alignedat}
	\end{equation}
where in $\eqref{eqallglong1}_1$ we used the chain rule for the material derivative $\frac{\DD}{\DD t}[\widehat \sigma(\log B)]$, in $\eqref{eqallglong1}_2$ that \break $\frac{\DD}{\DD t} B = \frac{\DD^{\circ}}{\DD t}[B] + \Omega^{\circ} \, B - B \, \Omega^{\circ}$ and in $\eqref{eqallglong1}_4$ that $\log B$ is a primary matrix function which admits a chain rule due to Proposition \ref{propallgchain01}. 
\end{proof}
\begin{rem}
The last equation of \eqref{eqallglong1} also shows that
	\begin{align}
	\frac{\DD^{\circ}}{\DD t}[\widehat \sigma] = \DD_B \sigma(B).[\DD_B \log B]^{-1}. \frac{\DD^{\circ}}{\DD t}[\log B],
	\end{align}
which extends results by Norris \cite[Lemma 2]{Norris2008} from primary matrix functions to arbitrary isotropic tensor functions.
\end{rem}
We conjecture that the corotational stress rates are, in fact, the only objective rates that satisfy a classical chain rule.
\begin{conjecture}[Chain rule and corotational rates]
Let $\frac{\DD^{\sharp}}{\DD t}$ be any objective rate. Assume that for all differentiable isotropic tensor functions $\sigma\colon \Sym^{++}(3) \to \Sym(3), \, B \mapsto \sigma(B)$ and all $B \in \Sym^{++}(3)$ we have the chain rule
	\begin{align}
	\frac{\DD^{\sharp}}{\DD t}[\sigma(B)] = \DD_B \sigma(B). \frac{\DD^{\sharp}}{\DD t}[B].
	\end{align}
Then $\frac{\DD^{\sharp}}{\DD t}$ is corotational.
\end{conjecture}
\begin{rem}[Leibniz-rule for derivations]
Any objective derivative can be seen as a covariant derivative \cite{kolev2024objective} which automatically satisfies the Leibniz rule of differentiation (cf.~\cite[p.~10]{kolev2024objective}). This means for $f \in C^1(\R, \R)$ and differentiable $\sigma\colon \Sym^{++}(3) \to \Sym(3), \, B \mapsto \sigma(B)$ we have
	\begin{align}
	\frac{\DD^{\sharp}}{\DD t}[f(t) \, \sigma(t)] = f(t) \, \frac{\DD^{\sharp}}{\DD t}[\sigma(t)] + f'(t) \, \sigma(t).
	\end{align}
We may check the latter for a general corotational rate:
	\begin{equation}
	\begin{alignedat}{2}
	\frac{\DD^{\circ}}{\DD t}[f \, \sigma] &= \frac{\DD}{\DD t}[f \, \sigma] + (f \, \sigma) \, \Omega^{\circ} - \Omega^{\circ} \, (f \, \sigma) \\
	&= f' \, \sigma + f \, \frac{\DD}{\DD t}[\sigma] + f \, (\sigma \, \Omega^{\circ} - \Omega^{\circ} \, \sigma) = f \, \frac{\DD^{\circ}}{\DD t}[\sigma] + f' \, \sigma.
	\end{alignedat}
	\end{equation}
\end{rem}
\begin{prop}[Product rule for corotational rates]
For two isotropic and differentiable tensor functions $\sigma_1, \sigma_2: \Sym^{++}(3) \to \Sym(3)$ we have the product rule in the form
	\begin{align}
	\frac{\DD^{\circ}}{\DD t}[\sigma_1(B) \, \sigma_2(B)] = \frac{\DD^{\circ}}{\DD t}[\sigma_1(B)] \, \sigma_2(B) + \sigma_1(B) \, \frac{\DD^{\circ}}{\DD t}[\sigma_2(B)] \, .
	\end{align}
\end{prop}
\begin{proof}
The proof follows easily by application of the chain rule,
	\begin{equation}
	\begin{alignedat}{2}
	\frac{\DD^{\circ}}{\DD t}[\sigma_1 \, \sigma_2] \overset{\textnormal{chain rule}}&{=} \DD_B[\sigma_1 \, \sigma_2] . \frac{\DD^{\circ}}{\DD t}[B] \overset{\text{product rule}}{=} [\DD_B \sigma_1 \, \sigma_2 + \sigma_1 \, \DD_B \sigma_2] . \frac{\DD^{\circ}}{\DD t}[B] \\
	&=\left(\DD_B \sigma_1. \frac{\DD^{\circ}}{\DD t}[B]\right) \, \sigma_2 + \sigma_1 \, \left(\DD_B \sigma_2.\frac{\DD^{\circ}}{\DD t}[B]\right) \overset{\textnormal{chain rule}}{=} \frac{\DD^{\circ}}{\DD t}[\sigma_1] \, \sigma_2 + \sigma_1 \, \frac{\DD^{\circ}}{\DD t}[\sigma_2]
	\,,
	\end{alignedat}
	\end{equation}
or alternatively by direct computation (independent of the chain rule):
	\begin{align}
	\frac{\DD^{\circ}}{\DD t}[\sigma_1 \, \sigma_2] &= \frac{\DD}{\DD t}[\sigma_1 \, \sigma_2] + \sigma_1 \, \sigma_2 \, \Omega^{\circ} - \Omega^{\circ} \, \sigma_1 \, \sigma_2 \notag \\
	&= \frac{\DD}{\DD t}[\sigma_1] \, \sigma_2 + \sigma_1 \, \frac{\DD}{\DD t}[\sigma_2] + \sigma_1 \, \sigma_2 \, \Omega^{\circ} - \sigma_1 \, \Omega^{\circ} \, \sigma_2 + \sigma_1 \, \Omega^{\circ} \, \sigma_2 - \Omega^{\circ} \, \sigma_1 \, \sigma_2 \\
	&= \left(\frac{\DD}{\DD t}[\sigma_1] + \sigma_1 \, \Omega^{\circ} - \Omega^{\circ} \, \sigma_1\right) \, \sigma_2 + \sigma_1 \, \left(\frac{\DD}{\DD t}[\sigma_2] + \sigma_2 \, \Omega^{\circ} - \Omega^{\circ} \, \sigma_2\right) = \frac{\DD^{\circ}}{\DD t}[\sigma_1] \, \sigma_2 + \sigma_1 \, \frac{\DD^{\circ}}{\DD t}[\sigma_2] \, . \notag \qedhere
	\end{align}
\end{proof}
\begin{rem}[Leibniz and product rules for the Truesdell rate]
Recalling the definition of the Truesdell rate (cf.~\cite[eq.~3]{truesdellremarks})
	\begin{align}
	\frac{\DD^{\TR}}{\DD t}[\sigma] = \frac{\DD}{\DD t}[\sigma] - L \, \sigma - \sigma \, L^T + \sigma \, \tr(D) \, ,
	\end{align}
we observe that the Truesdell rate satisfies the Leibniz rule, i.e.
	\begin{equation}
	\begin{alignedat}{2}
	\frac{\DD^{\TR}}{\DD t}[f(t) \, \sigma] &= f' \, \sigma + f \, \frac{\DD}{\DD t}[\sigma] - L \, (f \, \sigma) - (f \, \sigma) \, L^T + (f \, \sigma) \, \tr(D) \\ 
	&= f' \, \sigma + f \, \left(\frac{\DD}{\DD t}[\sigma] - L \, \sigma - \sigma \, L^T + \sigma \, \tr(D)\right) = f' \, \sigma + f \, \frac{\DD^{\TR}}{\DD t}[\sigma] \, .
	\end{alignedat}
	\end{equation}
However, \textbf{the product rule is not satisfied for the Truesdell rate} since
	\begin{equation}
	\begin{alignedat}{2}
	\frac{\DD^{\TR}}{\DD t}[\sigma_1 \, \sigma_2] &= \frac{\DD}{\DD t}[\sigma_1] \, \sigma_2 + \sigma_1 \, \frac{\DD}{\DD t}[\sigma_2] - L \, (\sigma_1 \, \sigma_2) - (\sigma_1 \, \sigma_2) \, L^T + (\sigma_1 \, \sigma_2) \, \tr(D) \\ 
	&= \sigma_1 \, \frac{\DD^{\TR}}{\DD t}[\sigma_2] + \frac{\DD}{\DD t}[\sigma_1] \, \sigma_2 - L \, \sigma_1 \, \sigma_2 + \sigma_1 \, L \, \sigma_2 \\
	&= \sigma_1 \, \frac{\DD^{\TR}}{\DD t}[\sigma_2] + \frac{\DD^{\TR}}{\DD t}[\sigma_1] \, \sigma_2 + \sigma_1 \, L \, \sigma_2 + \sigma_1 \, L^T \, \sigma_2 - \sigma_1 \, \sigma_2 \, \tr(D) \\
	&= \sigma_1 \, \frac{\DD^{\TR}}{\DD t}[\sigma_2] + \frac{\DD^{\TR}}{\DD t}[\sigma_1] \, \sigma_2 + \sigma_1 \, \left(2 \, D - \id \, \tr(D)\right) \, \sigma_2.
	\end{alignedat}
	\end{equation}
This underlines again our choice for using only corotational rates $\frac{\DD^{\circ}}{\DD t}$ (cf.~\cite{neffkoro2024}).
\end{rem}
\begin{rem}
The previously proven chain- and product rules are correct even if $\sigma$ is not symmetric and $\sigma_1 \, \sigma_2 \neq \sigma_2 \, \sigma_1$, since symmetry is not required in any of the proofs.
\end{rem}
\begin{rem}[Perfect elastic fluid]
For a perfect elastic fluid (cf.~\cite[p.~10]{Marsden83}) and an arbitrary corotational rate described by
	\begin{align}
	\frac{\DD^{\circ}}{\DD t}[\sigma] = \frac{\DD}{\DD t}[\sigma] - \Omega^{\circ} \, \sigma + \sigma \, \Omega^{\circ}
	\end{align}
with the spin tensor $\Omega^{\circ} \in \mathfrak{so}(3)$, we already see that the corotational derivative reduces to the material time derivative\footnote
{
This feature is lost for only objective rates $\frac{\DD^{\sharp}}{\DD t}$, such as the Truesdell rate.
}
of $\sigma$, i.e.
	\begin{align}
	\frac{\DD^{\circ}}{\DD t}[\sigma] = \frac{\DD}{\DD t}[\sigma] \qquad \textnormal{for $\sigma = h'(\sqrt{\det B}) \, \id$} \, ,
	\end{align}
so that for \textbf{any corotational rate}, $\H^{\circ}(\sigma).D$ for the perfectly compressible fluid has the form (for more details about the calculation see the Appendix~\ref{apphypcomp})
	\begin{align}
	\H^{\circ}(\sigma).D = h''(\sqrt{\det B}) \, \sqrt{\det B} \, \tr(D) \, \id \, .
	\end{align}
Checking positive definiteness of $\H^{\circ}(\sigma)$ leads to
	\begin{align}
	\langle \H^{\circ}(\sigma).D , D \rangle = \langle h''(\sqrt{\det B}) \, \sqrt{\det B} \, \tr(D) \, \id, D \rangle = h''(\sqrt{\det B}) \, \underbrace{\sqrt{\det B} \, \tr^2(D)}_{\ge \; 0}
	\end{align}
so that $\H^{\circ}(\sigma) \in \Sym^{++}_4(6)$ if and only if $h(x)$ is convex. This result remains true for the Cauchy-elastic case.
\end{rem}

\subsection{Corotational stability for the logarithmic rate} \label{corotationalstablog}
In Xiao et al.~\cite{xiao97_2, xiao97, xiao98_2, xiao98_1} the authors present the unique corotational rate in the family of corotational rates with spin tensors $\Omega^{\circ}$ of the form
	\begin{align}
	\label{eqmatspin11}
	\Omega^{\circ} &= W + \widetilde \Upsilon(B,D) = W + \nu_1 \, \sk(B \, D) + \nu_2 \, \sk(B^2 \, D) + \nu_3 \, \sk(B \, D \, B^2),
	\end{align}
called the ``logarithmic rate'' (cf.~\cite{Zhilin2013}). Here, each coefficient $\nu_k$ is an isotropic invariant of $B$, i.e. \break $\nu_k = \nu_k(I_1, I_2, I_3)$ with $I_1 = \tr B, \; I_2 = \tr(\Cof B), \; I_3 = \det B$. The logarithmic rate is given by
	\begin{align}
	\frac{\DD^{\log}}{\DD t}[\sigma] = \frac{\DD}{\DD t}[\sigma] + \sigma \, \Omega^{\log} - \Omega^{\log} \, \sigma, \qquad \Omega^{\log} = \Omega^{\log}(W, B,D) \in \mathfrak{so}(3),
	\end{align}
with the ``logarithmic spin'' $\Omega^{\log}$ (discovery roughly at the same time by different groups, cf.~Lehmann, Guo and Liang \cite{lehmann1991}, Reinhardt and Dubey \cite{reinhardt1996}, Xiao, Bruhns and Meyers \cite{xiao98_2} and Zhilin et al.~\cite{Zhilin2013}) that satisfies the defining relation
	\begin{align}
	\label{eqreallograte}
	\frac{\DD^{\log}}{\DD t}[\log V] = D \qquad \Bigl(\textnormal{and thus }\; \frac{\DD^{\log}}{\DD t}[\log B] = 2 \, D \Bigr).
	\end{align}
This shows conclusively that the Eulerian stretching $D = \sym \DD v$ is a true rate of the spatial logarithmic strain tensor $\log V$. While the analytical expression for the spin tensor $\Omega^{\log}$ (cf.~Footnote \ref{footnoteomega} in the Appendix) is involved, we can use the characterization \eqref{eqreallograte} to show now easily that
	\begin{align}
	\langle \frac{\DD^{\log}}{\DD t}[\sigma], D \rangle > 0 \quad \forall \, D \in \Sym(3) \! \setminus \! \{0\} \qquad \iff \qquad \langle \DD_{\log B} \widehat \sigma(\log B). D, D \rangle > 0 \quad \forall \, D \in \Sym(3) \! \setminus \! \{0\},
	\end{align}
i.e.~the equivalence of CSP with TSTS-M$^{++}$ also for the logarithmic rate. 
\begin{prop}(Equivalence of CSP with TSTS-M$^{++}$ for the logarithmic rate $\frac{\DD^{\log}}{\DD t}$) \label{proplograte}\\
Let $\sigma = \sigma(B) \in C^1(\Sym^{++}(3), \Sym(3))$ be an isotropic tensor function. Then for the logarithmic rate
	\begin{align}
	\frac{\DD^{\log}}{\DD t}[\sigma] = \frac{\DD}{\DD t}[\sigma] + \sigma \, \Omega^{\log} - \Omega^{\log} \, \sigma
	\end{align}
with the logarithmic spin $\Omega^{\log} = \dot Q^{\log} \, (Q^{\log})^T, \; Q^{\log} \in \OO(3)$ we have CSP $\iff$ TSTS-M$^{++}$.
\end{prop}
\begin{proof}
We write $\sigma(B) = \widehat \sigma(\log B)$ and calculate with the chain rule for corotational rates (Proposition \ref{apppropa20})
	\begin{align}
	\frac{\DD^{\log}}{\DD t}[\widehat \sigma(\log B)] = \DD_{\log B} \widehat \sigma (\log B). \frac{\DD^{\log}}{\DD t}(\log B) \overset{\eqref{eqreallograte}}{=} 2 \, \DD_{\log B} \widehat \sigma(\log B). D \, .
	\end{align}
In particular,
	\begin{equation}
	\langle \frac{\DD^{\log}}{\DD t}[\widehat \sigma(\log B)].D, D \rangle > 0 \quad\iff\quad \langle \DD_{\log B} \widehat \sigma(\log B). D , D \rangle > 0 \qquad\forall \, D \in \Sym(3) \! \setminus \! \{0\}
	\,,
	\end{equation}
which immediately implies the equivalence between CSP and $\DD_{\log B} \widehat \sigma(\log B) \in \Sym^{++}_4(6)$.
\end{proof}
\begin{rem}
The induced tangent stiffness tensor $\H^{\log}(\sigma)$ for the logarithmic rate takes an exceptionally simple form. Indeed, since
	\begin{align}
	\frac{\DD^{\log}}{\DD t}[\widehat \sigma(\log B)] = \DD_{\log B} \widehat \sigma(\log B).\!\left[ \frac{\DD^{\log}}{\DD t} [\log B] \right] = 2 \, \DD_{\log B} \widehat \sigma(\log B). D
	\end{align}
we obtain
	\begin{align}
	\H^{\log}(\sigma)\colonequals 2 \, \DD_{\log B} \widehat \sigma(\log B) = \DD_{\log V} \widehat \sigma(\log V)
	\end{align}
which can be truly called ``\textbf{logarithmic tangent stiffness}''. Moreover,
	\begin{align}
	\det \H^{\log}(\sigma) \neq 0 \qquad \iff \qquad \det \DD_{\log B} \widehat \sigma(\log B) \neq 0 \qquad \iff \qquad \det \DD_B \sigma(B) \neq 0,
	\end{align}
complementing Proposition \ref{propinvert1}.
\end{rem}
Furthermore, the corotational stability postulate takes a suggestive form for the logarithmic derivative (cf.~\eqref{eqdrucker}). Indeed,
	\begin{align}
	0 < \langle \frac{\DD^{\log}}{\DD t}[\sigma], D \rangle \overset{\textnormal{log-rate}}&{=} \langle \frac{\DD^{\log}}{\DD t}[\sigma], \frac{\DD^{\log}}{\DD t}[\log V] \rangle = \langle Q^{\log} \, \frac{\DD}{\DD t}[(Q^{\log})^T \, \sigma \, Q^{\log}] \, Q^{\log}, Q^{\log} \, \frac{\DD}{\DD t}[(Q^{\log})^T \, (\log V) \, Q^{\log}] \, Q^{\log} \rangle \notag \\
	&= \langle \frac{\DD}{\DD t}[(Q^{\log})^T \, \sigma \, Q^{\log}], \frac{\DD}{\DD t}[(Q^{\log})^T \, (\log V) \, Q^{\log}] \rangle \notag \\
	&= \langle \frac{\DD}{\DD t}[(Q^{\log})^T \, \widehat \sigma(\log V) \, Q^{\log}], \frac{\DD}{\DD t}[(Q^{\log})^T \, (\log V) \, Q^{\log}] \rangle \\
	\overset{\substack{\textnormal{isotropy} \\ \textnormal{of} \; \sigma}}&{=} \langle \frac{\DD}{\DD t}[\widehat \sigma((Q^{\log})^T \, (\log V) \, Q^{\log})], \frac{\DD}{\DD t}[(Q^{\log})^T \, (\log V) \, Q^{\log}] \rangle \notag \\
	\overset{\substack{\textnormal{isotropy} \\ \textnormal{of} \; \log}}&{=} \langle \frac{\DD}{\DD t}[\widehat \sigma(\underbrace{\log((Q^{\log})^T \, V \, Q^{\log})}_{\equalscolon \, \widetilde \varepsilon})], \frac{\DD}{\DD t}[\underbrace{\log((Q^{\log})^T \, V \, Q^{\log})}_{\equalscolon \,\widetilde \varepsilon}] \rangle= \langle \frac{\DD}{\DD t}[\widehat \sigma(\widetilde \varepsilon)], \frac{\DD}{\DD t} \widetilde \varepsilon \rangle = \langle \DD \widehat \sigma(\widetilde \varepsilon). \dot{\widetilde \varepsilon}, \dot{\widetilde \varepsilon} \rangle \, . \notag
	\end{align}
Thus
	\begin{align}
	\hspace*{-10pt} \textnormal{CSP} \;\; \iff \;\; \sym \, \DD \widehat \sigma(\widetilde \varepsilon) \in \Sym^{++}_4(6) \;\; \iff \;\; \sym \DD_{\log V} \widehat \sigma(\log V) \in \Sym^{++}_4(6) \;\; \iff \;\; \textnormal{TSTS-M$^{++}$} ,
	\end{align}
arriving, once again, at the correspondence (now for the \textbf{logarithmic rate}):
	\begin{center}
	\fbox{
	\begin{minipage}[h!]{0.85\linewidth}
		\centering
		\textbf{For an isotropic Cauchy-elastic material, the corotational stability postulate (CSP) $\langle \frac{\DD^{\log}}{\DD t}[\sigma] , D \rangle > 0$ is equivalent to the strong Hilbert-monotonicity of $\widehat \sigma$ in $\log V$, which is TSTS-M$^{++}$.}
	\end{minipage}}
	\end{center}
\subsection{Material spins, invertibility considerations and positive corotational rates} \label{app13} \label{appendixAA}
The material and exposition in this section is needed as a preparatory step for tackling Conjecture \ref{conjCSPgeneral} in terms of matrix analysis instead of the cumbersome method of Lagrangean axes. The bulk of this work is displayed in \cite{neffkoro2024}. Here, we will solely give a short overview of the general concept and its implications on invertibility for a fourth order tangent stiffness tensor $\H^{\circ}(\sigma)$ induced by an arbitrary corotational rate $\frac{\DD^{\circ}}{\DD t}$.

Therefore, we begin by asking the leading question, whether or not it is always possible to write an induced fourth-order tangent stiffness tensor $\H^{\circ}(\sigma)$ in the form
	\begin{align}
	\label{eqfirsthalf}
	\H^{\circ}(\sigma).D = \frac{\DD^{\circ}}{\DD t}[\sigma] = \frac{\DD^{\circ}}{\DD t}[\widehat \sigma(\log B)] \overset{\text{chain rule}}{=} \DD_{\log B} \widehat \sigma(\log B). \DD_B \log B. \mathbb{A}^{\circ}(B).D \, ,
	\end{align}
as this identity would require the corotational derivative $\frac{\DD^{\circ}}{\DD t}$ to have the representation
	\begin{align}
	\label{eqisthatcorrect1}
	\frac{\DD^{\circ}}{\DD t}[B] = \frac{\DD}{\DD t}[B] + B \, \Omega^{\circ} - \Omega^{\circ} \, B = \mathbb{A}^{\circ}(B).D \, .
	\end{align}
In an attempt to satisfyingly answer this question, we resort to the already mentioned (cf.~\eqref{conjCSPgeneral}) class of objective corotational rates with so-called \textbf{material spins}. In this regard, we have the following Theorem \cite[p.22]{xiao98_1} and Definition \cite[p.25]{xiao98_1} from Xiao et al.:
\begin{thm} \label{thm3.19}
Let the spin tensor $\Omega^{\circ}$ of an objective, corotational rate be associated with the deformation and rotation of a deforming material body as indicated by 
	\begin{align}
	\Omega^{\circ} = \Upsilon(B, D, W)\,.
	\end{align}
Moreover, let the tensor function $\Upsilon$ be continuous with respect to the argument $B$. Then the corotational rate of any Eulerian strain measure $e$ defined by $\Omega^{\circ}$ is objective if and only if
	\begin{align}
	\label{eqxiao1}
	\Omega^{\circ} = W + \widetilde \Upsilon(B,D),
	\end{align}
where $\widetilde \Upsilon$ is an isotropic skew symmetric tensor-valued function of $B$ and $D$ that is continuous with respect to the argument $B$.
\end{thm}
\begin{proof}
See proof of Theorem 4 in \cite[p.~22]{xiao98_1}.
\end{proof}
\begin{definition}[\textbf{Material spins}] \label{appmaterialspins}
An objective corotational rate of the form \eqref{eq:general_corotational_rate} with $\Omega^{\circ}$ given by \eqref{eqxiao1} is called a \emph{material spin} if $\widetilde \Upsilon(B,D)$ is \textbf{linear in $D$}, \textbf{isotropic in $B$ and $D$} and satisfies the \textbf{homogeneity condition $\widetilde \Upsilon(\alpha \, B, D) = \widetilde \Upsilon(B,D)$} for all $\alpha > 0$.
\end{definition}
\noindent This set defines a physically reasonably large subclass of spin tensors for objective corotational rates which includes all known corotational rates. Additionally, these objective corotational rates $\frac{\DD^{\circ}}{\DD t}$ admit the general structure (cf.~\cite{neffkoro2024})
	\begin{equation}
	\label{eqAAb}
	\begin{alignedat}{2}
	\frac{\DD^{\circ}}{\DD t}[B] &= D \, B + B \, D + \nu_1 \, (B \, \sk(B \, D) - \sk(B \, D) \, B) +\nu_2 \, (B \, \sk(B^2 \, D) - \sk(B^2 \, D) \, B) \\
	&\qquad + \nu_3 \, (B \, \sk(B^2 \, D \, B) - \sk(B^2 \, D \, B) \, B) \equalscolon \mathbb{A}^{\circ}(B).D \, ,
	\end{alignedat}
	\end{equation}
so that \eqref{eqisthatcorrect1} is satisfied and $\H^{\circ}(\sigma)$ can indeed be expressed in the form \eqref{eqfirsthalf}, i.e.
	\begin{align}
	\label{eqapp227}
	\H^{\circ}(\sigma).D = \DD_{\log B} \widehat \sigma(\log B). \DD_B \log B. \mathbb{A}^{\circ}(B).D = \DD_B \sigma(B).\mathbb{A}^{\circ}(B).D
	\end{align}
with $\mathbb{A}^{\circ}(B)$ given in \eqref{eqAAb}, extending the result in Norris \cite[Lemma 1]{Norris2008} from primary matrix functions to arbitrary isotropic tensor functions. Additionally, it is shown in \cite{neffkoro2024} that the tensor $\mathbb{A}^{\circ}(B)$ is \textbf{minor and major symmetric}. Evaluation of \eqref{eqapp227} at $B = \id$ with stress free initial state ($\sigma(\id) = 0$) gives the linear elastic-like relation
	\begin{align}
	\H^{\circ}(0).D = 2 \, \DD_{\log B} \widehat \sigma(0).D = \DD_{\log V} \widehat \sigma(0).D = \C^{\iso}.D = 2 \, \mu \, D + \lambda \, \tr(D) \, \id \, .
	\end{align}
We would also like to point out the following result regarding the invertibility of $\mathbb{A}^{\circ}(B)$, further complementing Proposition \ref{propinvert1}:
\begin{prop} \label{prop3.23a}
Let $\sigma\colon \Sym^{++}(3) \subset \Sym(3) \to \Sym(3)$ be an invertible, isotropic and differentiable tensor function. Consider \textbf{any} corotational rate $\frac{\DD^{\circ}}{\DD t}$ with material spin tensor $\Omega^{\circ}$ of the form \eqref{eqxiao1} and let $\H^{\circ}(\sigma)$ be the induced tangent stiffness tensor. If $\mathbb{A}^{\circ}(B)$ is invertible, then we have the equivalence
	\begin{align}
	\det \H^{\circ}(\sigma) \neq 0 \quad \iff \quad \det \DD_{\log B} \widehat \sigma(\log B) \neq 0 \quad \iff \quad \det \DD_B \sigma(B) \neq 0
	\end{align}
where $\mathbb{A}^{\circ}(B)$ is given in \eqref{eqAAb}.
\end{prop}
\begin{proof}
This follows instantly by the representation formula \eqref{eqapp227} on noting that $\DD_B \log B \in \Sym^{++}_4(6)$.
\end{proof}
\noindent Lastly, in view of the corotational stability postulate, it now appears reasonable to focus the attention on \textbf{positive corotational derivatives}, a subclass of corotational derivatives with material spins given by
\begin{definition}[Positive corotational derivatives] \label{appendixpcd}
Positive corotational derivatives $\frac{\DD^{\circ}}{\DD t}$ are those objective corotational derivatives with material spins given by \eqref{eqxiao1} such that for $\mathbb{A}^{\circ}(B)$ defined by $\frac{\DD^{\circ}}{\DD t}[B] = \mathbb{A}^{\circ}(B).D$ we have $\mathbb{A}^{\circ}(B) \in \Sym^{++}_4(6)$.
\end{definition}
\noindent In the sense of this definition, $\frac{\DD^{\ZJ}}{\DD t}, \frac{\DD^{\GN}}{\DD t}$ and $\frac{\DD^{\log}}{\DD t}$ qualify as positive corotational rates. As mentioned in Conjecture \ref{conjCSPgeneral} we intend to show CSP $\iff$ TSTS-M$^{++}$ for all positive corotational rates in a future contribution without resorting to the burdensome principal axis calculus given in \cite{tobedone}. The result for the logarithmic rate in Proposition \ref{proplograte} shows that this result is plausible.

%
%\red{
%TODO positive definiteness in the representation with $V$ \\
%\\
%$\frac{\DD^{\circ}}{\DD t}[V] = \overline{\mathbb{A}}^{\circ}(V).D; \; \overline{\mathbb{A}}^{\circ}(V) \in \Sym^{++}_4(6) \; \overset{?}{\iff} \; \mathbb{A}^{\circ}(B) \in \Sym^{++}_4(6)$ \\
%\\
%Calculate $\ZJ, \GN, \log$ for $V$ instead of $B$.
%}
%
%
%
%
%
\clearpage
\section{Synthesis of the results} \label{sec6}
\subsection{Connection to the BCH-formula}
A much needed perspective for the obtained result will be tentatively presented next. We have seen in Theorem \ref{thm4.3} by lengthy calculations for the Zaremba-Jaumann rate (cf.~\cite{tobedone}) that for an isotropic tensor function \break $\widehat \sigma\colon \Sym(3) \to \Sym(3)$ satisfying $\widehat \sigma(Q^T \, S \, Q) = Q^T \, \widehat \sigma(S) \, Q$ for all $Q \in \OO(3)$,
	\begin{align}
	\hspace*{-8pt} \langle \DD_{\log B} \widehat \sigma(\log B) . \DD_B \log B . [B \, D + D \, B], D \rangle > 0 \quad \forall \, D \in \Sym(3) \! \setminus \! \{0\} \quad \iff \quad \langle \DD_{\log B} \widehat \sigma(\log B). D, D \rangle > 0.
	\end{align}
Thus, this equivalence would have been an easy observation if $\DD_B \log B.[B \, D + D \, B] = 2 \, D$, which is, however \textbf{false} if $B \, D \neq D \, B$ (cf.~\eqref{eqboxed2D}). Instead, the commutator $[B, D] \colonequals B \, D - D \, B$ naturally needs to be taken into account. In the following table we gather some pertinent observations in this direction. \\

\begingroup
\renewcommand{\arraystretch}{1.35}
\centering
\begin{tabular}{|c | c | c|}
\hline
On the one hand... & but still... & as seen in... \\
\hline \hline
$\exp(B + D)\neq \exp(B)\, \exp(D)$ & \makecell{if $B \, D = D \, B$ then \\ $\exp(B + D) = \exp(B)\, \exp(D)$} & \eqref{eqappendixcommute} \\
\hline
$\DD_B \log B.[B \, D + B \, D] \neq 2 \, D$ & \makecell{if $B \, D = D \, B$ then \\ $\DD_B \log B.[B \, D + D \, B] = 2 \, D$} & Proposition \ref{propA17} \\ 
\hline
$\DD_B \log B.H\neq B^{-1}H$ & \makecell{if $B \, H = H \, B$ then \\ $\DD_B \log B. H = B^{-1} \, H$} & \eqref{eqproofofa31} \\
\hline \hline
\makecell{$\DD_B \log B.H \neq B^{-1} \, H$ \\in general if $BH\neq HB$}& \makecell{$\DD_B [\widehat \WW(\log B)] = \DD_{\log B} \widehat \WW(\log B) \, B^{-1}$ \\ for an isotropic and scalar valued function $\WW$} & \makecell{Richter-Valleé \\ representation \\ cf.~Appendix~\ref{Appensansour}} \\
\hline 
$\DD_B \log B.[B \, D + B \, D] \neq 2 \, D$ & $\tr(\DD_B \log B . [B \, D + D \, B]) = \tr(2 \, D)$ & \eqref{eqappendixtrace1} \\
\hline
$\DD_B \log B.[B \, D + B \, D] \neq 2 \, D$ & $\langle \DD_B \log B.[B \, D + D \, B], D \rangle \ge c^+ \, \norm{2 \, D}^2$ & Lemma \ref{dopeasslemma001} \\
\hline
\makecell{$\exp(B+D) \neq \exp(B) \, \exp(D)$ \\ in general if $B \, D \neq D \, B$} & $\tr(\exp(B+D)) \le \tr(\exp(B) \, \exp(D))$ & \makecell{Golden-Thompson \\ inequality \\ cf.~Appendix~\ref{appgoldenth}} \\
\hline
\makecell{$\exp(B+D) \neq \exp(B) \, \exp(D)$ \\ in general if $B \, D \neq D \, B$} & $\tr(\log(\exp(B) \, \exp(D))) = \tr(B) + \tr(D)$ & implication of BCH\footnotemark \\
\hline
$\DD_B \log B.[B \, D + D \, B] \neq 2 \, D$ & \makecell{$\widehat \sigma\colon \Sym(3) \to \Sym(3)$ isotropic tensor function \\ $\langle \DD_{\log B} \widehat \sigma(\log B).(\DD_B \log B.[B \, D + D \, B]), D \rangle > 0$ \\ $\iff \quad \langle \DD_{\log B} \widehat \sigma(\log B). 2 \, D, D \rangle > 0, \quad \forall D \in \Sym(3) \! \setminus \! \{0\}$ \\ $\iff \quad \sym \, \DD_{\log B} \widehat \sigma(\log B) \in \Sym^{++}_4(6)$} & \makecell{CSP $\iff$ TSTS-M$^+$ \\ ``corotational stability \\ postulate''} \\
\hline
$\DD_B \log B.[B \, D + D \, B] \neq 2 \, D$ & \makecell{$\widehat \tau\colon\Sym(3) \to \Sym(3)$ isotropic tensor function \\ $\langle \DD_{\log B} \widehat \tau(\log B).(\DD_B \log B.[B \, D + D \, B]), D \rangle > 0$ \\ $\iff \quad \langle \DD_{\log B} \widehat \tau(\log B). 2 \, D, D \rangle > 0, \quad \forall D \in \Sym(3) \! \setminus \! \{0\}$ \\ $\iff \quad \DD_{\log B} \widehat \tau(\log B) \in \Sym^{++}_4(6)$}  & \makecell{hyperelastic \\ Hill's inequality \\ cf.~\cite{tobedone}} \\ 
\hline
\end{tabular}
\footnotetext
{
In general the \emph{Baker-Campbell-Hausdorff} (BCH) formula (cf.~\cite[p.861]{hofmann1998}) determines the expression $Z = Z(B, D)$ so that
	\begin{align*}
	\exp(B) \, \exp(D) = \exp(Z(B,D)) \qquad \iff \qquad \log(\exp(B) \, \exp(D)) = Z(B, D)
	\end{align*}
and $Z(B, D)$ involves iterated commutator-brackets, i.e.
	\begin{align*}
	Z(B,D) = \log(\exp(B) \, \exp(D)) = B + D + \frac12 \, [B, D] + \frac{1}{12} \, [B,[B,D]] - \frac{1}{12} \, [D,[B,D]] + \textnormal{h.o.t.}
	\end{align*}
}
\endgroup \\
\\[2em]
Note that this set of properties does have a common theme. Indeed, isotropic functions or operators can make up for lack of commutation (i.e.\ $[B , D] \neq 0$) in the Baker-Campbell-Hausdorff (BCH) \cite[p.861]{hofmann1998} formula if inequalities are concerned and our statement in Theorem \ref{thm4.3} is just another instance of this general observation.
\subsection{Continuum mechanics perspective - constitutive requirements}
From a continuum mechanics perspective with a view on constitutive requirements we can summarize our results in the following diagram (the result holds as well with $\frac{\DD^{\ZJ}}{\DD t}$ replaced by $\frac{\DD^{\log}}{\DD t}$).
\begin{center}
	\small
	\hspace*{-3.43em}
	\begin{tikzpicture}
	\node (A) at (0,0) {%
		$\displaystyle
		\left.
			\begin{array}{llc}
			&\Bigg\{&
				\begin{array}{c}
				\textbf{corotational stability postulate (CSP)} \\ 
				\textbf{for the Cauchy stress $\sigma$} \\
				\langle \frac{\DD^{\ZJ}}{\DD t}[\sigma], D \rangle = \langle \H^{\ZJ}(\sigma).D,D \rangle > 0 \\
				\end{array}
			\\
			\iff \quad &\Bigg\{&
				\begin{array}{c}
				\textbf{TSTS-M$^{++}$}: \\
				\sym \DD_{\log B} \widehat \sigma(\log B) \in \Sym^{++}_4(6) \\
				\implies \quad \langle \widehat \sigma(\log V_1) - \widehat \sigma(\log V_2), \log V_1 - \log V_2 \rangle > 0
				\end{array}
			\end{array}
		\right\}
		$
	};
	\node at (5.88,0) {%
	$\displaystyle
		\notiff
		$
	};
	\node (B) at (10.29,0) {%
	$\displaystyle
		\left\{
			\begin{array}{l}
			\textbf{Legendre-Hadamard ellipticity}: \\
			\DD_F^2 \WW(F). (\xi \otimes \eta, \xi \otimes \eta) \ge c^+ \, |\xi|^2 \cdot |\eta|^2 \\
			\langle S_1(F + \xi \otimes \eta) - S_1(F), \xi \otimes \eta \rangle \ge c^+ \, |\xi|^2 \, |\eta|^2
			\end{array} \right.
			\vphantom{
				\begin{array}{llc}
				&\Bigg\{&
					\begin{array}{c}
					\textbf{corotational stability postulate (CSP)} \\ 
					\textbf{for the Cauchy stress $\sigma$} \\
					\langle \frac{\DD^{\ZJ}}{\DD t}[\sigma], D \rangle = \langle \H^{\ZJ}(\sigma).D,D \rangle > 0 \\
					\end{array}
				\\
				\iff \quad &\Bigg\{&
					\begin{array}{c}
					\textbf{TSTS-M$^{++}$}: \\
					\sym \DD_{\log B} \widehat \sigma(\log B) \in \Sym^{++}_4(6) \\
					\implies \quad \langle \widehat \sigma(\log V_1) - \widehat \sigma(\log V_2), \log V_1 - \log V_2 \rangle > 0
					\end{array}
				\end{array}
			}
		$
	};
	\node (C) at (-.98,-3.43) {%
		$\displaystyle
		\textnormal{BE}: \quad (\sigma_i - \sigma_j) \, (\lambda_i - \lambda_j) \ge 0
		\vphantom{
			\textnormal{TE\footnotemark}: \quad \frac{\partial \sigma_i(\lambda_1, \lambda_2, \lambda_3)}{\partial \lambda_i} \ge 0 \quad \iff \quad \frac{\partial \widehat \sigma_i(\log \lambda_1, \log \lambda_2, \log \lambda_3)}{\partial \log \lambda_i} \ge 0
		}
		$
	};
	\node (D) at (8.33,-3.43) {%
		$\displaystyle
		\textnormal{TE\footnotemark}: \quad \frac{\partial \sigma_i(\lambda_1, \lambda_2, \lambda_3)}{\partial \lambda_i} \ge 0 \quad \iff \quad \frac{\partial \widehat \sigma_i(\log \lambda_1, \log \lambda_2, \log \lambda_3)}{\partial \log \lambda_i} \ge 0
		$
	};
	
	\coordinate (X) at ($(B.south west)!.5!(B.south)$);
	\draw[-implies,double equal sign distance] (A) -- (C.north-|A);
	\draw[-implies,double equal sign distance] (A.south) -- ($(D.north west)!.5!(D.north)$);
	\draw[-implies,double equal sign distance] (X) -- (C.north east);
	\draw[-implies,double equal sign distance] (X) -- (D.north-|X);
	\end{tikzpicture}
\end{center}
\noindent Note that TSTS-M$^{++}$ implies the tension-extension (TE) inequality since the positive definiteness of (the symmetric part of) $\frac{\partial \widehat \sigma_i(\log \lambda_1, \log \lambda_2, \log \lambda_3)}{\partial \log \lambda_j}$ implies that the diagonal entries $\frac{\partial \widehat \sigma_i}{\partial \log \lambda_i}$ are positive.

\smallskip\noindent For purely volumetric energy functions (elastic fluids, cf.~\cite[p.~10]{Marsden83} and the Appendix~\ref{appvolufunc}), the equivalences
	\begin{center}
	\fbox{
	\begin{minipage}[b][1.05cm]{0.8\linewidth}
	\begin{align*}
	&\WW(F) = h(\det F) \\[.343em]
	\textnormal{CSP} \quad \iff \quad \textnormal{TSTS-M$^{++}$} \quad \iff \quad &\textnormal{LH-ellipticity} \iff \quad \WW \; \textnormal{polyconvex} \quad \iff \quad h \; \textnormal{convex}
	\end{align*}
	\end{minipage}}
	\end{center}
%\footnotetext
%{
%The implication TSTS-M$^{++}$ $\implies$ tension-extension (TE) inequality is a consequence of \newline $\sym \frac{\partial \widehat \sigma_i(\log \lambda_1, \log \lambda_2, \log \lambda_3)}{\partial \log \lambda_j} \in \Sym^{++}(3)$ which implies the positivity of the diagonal entries $\frac{\partial \widehat \sigma_i}{\partial \log \lambda_i}$.
%}
%
\noindent hold, which gives support to using exclusively the Cauchy stress $\sigma$ in setting up the CSP-condition. Indeed, if we would use the spatial Kirchhoff stress $\tau = J \, \sigma$ instead, a non-convex $h$ would be permitted, e.g.~$h(\det F) = (\log \det F)^2 = \tr^2(\log V)$, which is unphysical. 

In this way, we answer the issues concerning the use of hypo-elasticity to discern constitutive properties for nonlinear elasticity as raised in the introduction as follows:
	\begin{enumerate}
	\item[1a)] which rate $\frac{\DD^{\sharp}}{\DD t}$? $\rightarrow$ objective, corotational rates $\frac{\DD^{\circ}}{\DD t}$ only
	\item[1b)] which corotational rate $\frac{\DD^{\circ}}{\DD t}$? $\rightarrow$ presumably largely irrelevant for the equivalence CSP $\iff$ TSTS-M$^{++}$ \break \hspace*{133pt} (any positive corotational rate might do, cf.~Section \ref{appendixAA} and \cite{neffkoro2024})
	\item[2)] which stress? $\rightarrow$ only the true/Cauchy stress $\sigma$
	\item[3)] which tangent stiffness tensor $\H$? $\rightarrow$ only the induced tangent stiffness tensor $\H^{\circ}(\sigma)$ \newline \hspace*{5.6cm} for the corotational rate $\frac{\DD^{\circ}}{\DD t}$.
	\end{enumerate}
Expressed by the logarithmic strain tensor $\log V$, we obtain for the idealized isotropic nonlinear elastic constitutive law given as a mapping $\log V \to \widehat \sigma(\log V) = \sigma(V)$ in terms of a rate-formulation in the spatial setting the following concordance:
	\begin{alignat}{3}
	\log V &\mapsto \widehat \sigma(\log V) \qquad && \longhookrightarrow \quad \H^{\ZJ}(B)\,, \nonumber\\
	\log V &\mapsto \widehat \sigma(\log V) \; \textnormal{invertible} \qquad && \longhookrightarrow \quad \H^{\ZJ}(\sigma)\,, \nonumber\\
	\log V &\mapsto \widehat \sigma(\log V) \; \textnormal{invertible and} \det \DD_{\log V} \widehat \sigma(\log V) \neq 0 \qquad &&\longhookrightarrow \quad \H^{\ZJ}(\sigma) \; \textnormal{invertible}\,, \nonumber\\[.7em]
	\log V &\mapsto \widehat \sigma(\log V) \; \textnormal{strongly monotone} \qquad &&\!\! \iff \quad \sym \H^{\ZJ}(\sigma) \in \Sym^{++}_4(6)\,. \label{eq:ut_tensio}
	\end{alignat}
The last requirement \eqref{eq:ut_tensio} represents our proposal for the notion that \underline{\textbf{{stress increases with strain}}} in idealized isotropic nonlinear elasticity, generalizing Hooke's linear elasticity law \underline{$\boldsymbol{\mathfrak{ut \; tensio\smash{,}\,\, sic \; vis}}$}.

\footnotesize
\clearpage
\section{Acknowledgement}
The first author is especially grateful to \textbf{Sergey N. Korobeynikov} (Lavrentyev Institute of Hydrodynamics of Russian Academy of Science, Novosibirsk). He acknowledges critical discussions and helpful remarks by \textbf{Zden\v{e}k Fiala} (Institute of Theoretical and Applied Mechanics of the Academy of Sciences of the Czech Republic, Prague, Czech Republic), \textbf{Daniel Aubram} (Technische Universität Berlin), \textbf{Salvatore Federico} (University of Calgary, Canada), \textbf{Ray Ogden} (University of Glasgow), 
\textbf{Reza Naghdabadi} (Sharif University of Technology, Tehran), \textbf{Soumya Mukherjee} (Indian Institute of Technology, Madras), 
%\red{\textbf{Andrew Norris} (Rutgers University, New Jersey)}, 
\textbf{Sanjay Govindjee} (University of California, Berkeley), \textbf{David Steigmann} (University of California, Berkeley), \textbf{Otto Timme Bruhns} (Ruhr-Universität Bochum), \textbf{Frank Ihlenburg} (Hochschule für Angewandte Wissenschaften, Hamburg), \textbf{Jean-Baptiste Leblond} (member of l'Académie des Sciences and professor at the Pierre-et-Marie-Curie University, Paris) and \textbf{Rebecca Brannon} (University of Utah) at various times in setting up this manuscript. \\
\\
\printbibliography
%\bibliographystyle{plain} %plain
%\bibliography{Leblondrefs}
%
\normalsize
\begin{appendix}
\section{Appendix}
\subsection{Notation} \label{appendixnotation}
\textbf{The deformation $\varphi(x,t)$, the material time derivative $\frac{\DD}{\DD t}$ and the partial time derivative $\partial_t$} \\
\\
In accordance with \cite{Marsden83} we agree on the following convention regarding an elastic deformation $\varphi$ and time derivatives of material quantities:

Given two sets $\Omega, \Omega_{\xi} \subset \R^3$ we denote by $\varphi: \Omega \to \Omega_{\xi}, x \mapsto \varphi(x) = \xi$ the deformation from the \emph{reference configuration} $\Omega$ to the \emph{current configuration} $\Omega_{\xi}$. A \emph{motion} of $\Omega$ is a time-dependent family of deformations, written $\xi = \varphi(x,t)$. The \emph{velocity} of the point $x \in \Omega$ is defined by $\overline{V}(x,t) = \partial_t \varphi(x,t)$ and describes a vector emanating from the point $\xi = \varphi(x,t)$ (see also Figure \ref{yfig1}). Similarly, the velocity viewed as a function of $\xi \in \Omega_{\xi}$ is denoted by $v(\xi,t)$. 

	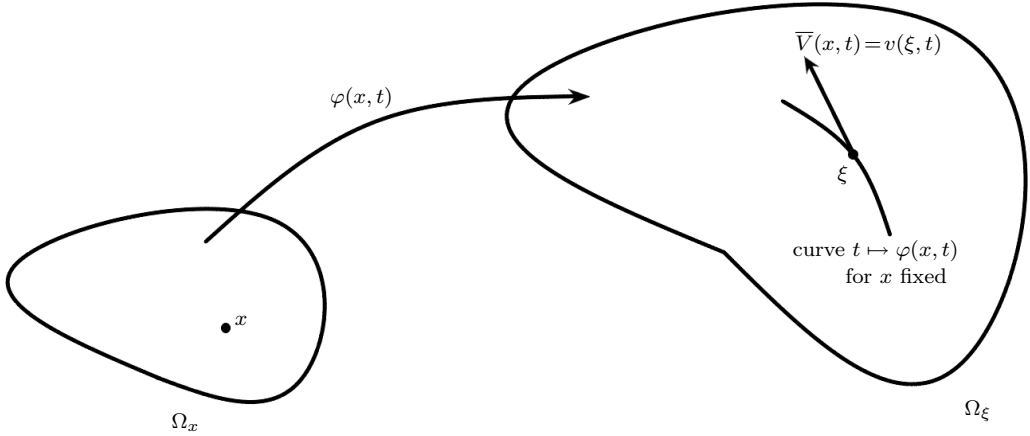
\begin{figure}[h!]
		\begin{center}		
		\begin{minipage}[h!]{0.8\linewidth}
			\centering
			\hspace*{-40pt}
			\includegraphics[scale=0.4]{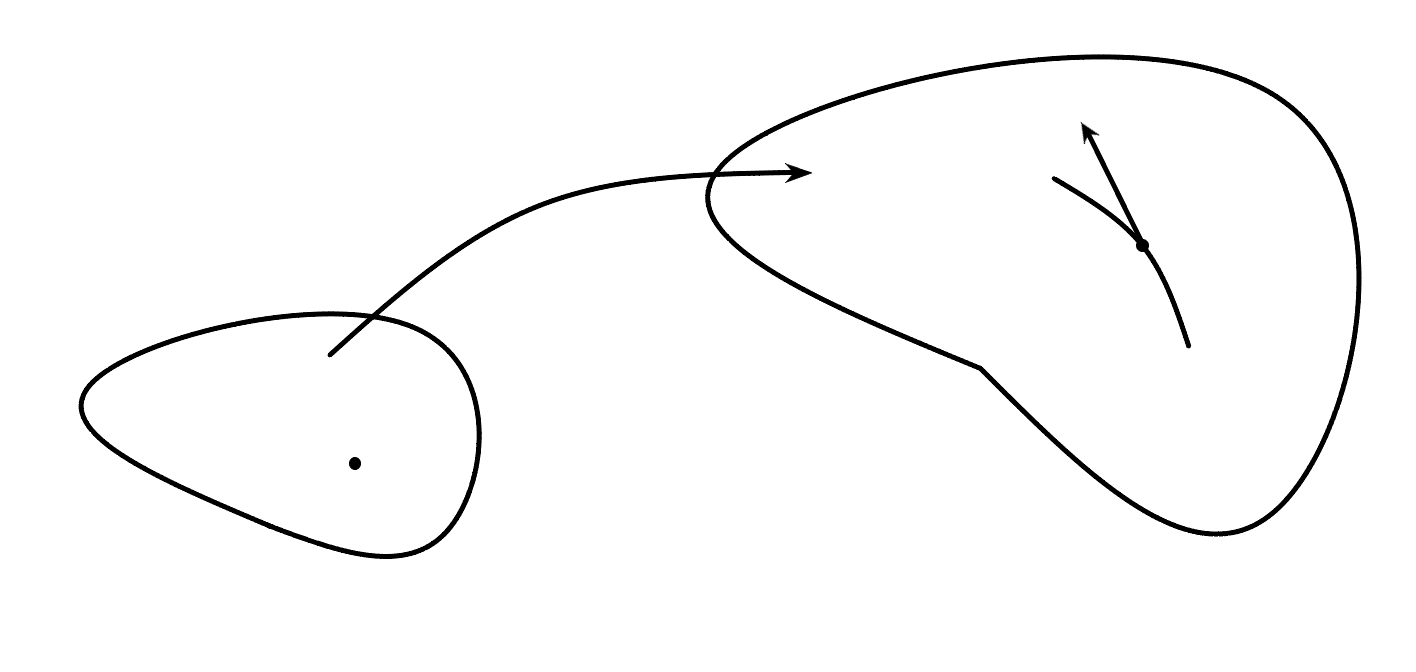}
			\put(-40,30){\footnotesize $\Omega_\xi$}
			\put(-340,25){\footnotesize $\Omega_x$}
			\put(-316,64){\footnotesize $x$}
			\put(-280,148){\footnotesize $\varphi(x,t)$}
			\put(-104,168){\footnotesize $\overline V(x,t) \!=\! v(\xi,t)$}
			\put(-88,119){\footnotesize $\xi$}
			\put(-105,90){\footnotesize curve $t \mapsto \varphi(x,t)$}
			\put(-85,80){\footnotesize  for $x$ fixed}
		\end{minipage} 
		\caption{Illustration of the deformation $\varphi(x,t): \Omega_x \to \Omega_{\xi}$ and the velocity $\overline V(x,t) = v(\xi,t)$.}
		\label{yfig1}
		\end{center}
	\end{figure}

Considering an arbitrary material quantity $Q(x,t)$ on $\Omega$, equivalently represented by $q(\xi,t)$ on $\Omega_\xi$, we obtain by the chain rule for the time derivative of $Q(x,t)$
	\begin{align}
	\frac{\DD}{\DD t}q(\xi,t) \colonequals \frac{\dif}{\dif t}[Q(x,t)] = \DD_\xi q(\xi,t).v + \partial_t q(\xi,t) \, .
	\end{align}
Since it is always possible to view any material quantity $Q(x,t) = q(\xi,t)$ from two different angles, namely by holding $x$ or $\xi$ fixed, we agree to write
	\begin{itemize}
	\item $\dot q \colonequals \dd \frac{\DD}{\DD t}[q]$ for the material (substantial) derivative of $q$ with respect to $t$ holding $x$ fixed and
	\item $\partial_t q$ for the derivative of $q$ with respect to $t$ holding $\xi$ fixed.
	\end{itemize}
For example, we obtain the velocity gradient $L := \DD_\xi v(\xi,t)$ by
	\begin{align}
	L = \DD_\xi v(\xi,t) = \DD_\xi \overline V(x,t) \overset{\text{def}}&{=} \DD_\xi \frac{\dif}{\dif t} \varphi(x,t) = \DD_{\xi} \partial_t \varphi(\varphi^{-1}(\xi,t),t) = \partial_t \DD_x \varphi(\varphi^{-1}(\xi,t),t) \, \DD_\xi \big(\varphi^{-1}(\xi,t)\big) \notag \\
&=  \partial_t \DD_x \varphi(\varphi^{-1}(\xi,t),t) \, (\DD_x \varphi)^{-1}(\varphi^{-1}(\xi,t),t) = \dot F(x,t) \, F^{-1}(x,t) = L \, ,
	\end{align}
where we used that $\partial_t = \frac{\dif}{\dif t} = \frac{\DD}{\DD t}$ are all the same, if $x$ is fixed. \\
\\
As another example, when determining a corotational rate $\frac{\DD^{\circ}}{\DD t}$ we write
	\begin{align}
	\frac{\DD^{\circ}}{\DD t}[\sigma] = \frac{\DD}{\DD t}[\sigma] + \sigma \, \Omega^{\circ} - \Omega^{\circ} \, \sigma = \dot \sigma + \sigma \, \Omega^{\circ} - \Omega^{\circ} \, \sigma \, .
	\end{align}
However, if we solely work on the current configuration, i.e.~holding $\xi$ fixed, we write $\partial_t v$ for the time-derivative of the velocity (or any quantity in general). \\
\\
\noindent \textbf{Inner product} \\
\\
For $a,b\in\R^n$ we let $\langle {a},{b}\rangle_{\R^n}$  denote the scalar product on $\R^n$ with associated vector norm $\norm{a}_{\R^n}^2=\langle {a},{a}\rangle_{\R^n}$. We denote by $\R^{n\times n}$ the set of real $n\times n$ second order tensors, written with capital letters. The standard Euclidean scalar product on $\R^{n\times n}$ is given by
$\langle {X},{Y}\rangle_{\R^{n\times n}}=\tr{(X Y^T)}$, where the superscript $^T$ is used to denote transposition. Thus the Frobenius tensor norm is $\norm{X}^2=\langle {X},{X}\rangle_{\R^{n\times n}}$, where we usually omit the subscript $\R^{n\times n}$ in writing the Frobenius tensor norm. The identity tensor on $\R^{n\times n}$ will be denoted by $\id$, so that $\tr{(X)}=\langle {X},{\id}\rangle$. \\
\\
\noindent \textbf{Frequently used spaces} 
	\begin{itemize}
	\item $\Sym(n), \rm \Sym^+(n)$ and $\Sym^{++}(n)$ denote the symmetric, positive semi-definite symmetric and positive definite symmetric second order tensors respectively. Note that $\Sym^{++}(n)$ is considered herein only as an algebraic subset of $\Sym(n)$, not endowed with a Riemannian geometry \cite{fiala2009, fiala2016, fiala2020objective, kolev2024objective}.
	\item ${\rm GL}(n)\colonequals\{X\in\R^{n\times n}\;|\det{X}\neq 0\}$ denotes the general linear group.
	\item ${\rm GL}^+(n)\colonequals\{X\in\R^{n\times n}\;|\det{X}>0\}$ is the group of invertible matrices with positive determinant.
	\item ${\rm SL}(n)\colonequals\{X\in {\rm GL}(n)\;|\det{X}=1\}$.
	\item $\mathrm{O}(n)\colonequals\{X\in {\rm GL}(n)\;|\;X^TX=\id\}$.
	\item ${\rm SO}(n)\colonequals\{X\in {\rm GL}(n,\R)\;|\; X^T X=\id,\;\det{X}=1\}$.
	\item $\mathfrak{so}(3)\colonequals\{X\in\mathbb{R}^{3\times3}\;|\;X^T=-X\}$ is the Lie-algebra of skew symmetric tensors.
	\item $\mathfrak{sl}(3)\colonequals\{X\in\mathbb{R}^{3\times3}\;|\; \tr({X})=0\}$ is the Lie-algebra of traceless tensors.
	\item The set of positive real numbers is denoted by $\R_+\colonequals(0,\infty)$, while $\overline{\R}_+=\R_+\cup \{\infty\}$.
	\end{itemize}
\textbf{Frequently used tensors}
	\begin{itemize}
	\item $F = \DD \varphi(x,t)$ is the Fréchet derivative (Jacobean) of the deformation $\varphi(\,,t)\colon\Omega_x \to \Omega_{\xi} \subset \R^3$. $\varphi(x,t)$ is usually assumed to be a diffeomorphism at every time $t \ge 0$ so that the inverse mapping $\varphi^{-1}(\,,t)\colon\Omega_{\xi} \to \Omega_x$ exists.
	\item $C=F^T \, F$ is the right Cauchy-Green strain tensor.
	\item $B=F\, F^T$ is the left Cauchy-Green (or Finger) strain tensor.
	\item $U = \sqrt{F^T \, F} \in \Sym^{++}(3)$ is the right stretch tensor, i.e.~the unique element of ${\rm Sym}^{++}(3)$ with $U^2=C$.
	\item $V = \sqrt{F \, F^T} \in \Sym^{++}(3)$ is the left stretch tensor, i.e.~the unique element of ${\rm Sym}^{++}(3)$ with $V^2=B$.
	\item $\log V = \frac12 \, \log B$ is the spatial logarithmic strain tensor or Hencky strain.
	\item $L = \dot F \, F^{-1} = \DD_\xi v(\xi)$ is the spatial velocity gradient.
	\item $v = \frac{\DD}{\DD t} \varphi(x, t)$ denotes the Eulerian velocity.
	\item $D = \sym \, L$ is the spatial rate of deformation, the Eulerian strain rate tensor.
	\item $W = \sk \, L$ is the vorticity tensor.
	\item We also have the polar decomposition $F = R \, U = V R \in {\rm GL}^+(3)$ with an orthogonal matrix $R \in \OO(3)$ (cf. Neff et al.~\cite{Neffpolardecomp}), see also \cite{LankeitNeffNakatsukasa,Neff_Nagatsukasa_logpolar13}.
	\end{itemize}
\noindent \textbf{Frequently used rates}
	\begin{multicols}{2}
	\begin{itemize}
	\item $\dd \frac{\DD^{\sharp}}{\DD t}$ denotes an arbitrary objective derivative,
	\item $\dd \frac{\DD^{\circ}}{\DD t} \begin{array}{l} \text{denotes an arbitrary corotational} \\ \text{derivative,} \end{array}$
	\item $\dd \frac{\DD^{\ZJ}}{\DD t} \begin{array}{l} \text{denotes the Zaremba-Jaumann} \\ \text{derivative,} \end{array}$
	\item $\dd \frac{\DD^{\GN}}{\DD t}$ denotes the Green-Naghdi derivative.
	\item $\dd \frac{\DD^{\log}}{\DD t}$ denotes the logarithmic derivative.
	\item $\dd \frac{\DD}{\DD t}$ denotes the material derivative.
	\end{itemize}
	\end{multicols}
\noindent \textbf{Calculus with the material derivative -- some examples} \\
\\
Consider the spatial Cauchy stress
	\begin{align}
	\label{eqmat01}
	\sigma(\xi,t) \colonequals \Sigma(B) = \Sigma(F(x,t) \, F^T(x,t)) = \Sigma(F(\varphi^{-1}(\xi,t),t) \, F^T(\varphi^{-1}(\xi,t),t)) \, .
	\end{align}
Then, on the one hand we have for the material derivative
	\begin{align}
	\label{eqmat02}
	\frac{\DD}{\DD t}[\sigma] = \DD_\xi \sigma(\xi,t).v(\xi,t) + \partial_t \sigma(\xi,t) \cdot 1
	\end{align}
and on the other hand equivalently
	\begin{equation}
	\label{eqmat03}
	\begin{alignedat}{2}
	\frac{\DD}{\DD t}[\sigma] &= \frac{\DD}{\DD t}[\Sigma(F(x,t) \, F(x,t)^T)] \overset{(1)}{=} \frac{\dif}{\dif t}[\Sigma(F(x,t) \, F^T(x,t))] \\
	\overset{\substack{\text{standard} \\ \text{chain rule}}}&{=} \DD_B \Sigma(F(x,t) \, F^T(x,t)). \frac{\dif}{\dif t}[(F(x,t) \, F^T(x,t))] = \DD_B \Sigma(F(x,t) \, F^T(x,t)).(\dot F \, F^T + F \, \dot{F}^T) \\
	&= \DD_B \Sigma(F(x,t) \, F^T(x,t)).[\dot F \, F^{-1} \, F \, F^T + F \, F^T \, F^{-T} \, \dot{F}^T] =\DD_B \Sigma(F(x,t) \, F^T(x,t)).[L \, B + B \, L^T] \, .
	\end{alignedat}
	\end{equation}
In $\eqref{eqmat03}_{1}$ we have used the fact that there is already a material representation which allows to set $\frac{\DD}{\DD t} = \frac{\dif}{\dif t}$. Of course, \eqref{eqmat02} is equivalent to \eqref{eqmat03}. From the context it should be clear which representation of $\sigma$ (referential or spatial) we are working with and by abuse of notation we do not distinguish between $\sigma$ and $\Sigma$. \\
\\
The same must be observed when calculating with corotational derivatives
	\begin{align}
	\label{eqmat04}
	\frac{\DD^{\circ}}{\DD t}[\sigma] = \frac{\DD}{\DD t}[\sigma] + \sigma \, \Omega^{\circ} - \Omega^{\circ} \, \sigma, \qquad \Omega^{\circ} = \frac{\DD}{\DD t}Q^{\circ}(x,t) \, (Q^{\circ})^T(x,t) = \frac{\dif}{\dif t} Q^{\circ}(x,t) \, (Q^{\circ})^T(x,t) \, .
	\end{align}
Here, we have
	\begin{align}
	\label{eqmat05}
	\frac{\DD^{\circ}}{\DD t}[\sigma] \overset{(\ast \ast)}&{=} Q^{\circ}(x,t) \, \frac{\DD}{\DD t}[(Q^{\circ})^T(x,t) \, \sigma \, Q^{\circ}(x,t)] \, (Q^{\circ})^T(x,t) \\
	&=Q^{\circ}(x,t) \, \left\{\frac{\DD}{\DD t}(Q^{\circ})^T(x,t) \, \sigma \, Q^{\circ}(x,t) + (Q^{\circ})^T(x,t) \, \frac{\DD}{\DD t}[\sigma] \, Q^{\circ}(x,t) + (Q^{\circ})^T(x,t) \, \sigma \, \frac{\DD}{\DD t}Q^{\circ}(x,t) \right\} \, (Q^{\circ})^T(x,t) \notag \\
	&=Q^{\circ}(x,t) \, \bigg\{\frac{\dif}{\dif t}(Q^{\circ})^T(x,t) \, \sigma \, Q^{\circ}(x,t) + (Q^{\circ})^T(x,t) \, \underbrace{\frac{\DD}{\DD t}[\sigma]}_{(\ast \ast \ast)} \, Q^{\circ}(x,t) + (Q^{\circ})^T(x,t) \, \sigma \, \frac{\dif}{\dif t}Q^{\circ}(x,t) \bigg\} \, (Q^{\circ})^T(x,t) \notag 
	\end{align}
and we can decide for $(\ast \ast \ast)$ to continue the calculus with \eqref{eqmat02} or \eqref{eqmat03}. In either case one has to decide viewing the functions as defined on the reference configuration $\Omega$ or in the spatial configuration $\Omega_\xi$.

In \eqref{eqmat05} we used $Q = Q(x,t)$ and $\Omega = \Omega(x,t)$. This means that the ``Lie-type'' representation $(\ast \, \ast)$ necessitates the definition of a reference configuration, so that we can switch between $\xi = \varphi(x,t)$ and $x$.

The interpretation $(\ast \, \ast)$ is most clearly represented for the Green-Naghdi rate, in which the spin \break $\Omega^{\GN} \colonequals \frac{\dif}{\dif t}R(x,t) \, R^T(x,t) = \dot R(x,t) \, R^T(x,t)$ is defined according to the polar decomposition $F = R \, U$ and in
	\begin{align}
	\frac{\DD^{\GN}}{\DD t}[\sigma] = \frac{\DD}{\DD t}[\sigma] + \sigma \, \Omega^{\GN} - \Omega^{\GN} \, \sigma = R \, \frac{\DD}{\DD t}[R^T \, \sigma \, R] \, R^T
	\end{align}
the term $[R^T \, \sigma \, R]$ is called \emph{corotational stress tensor} (cf.~\cite[p.~142]{Marsden83}). \\
\\
\noindent \textbf{Tensor domains} \\
\\
Denoting the reference configuration by $\Omega_x$ with tangential space $T_x \Omega_x$ and the current/spatial configuration by $\Omega_\xi$ with tangential space $T_\xi \Omega_\xi$ as well as $\varphi(x) = \xi$, we have the following relations (see also Figure \ref{yfig2}):

	\begin{figure}[h!]
		\begin{center}		
		\begin{minipage}[h!]{0.8\linewidth}
			\centering
			\hspace*{-80pt}
			\includegraphics[scale=0.5]{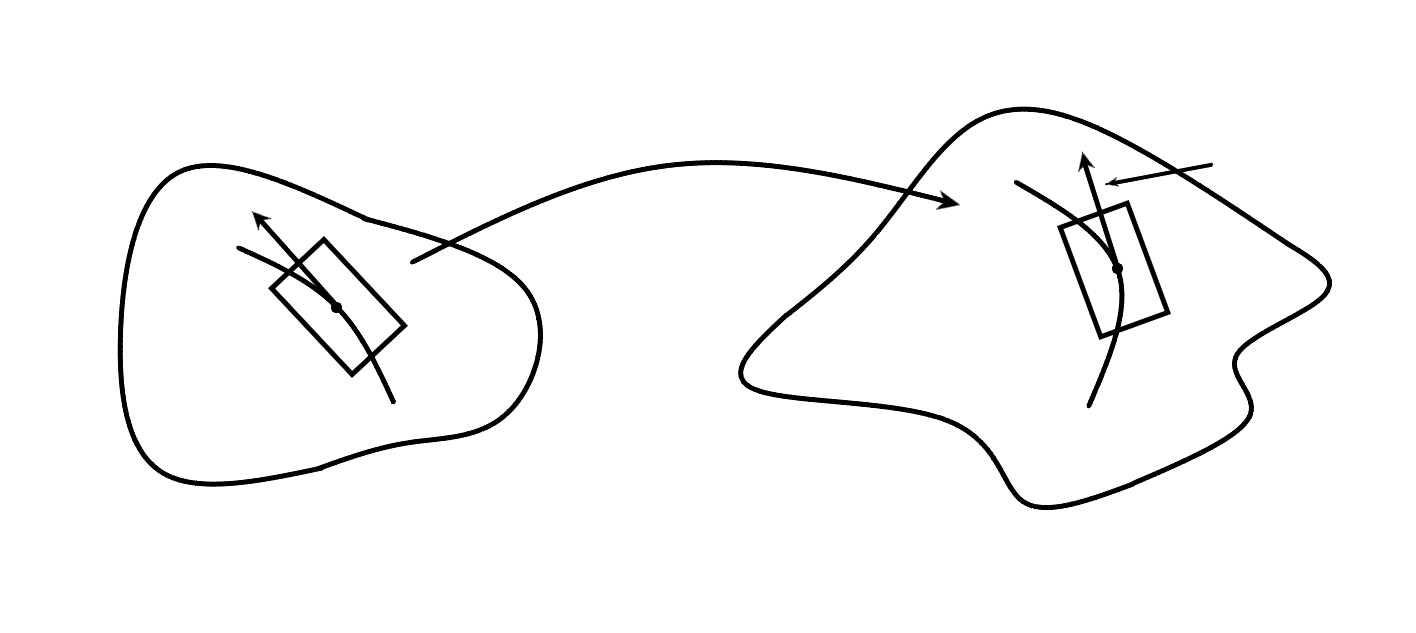}
			\put(-100,45){\footnotesize $\Omega_\xi$}
			\put(-390,55){\footnotesize $\Omega_x$}
			\put(-412,115){\footnotesize $x$}
			\put(-430,155){\footnotesize $\dot \gamma(0)$}
			\put(-377,105){\footnotesize $T_x \Omega_x$}
			\put(-383,85){\footnotesize $\gamma(s)$}
			\put(-280,183){\footnotesize $\varphi(x,t_0)$}
			\put(-120,131){\footnotesize $\xi$}
			\put(-78,181){\footnotesize $\frac{\dif}{\dif s}\varphi(\gamma(s),t_0)\bigg\vert_{s=0}$}
			\put(-88,119){\footnotesize $T_\xi \Omega_\xi$}
			\put(-115,88){\footnotesize $\varphi(\gamma(s),t_0)$}
		\end{minipage} 
		\caption{Illustration of the curve $s \mapsto \varphi(\gamma(s),t_0), \; \gamma(0) = x$ for a fixed time $t = t_0$ with vector field \break $s \mapsto \frac{\dif}{\dif s} \varphi(\gamma(s),t) \in T_\xi \Omega_\xi$.}
		\label{yfig2}
		\end{center}
	\end{figure}
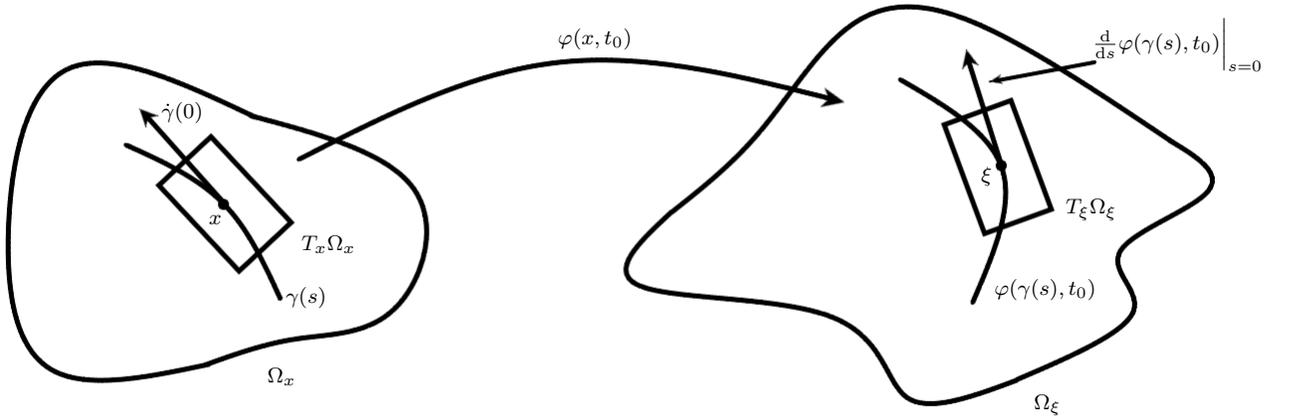
	\begin{multicols}{2}
	\begin{itemize}
	\item $F\colon T_x \Omega_x \to T_\xi \Omega_\xi$
	\item $Q\colon T_x \Omega_x \to T_\xi \Omega \xi$
	\item $F^T\colon T_\xi \Omega_\xi \to T_x \Omega_x$
	\item $Q^T\colon T_\xi \Omega_\xi \to T_x \Omega_x$
	\item $C = F^T \, F\colon T_x \Omega_x \to T_x \Omega_x$
	\item $B = F \, F^T\colon T_\xi \Omega_\xi \to T_\xi \Omega_\xi$
	\item $\sigma\colon T_\xi \Omega_\xi \to T_\xi \Omega_\xi$
	\item $\tau\colon T_\xi \Omega_\xi \to T_\xi \Omega_\xi$
	\item $S_2 \colon T_x \Omega_x \to T_x \Omega_x$
	\item $S_1 \colon T_x \Omega_x \to T_\xi \Omega_\xi$
	\item $Q^T \, \sigma \, Q\colon T_x \Omega_x \to T_x \Omega_x$
	\end{itemize}
	\end{multicols}
\noindent \textbf{The strain energy function $\WW(F)$} \\
\\
We are only concerned with rotationally symmetric functions $\WW(F)$ (objective and isotropic), i.e.
	\begin{equation*}
	\WW(F)={\WW}(Q_1^T\, F\, Q_2), \qquad \forall \, F \in {\rm GL}^+(3), \qquad  Q_1 , Q_2 \in {\rm SO}(3).
	\end{equation*}
\textbf{Primary matrix functions} \\
\\
We define primary matrix functions as those functions $\Sigma \colon \Sym^{++}(3) \to \Sym(3)$, such that
	\begin{align}
	\Sigma(V) = \Sigma(Q^T \, \textnormal{diag}_V(\lambda_1, \lambda_2, \lambda_3) \, Q) = Q^T \Sigma(\textnormal{diag}_V(\lambda_1, \lambda_2, \lambda_3)) \, Q = Q^T \,
		\begin{pmatrix}
		f(\lambda_1) & 0 & 0 \\
		0 & f(\lambda_2) & 0 \\
		0 & 0 & f(\lambda_3)
		\end{pmatrix} \, Q
	\end{align}
with one given real-valued scale-function $f \colon \mathbb{R}_+ \to \mathbb{R}$. Any primary matrix function is an isotropic matrix function but not vice-versa as shows e.g.~$\Sigma(V) = \det V \, \id$. \\
\\
\textbf{List of additional definitions and useful identities}
	\begin{itemize}
	\item For two metric spaces $X, Y$ and a linear map $L: X \to Y$ with argument $v \in X$ we write $L.v\colonequals L(v)$. This applies to a second order tensor $A$ and a vector $v$ as $A.v$ as well as a fourth order tensor $\C$ and a second order tensor $H$ as $\C.H$. Sometimes we may emphasize the usual matrix product of two second order tensors $A, B$ as $A \cdot B$.
	\item We define $J = \det{F}$ and denote by $\Cof(X) = (\det X)X^{-T}$ the cofactor of a matrix in ${\rm GL}^{+}(3)$.
	\item We define $\sym X = \frac12 \, (X + X^T)$ and $\sk X = \frac12 \, (X - X^T)$ as well as $\dev X = X - \frac13 \, \tr(X) \, \id$.
	\item For all vectors $\xi,\eta\in\R^3$ we have the tensor product $(\xi\otimes\eta)_{ij}=\xi_i\,\eta_j$.
	\item $S_1=\DD_F \WW(F) = \sigma \, \Cof F$ is the non-symmetric first Piola-Kirchhoff stress tensor.
	\item $S_2=F^{-1}S_1=2\,\DD_C \widetilde{\WW}(C)$ is the symmetric second  Piola-Kirchhoff stress tensor.
	\item $\sigma=\frac{1}{J}\,  S_1\, F^T=\frac{1}{J}\,  F\,S_2\, F^T=\frac{2}{J}\DD_B \widetilde{\WW}(B)\, B=\frac{1}{J}\DD_V \widetilde{\WW}(V)\, V = \frac{1}{J} \, \DD_{\log V} \widehat \WW(\log V)$ is the symmetric Cauchy stress tensor.
	\item $\sigma = \frac{1}{J} \, F\, S_2 \, F^T = \frac{2}{J} \, F \, \DD_C \widetilde{\WW}(C) \, F^T$ is the ``\emph{Doyle-Ericksen formula}'' \cite{doyle1956}.
	\item For $\sigma\colon \Sym(3) \to \Sym(3)$ we denote by $\DD_B \sigma(B)$ with $\sigma(B+H) = \sigma(B) + \DD_B \sigma(B).H + o(H)$ the Fréchet-derivative. For $\sigma\colon \Sym^+(3) \subset \Sym(3) \to \Sym(3)$ the same applies. Similarly, for $\WW\colon\R^{3 \times 3} \to \R$ we have $\WW(X + H) = \WW(X) + \langle \DD_X \WW(X), H \rangle + o(H)$.
	\item $\tau = J \, \sigma = 2\, \DD_B \widetilde{\WW}(B)\, B $ is the symmetric Kirchhoff stress tensor.
	\item $\tau = \DD_{\log V} \widehat{\WW}(\log V)$ is the ``\emph{Richter-formula}'' \cite{richter1948isotrope, richter1949hauptaufsatze}.
	\item $\sigma_i =\dd\frac{1}{\lambda_1\lambda_2\lambda_3}\dd\lambda_i\frac{\partial g(\lambda_1,\lambda_2,\lambda_3)}{\partial \lambda_i}=\dd\frac{1}{\lambda_j\lambda_k}\dd\frac{\partial g(\lambda_1,\lambda_2,\lambda_3)}{\partial \lambda_i}, \ \ i\neq j\neq k \neq i$ are the principal Cauchy stresses (the eigenvalues of the Cauchy stress tensor $\sigma$), where $g:\mathbb{R}_+^3\to \mathbb{R}$ is the unique function  of the singular values of $U$ (the principal stretches) such that $\WW(F)=\widetilde{\WW}(U)=g(\lambda_1,\lambda_2,\lambda_3)$.
	\item $\sigma_i =\dd\frac{1}{\lambda_1\lambda_2\lambda_3}\frac{\partial \widehat{g}(\log \lambda_1,\log \lambda_2,\log \lambda_3)}{\partial \log \lambda_i}$, where $\widehat{g}:\mathbb{R}^3\to \mathbb{R}$ is the unique function such that \\ \hspace*{0.3cm} $\widehat{g}(\log \lambda_1,\log \lambda_2,\log \lambda_3)\colonequals g(\lambda_1,\lambda_2,\lambda_3)$.
	\item $\tau_i =J\, \sigma_i=\dd\lambda_i\frac{\partial g(\lambda_1,\lambda_2,\lambda_3)}{\partial \lambda_i}=\frac{\partial \widehat{g}(\log \lambda_1,\log \lambda_2,\log \lambda_3)}{\partial \log \lambda_i}$ \, . 
	\end{itemize}

\vspace*{2em}
\noindent \textbf{Conventions for fourth order symmetric operators, minor and major symmetry} \\
\\
Fourth order tensors are written as $\H$ or $\C$. For a fourth order linear mapping $\C\colon\Sym(3) \to \Sym(3)$ we agree on the following convention. \\
\\
We say that $\C$ has \emph{minor symmetry} if
	\begin{align}
	\C.S \in \Sym(3) \qquad \forall \, S \in \Sym(3).
	\end{align}
This can also be written in index notation as $C_{ijkm} = C_{jikm} = C_{ijmk}$. If we consider a more general fourth order tensor $\C\colon\R^{3 \times 3} \to \R^{3 \times 3}$ then $\C$ can be transformed having minor symmetry by considering the mapping $X \mapsto \sym(\C. \sym X)$ such that $\C: \R^{3 \times 3} \to \R^{3 \times 3}$ is minor symmetric, if and only if $\C.X = \sym(\C.\sym X)$. \\
\\
We say that $\C$ has \emph{major symmetry} (or is \emph{self-adjoint}, respectively) if
	\begin{align}
	\langle \C. S_1, S_2 \rangle = \langle \C. S_2, S_1 \rangle \qquad \forall \, S_1, S_2 \in \Sym(3).
	\end{align}
Major symmetry in index notation is understood as $C_{ijkm} = C_{kmij}$. \\
\\
The set of positive definite, major symmetric fourth order tensors mapping $\R^{3 \times 3} \to \R^{3 \times 3}$ is denoted as $\Sym^{++}_4(9)$, in case of additional minor symmetry, i.e.~mapping $\Sym(3) \to \Sym(3)$ as $\Sym^{++}_4(6)$. By identifying $\Sym(3) \cong \R^6$, we can view $\C$ as a linear mapping in matrix form $\widetilde \C: \R^6 \to \R^6$. \newline If $H \in \Sym(3) \cong \R^6$ has the entries $H_{ij}$, we can write
	\begin{align}
	\label{eqvec1}
	h = \textnormal{vec}(H) = (H_{11}, H_{22}, H_{33}, H_{12}, H_{23}, H_{31}) \in \R^6 \qquad \textnormal{so that} \qquad \langle \C.H, H \rangle_{\Sym(3)} = \langle \widetilde \C.h, h \rangle_{\R^6}.
	\end{align}
If $\C: \Sym(3) \to \Sym(3)$, we can define $\bfsym \C$ by
	\begin{align}
	\langle \C.H, H \rangle_{\Sym(3)} = \langle \widetilde \C.h, h \rangle_{\R^6} = \langle \sym \widetilde \C. h, h \rangle_{\R^6} \equalscolon \langle \bfsym \C.H, H \rangle_{\Sym(3)}, \qquad \forall \, H \in \Sym(3).
	\end{align}
Major symmetry in these terms can be expressed as $\widetilde \C \in \Sym(6)$. \emph{In this text, however, we omit the tilde-operation and ${\bf sym}$ and write in short $\sym\C\in {\rm Sym}_4(6)$ if no confusion can arise.} In the same manner we speak about $\det \C$ meaning $\det \widetilde \C$. \\
\\
A linear mapping $\C\colon\R^{3 \times 3} \to \R^{3 \times 3}$ is positive definite if and only if
	\begin{align}
	\label{eqposdef1}
	\langle \C.H, H \rangle > 0 \qquad \forall \, H \in \R^{3 \times 3} \qquad \iff \qquad \C \in \Sym^{++}_4(9)
	\end{align}
and analogously it is positive semi-definite if and only if
	\begin{align}
	\label{eqpossemidef1}
	\langle \C.H, H \rangle \ge 0 \qquad \forall \, H \in \R^{3 \times 3} \qquad \iff \qquad \C \in \Sym^+_4(9).
	\end{align}
For $\C: \Sym(3) \to \Sym(3)$, after identifying $\Sym(3) \cong \R^6$, we can reformulate \eqref{eqposdef1} as $\widetilde \C \in \Sym^{++}(6)$ and \eqref{eqpossemidef1} as $\widetilde \C \in \Sym^+(6)$. 
\subsection{Hilbert-monotonicity }\label{ips2}
Regarding Hilbert-monotonicity, we recall the following properties from Ghiba et al.~\cite{ghiba2024}.
\begin{definition} \cite{NeffMartin14}
A tensor function $\Sigma_f:{\rm Sym}^{++}(3)\to\Sym^{++}\,$ is called  \emph{strictly Hilbert-monotone} if
	\begin{align}
	\label{eq:introductionMatrixMonotonicity}
	\iprod{\Sigma_f(U)-\Sigma_f(\overline U),\,U-\overline U}_{\R^{3\times 3}}>0\qquad\forall\,U\neq\overline U\in{\rm Sym}^{++}(3)\,.
	\end{align}
We refer to this inequality as strict \emph{Hilbert-space matrix-monotonicity} of the tensor function $\Sigma_f$. \\
Similarly, $\Sigma_f(U)$ is \emph{strongly Hilbert-monotone} if for all $H \in \Sym(3) \! \setminus \! \{0\}$ we have $\langle \DD_U \Sigma_f(U).H, H \rangle > 0$ or equivalently $\sym \DD_U \Sigma_f(U) \in \Sym^{++}_4(6)$.
\end{definition}
\begin{definition}\cite{NeffMartin14}
A vector function  $f:\R^3_+ \colonequals \R_+ \times \R_+ \times \R_+ \to\R^3$ is \emph{strictly vector monotone} if 
	\begin{align}
	\iprod{f(\lambda)-f(\overline\lambda),\,\lambda-\overline\lambda}_{\R^3}>0\qquad\forall\lambda\neq\overline\lambda\in\R^3_+.
	\end{align}
\end{definition}
\begin{definition}
A differentiable function $f$ between finite-dimensional Hilbert spaces is called \emph{strongly monotone} if $\DD f$ is positive definite everywhere. Thus ``stronlgy'' implies ``strictly'', see e.g.~Remark \ref{remarkmon}.
\end{definition}
\noindent Note that for an arbitrary vector function $f: \R^3_+ \to \R^3$, ${\rm D}f\,(\lambda_1,\lambda_2,\lambda_3)$ in itself might not be symmetric. One main goal of a forthcoming paper \cite{MartinVossGhibaNeff}  is to proof the following result, thereby elucidating on Ogden's work \cite[last page in Appendix]{Ogden83}, based on the seminal contributions of Hill \cite{hill1968constitutivea,hill1968constitutiveb,hill1970constitutive}: 
\begin{thm}
	\label{theorem:mainResult}
	A symmetric function $f:\R^3_+\to\R^3$ is strictly (strongly) vector-monotone if and only if $\Sigma_f$ is strictly (strongly) matrix-monotone.
\end{thm}
\begin{rem} \label{remA5}
Theorem \ref{theorem:mainResult} is decisive for the equivalence
	\begin{align}
	\sym \DD_{\log V} \widehat \sigma(\log V) \in \Sym^{++}_4(6) \qquad \iff \qquad \sym \frac{\partial \widehat \sigma_i}{\partial \log \lambda_j} \in \Sym^{++}(3),
	\end{align}
where the $\widehat \sigma_i$ are the principal Cauchy stresses expressed as function of the principle logarithmic strains $\log \lambda_i$.
\end{rem}
\subsection{True-Stress-True-Strain monotonicity (TSTS-M)} \label{appendixtstsm}
\begin{definition}
We define three notions of \emph{True-Stress-True-Strain monotonicity} as follows
	\begin{equation}
	\begin{alignedat}{2}
	&\textnormal{TSTS-M:} \qquad &\langle \widehat \sigma(\log V_1)- \widehat \sigma(\log V_2),\log V_1-\log V_2\rangle&\ge 0, \qquad \forall\, V_1, V_2\in {\rm Sym}^{++}(3), \ V_1\neq V_2, \\
	&\textnormal{TSTS-M$^+$:} \qquad &\langle \widehat \sigma(\log V_1)- \widehat \sigma(\log V_2),\log V_1-\log V_2\rangle&> 0, \qquad \forall\, V_1, V_2\in {\rm Sym}^{++}(3), \ V_1\neq V_2, \\
	&\textnormal{TSTS-M$^{++}$:} \qquad & \sym \, \DD_{\log V} \widehat \sigma(\log V) \in \Sym^{++}_4(6). &
	\end{alignedat}
	\end{equation}
Note that this is equivalent to monotonicity of $\widehat \sigma$ in $\log B$ since $\log B=2\, \log V$.
\end{definition}
\noindent Regarding TSTS-M$^+$, we have the following properties from \cite{NeffGhibaLankeit}.
\begin{rem}\label{remarkmon}
Sufficient for TSTS-M$^+$ is Jog and Patil's \cite{jog2013conditions} constitutive requirement that
	\begin{align}
	\textnormal{TSTS-M$^{++}$:} \qquad \Lambda \colonequals \sym \, \DD_{\log V}\,\widehat \sigma(\log V) \in \Sym^{++}_4(6),
	\end{align}
i.e.~in their notation (see also \cite[Remark 4.1]{NeffGhibaLankeit})
	\begin{align}
	\mathbb{Z} \colonequals \DD_{\log B} \widehat \sigma(\log B) \qquad \textnormal{with} \qquad \langle \mathbb{Z}. H, H \rangle > 0 \qquad \forall \, H \in \Sym(3) \! \setminus \! \{0\}.
	\end{align}
\end{rem}
\begin{proof}
Let us remark that for all $B_1, B_2\in {\rm Sym}^{++}(3)$ and $0\leq t\leq 1$, we have $2\,\log V_1=\log B_1, \, 2\,\log V_2=\log B_2$ and $t\, (\log
V_1-\log V_2)+\log V_2\in{\rm Sym}(3)$, where $V_1^2=B_1,\, V_2^2=B_2$ . Moreover, we have
	 \begin{align}\label{Joginespr}
	  \langle \widehat \sigma(\log B_1)&- \widehat \sigma(\log B_2),\log B_1-\log B_2\rangle=2\,\langle \widehat \sigma(2\,\log V_1)- \widehat \sigma(2\,\log V_2),\log V_1-\log
	  V_2\rangle\notag\\&=2\,\left\langle\left[\int_0^1 \frac{\rm d}{\rm dt}\, \widehat \sigma \bigg(2\,t\, (\log V_1-\log V_2)+2\,\log V_2\bigg) \dif t\right],\log V_1-\log
	  V_2\right\rangle\\&=4\,\int_0^1 \left\langle\left[\DD_{\log V}\, \widehat \sigma \bigg(2\,t\, (\log V_1-\log V_2)+2\,\log V_2\bigg).\,(\log V_1-\log V_2)\right],\log
	  V_1-\log V_2\right\rangle \dif t\,.\notag \\
	  &= 4\,\int_0^1 \left\langle\left[\sym \left( \DD_{\log V}\, \widehat \sigma \bigg(2\,t\, (\log V_1-\log V_2)+2\,\log V_2\bigg) \right) .\,(\log V_1-\log V_2)\right],\log
	  V_1-\log V_2\right\rangle \dif t\,.\notag
	 \end{align}
Where the last equation of \eqref{Joginespr} is due to the fact that 
for any skew symmetric matrix $A\in\mathfrak{so}(3)$, it always holds
$
\scal{A.v}{v}_{\bR^{3}}
=0
$
for every $v\in\bR^3$. Using that the integrand is non-negative, due to the assumption that $\Lambda = \sym \, \DD_{\log V} \widehat \sigma(\log V)$ is positive definite, the TSTS-M$^+$ condition
 follows.
\end{proof}
\begin{rem}
As an easy consequence of the previous remark we obtain the implications
	\begin{align}
	\textnormal{TSTS-M$^{++}$} \qquad \implies \qquad \textnormal{TSTS-M$^{+}$} \qquad \implies \qquad \textnormal{TSTS-M}
	\end{align}
as well as the equivalence \quad 
	\fbox{
	\begin{minipage}[h!]{0.65\linewidth}
		\centering
		TSTS-M$^{++}$ \qquad $\iff$ \qquad corotational stability postulate (CSP).
	\end{minipage}}
\end{rem}
\subsubsection{TSTS-M$^{++}$ for the exponentiated Hencky energy $V\mapsto \frac{\mu}{k}\,e^{k\,\|\log\,V\|^2}+\frac{\lambda}{2\widehat{k}}\,e^{\widehat{k}\,[{\rm tr}
(\log\,V)]^2}$}
Since our examples in the main part of this work regarding monotonicity of the Cauchy stress $\sigma$ as a function of the logarithmic strain $\log V$ are all \emph{not hyperelastic}, we show in this Appendix that this monotonicity requirement is nevertheless in principle consistent with hyperelasticity. Note, however, that the classical compressible Neo-Hooke or compressible Mooney-Rivlin models do not satisfy TSTS-M (example given in Appendix~\ref{appendixneohooke}). For incompressible response, the situation is different. As further references for the exponentiated Hencky energy see also \cite{montella2016, nedjar2018, xiao2002}.
\begin{prop}
 The Cauchy stress tensor $\sigma$ corresponding to the energy $V\mapsto \frac{\mu}{k}\,e^{k\,\|\log\,V\|^2}$ satisfies TSTS-M  for $k\geq \frac{3}{8}$ and
 TSTS-M$^{++}$  for $k> \frac{3}{8}$.
 \end{prop}
 \begin{proof}
 In order to show this, let us remark that  for the energy $V\mapsto \frac{\mu}{k}\,e^{k\,\|\log\,V\|^2}$ we have
	 \begin{align}\label{eqsigmataut}
	 \widehat{\tau}(\log\,V)&=\,2\,{\mu}\,e^{k\,\|\log\,V\|^2}\, \log\,V,\qquad
	  \widehat{\sigma}(\log\,V)=\,2\,{\mu}\,e^{k\,\|\log\,V\|^2-\tr(\log  V)}\, \log\,V.
	 \end{align}
  We compute
	 \begin{align}
	 \langle \DD_{X}\widehat{\sigma}(X).\,H,H\rangle=&\,2\,{\mu}\,e^{k\,\| X\|^2-\tr(X)}[2k\langle  X,H\rangle-\tr(H)]\langle
	 X,H\rangle\notag+2\,{\mu}\,e^{k\,\|X\|^2-\tr(X)}\| H\|^2\notag\\
	 =&\,2\,{\mu}\,e^{k\,\| X\|^2-\tr(X)}\{2\,k\,\langle  X,H\rangle^2-\tr(H)\langle  X,H\rangle+\| H\|^2\}.
	\end{align}
If $\tr(H)\langle  X,H\rangle<0$, then obviously $\langle \DD_{X}\widehat{\sigma}(X).\,H,H\rangle>0$. Otherwise, for $k\geq \frac{3}{8}$ it follows
	\begin{align}
	 \langle \DD_{X}\widehat{\sigma}(X).\,H,H\rangle\geq &\,2\,{\mu}\,e^{k\,\| X\|^2-\tr(X)}\{2k\langle  X,H\rangle^2-2\,\sqrt{\frac{2k}{3}}\tr(H)\langle
	 X,H\rangle+\| H\|^2\}\\
	 =&\,2\,{\mu}\,e^{k\,\| X\|^2-\tr(X)}\langle H-\sqrt{\frac{2\,k}{3}}\langle  X,H\rangle\, \id,H-\sqrt{\frac{2k}{3}}\langle  X,H\rangle\,
	 \id\rangle\notag\\
	 =&\,2\,{\mu}\,e^{k\,\| X\|^2-\tr(X)}\left\| H-\sqrt{\frac{2\,k}{3}}\langle  X,H\rangle\, \id\right\|^2\geq 0\notag.
	 \end{align}
Moreover, for $k> \frac{3}{8}$ we have $\langle \DD_{X}\widehat{\sigma}(X).\,H,H\rangle>0$ and the proof is complete.
 \end{proof}
 \begin{cor}
 The Cauchy stress tensor corresponding to the energy $V\mapsto \frac{\mu}{k}\,e^{k\,\|\log\,V\|^2}+\frac{\lambda}{2\widehat{k}}\,e^{\widehat{k}\,[{\rm tr}
 (\log\,V)]^2}$ satisfies TSTS-M  for $k\geq \frac{3}{8}$, $\widehat{k}\geq \frac{1}{8}$ and $\mu,\lambda>0$ and TSTS-M$^{++}$ for $k> \frac{3}{8}$,
 $\widehat{k}\geq \frac{1}{8}$ (or $k\geq \frac{3}{8}$, $\widehat{k}> \frac{1}{8}$) and $\mu,\lambda>0$.
 \end{cor}
 \begin{proof}
 From direct calculations we have
	 \begin{align}
	 \langle \DD_{X}e^{\widehat{k}\,(\tr (X))^2-\tr(X)}\, \tr(X)\, \id.\,H,H\rangle=e^{\widehat{k}\,(\tr (X))^2-\tr(X)} \{2\,\widehat{k}\,
	 [\tr(X)]^2-\tr(X)+1\}\,[\tr(H)]^2.
	 \end{align}
 Thus, if $\widehat{k}\geq \frac{1}{8}$, then
	  \begin{align}
	 \langle \DD_{X}e^{\widehat{k}\,(\tr (X))^2-\tr(X)}\, \tr(X)\, \id.\,H,H\rangle\geq e^{\widehat{k}\,(\tr (X))^2-\tr(X)}
	 \left(\frac{1}{2}\,\tr(X)-1\right)^2[\tr(H)]^2\geq 0.
	 \end{align}
The above inequality is strict for $\widehat{k}> \frac{1}{8}$.  The rest of the  proof follows from the previous theorem.
 \end{proof}
\noindent \textbf{Conclusion.} For the exponentiated Hencky energy $V\mapsto \frac{\mu}{k}\,e^{k\,\|\log\,V\|^2}+\frac{\lambda}{2\widehat{k}}\,e^{\widehat{k}\,[{\rm tr}
(\log\,V)]^2}$ our results are applicable, i.e.~we have corotational stability
	\begin{align}
	\sym \, \H^{\ZJ}_{\exp-\textnormal{Hencky}}(\sigma) \in \Sym^{++}_4(6) \qquad \iff \qquad \langle \H^{\ZJ}_{\exp-\textnormal{Hencky}}(\sigma). D , D \rangle > 0 \, .
	\end{align}
However, the exponentiated Hencky energy is in general \emph{not polyconvex} \cite{Ball77} and \emph{not LH-elliptic}. Hence it remains a \emph{major open question} (cf.~\cite{martin2017}) to 
\begin{center}
	\fbox{
	\begin{minipage}[h!]{1\linewidth}
		\centering
		\textbf{find a polyconvex (or LH-elliptic), isotropic and objective elastic energy such that the corotational stability postulate CSP is satisfied everywhere for the induced tangent stiffness tensor $\H^{\ZJ}(\sigma)$.}
	\end{minipage}}
	\end{center}

\subsection{Hypoelasticity for a slightly compressible Neo-Hooke type model} \label{appendixneohooke}
Consider a \textbf{polyconvex} slightly compressible Neo-Hooke type solid with elastic energy in volumetric-isochoric decoupled form (the vol-iso split also goes back to Richter \cite{richter1948isotrope, richter1949hauptaufsatze, Richter50, Richter52} and not to Flory \cite{flory1961}), i.e.
	\begin{equation}
	\label{eqappendixneohooke01}
	\begin{alignedat}{2}
	\WW_{\NH}(F) &= \frac{\mu}{2} \, \left(\frac{\norm{F}^2}{(\det F)^{\frac23}} - 3\right) + \kappa \, \mathrm{e}^{(\log \det F)^2}, \\
	\sigma_{\NH}(B) &= \mu \, (\det B)^{-\frac56} \, \dev_3 B + \kappa \, (\det B)^{-\frac12} \, (\log \det B) \, \mathrm{e}^{\frac14 \, (\log \det B)^2} \, \id,
	\end{alignedat}
	\end{equation}
with the shear modulus $\mu > 0$, the bulk modulus $\kappa$ and $\dev_3 X = X - \frac13 \tr(X) \, \id$. Then the constitutive law for $B \mapsto \sigma_{\NH}(B)$ is invertible, i.e.~there is a function $\mathcal{F}^{-1}: \Sym(3) \to \Sym^{++}(3), \; \mathcal{F}^{-1}(\sigma_{\NH}) = B$. For the proof of this statement we use the following
	\begin{lem}
	\label{psystin3}
	Let $S$ be a symmetric and trace free matrix, and $a>0$, then the system
		\begin{align}
		\label{systin2}
		\dev_n B=S,\qquad \det B=a,
		\end{align}
	has a unique solution $B\in {\rm Sym}^{++}(n)$ for $n=2,3$.
	\end{lem} 
	\begin{proof}
	The proof will be given in \cite{MartinVossGhibaNeff}.
	\end{proof}
\begin{figure}[h!]
\begin{center}
\begin{minipage}[h!]{0.95\linewidth}
\centering
\includegraphics[scale=0.25]{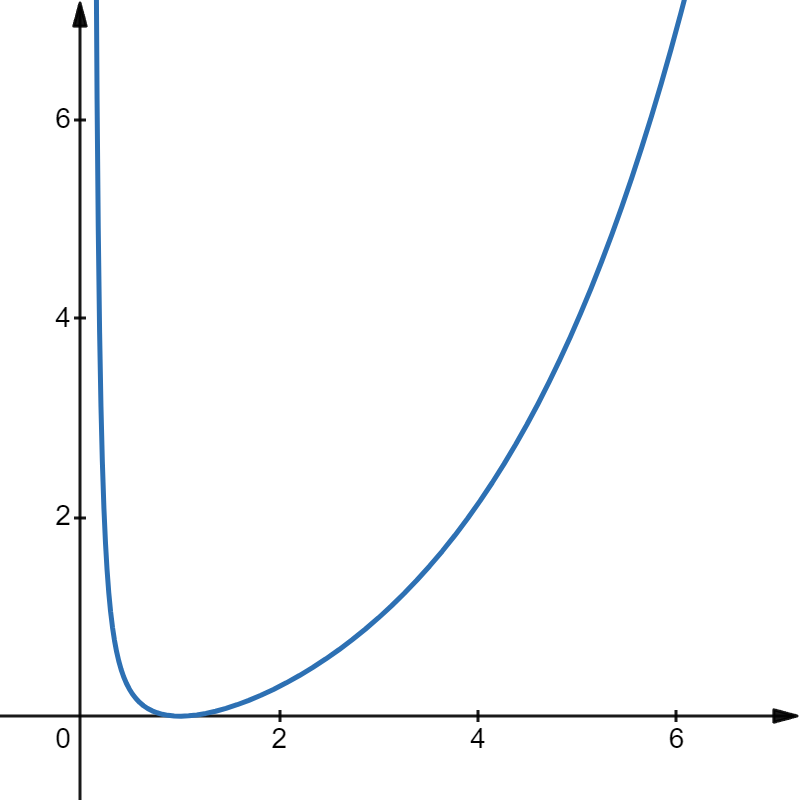}
\put(-5,10){\footnotesize{$\lambda$}}
\put(-220,187){\footnotesize{$W_{\NH}(\lambda)$}}
\qquad
\includegraphics[scale=0.25]{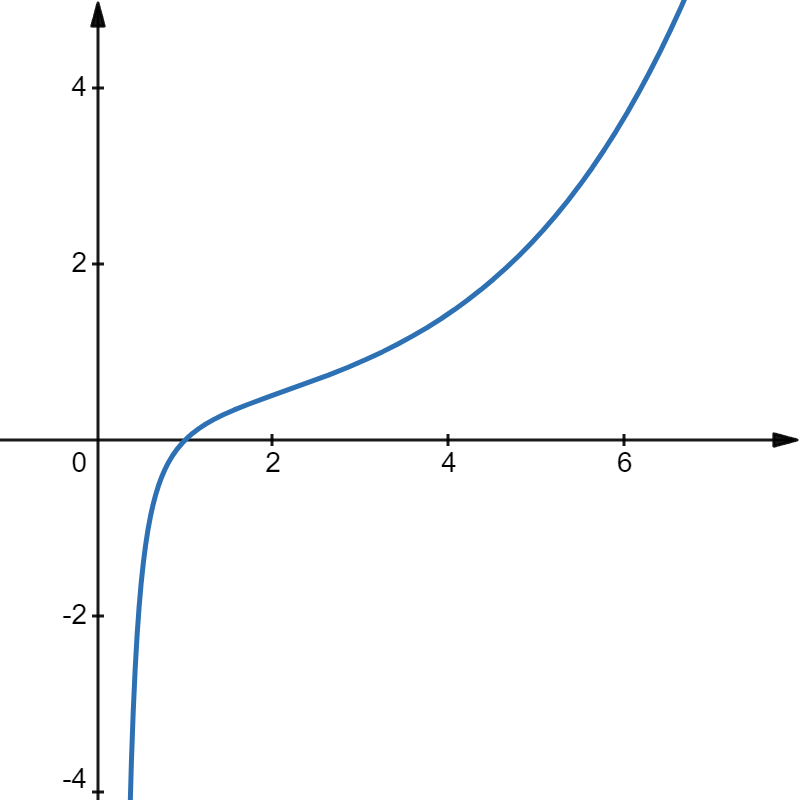}
\put(-7,80){\footnotesize{$\lambda$}}
\put(-200,187){\footnotesize{$\sigma(\lambda)$}}
\caption{Picture of the one-dimensional slightly compressible Neo-Hooke type energy \break $W_{\NH}(\lambda) = \frac{9}{42} (\lambda^{\frac43} + 2 \, \lambda^{-\frac23} + \mathrm{e}^{(\log \lambda)^2}-4)$ and the one-dimensional invertible Neo-Hooke Cauchy stress $\sigma_{\NH}(\lambda) = \frac{\dif}{\dif \lambda} W_{\NH}(\lambda)$, obtained from \eqref{eqappendixneohooke01}, by setting $F = \diag(\lambda, 1, 1) \implies B = \diag(\lambda^2, 1, 1)$.}
\label{fig10}
\end{minipage}
\end{center}
\end{figure}
\noindent Next, we write
	\begin{equation}
	\begin{alignedat}{2}
	\sigma_{\NH}(B) &= \underbrace{\mu \, (\det B)^{-\frac56} \, \dev_3 B}_{I_B} + \underbrace{\kappa \, (\det B)^{-\frac12} \, (\log \det B) \, \mathrm{e}^{\frac14 \, (\log \det B)^2}}_{II_B}  \, \id ,\\
	\sigma_{\NH}(B) &= \underbrace{\dev_3 X}_{I_X} + \underbrace{\frac13 \, \tr(X)}_{II_X} \, \id
	\end{alignedat}
	\end{equation}
and compare the components
	\begin{equation}
	\label{eqcomponents01}
	\begin{alignedat}{3}
	I: &&\qquad \mu \, (\det B)^{-\frac56} \, \dev_3 B &= \dev_3 X \, ,\\
	II: &&\qquad \kappa \, (\det B)^{-\frac12} \, (\log \det B) \, \mathrm{e}^{\frac14 \, (\log \det B)^2} &= \frac13 \, \tr(X) \, .
	\end{alignedat}
	\end{equation}
Defining the function
	\begin{align}
	f: \R^+ \to \R, \quad f(t) = t^{-\frac12} \, (\log t) \, \mathrm{e}^{\frac14 \, (\log t)^2} \quad \textnormal{with} \quad f'(t) = \frac12 \, t^{-\frac32} \, \mathrm{e}^{\frac14 \, (\log t)^2} \, \left(\left(\log t - \frac12 \right) + \frac74 \right) > 0
	\end{align}
we see (cf. Figure \ref{fig16}), that $f$ is bijective as function from $\R^+ \to \R$. Thus, we obtain a unique positive value $\det B = g_1(X)$ for any choice $X \in \R^{3 \times 3}$ in $\eqref{eqcomponents01}_2$. Inserting this ``solution'' for $\det B$ into $\eqref{eqcomponents01}_1$ determines $\dev_3 B = g_2(X)$ as function of $X$. Thus we search for solutions to the system
	\begin{align}
	\dev_3 B = g_2(X), \qquad \det B = g_1(X) > 0, \qquad \textnormal{with} \qquad \tr(g_2(X)) = 0,
	\end{align}
which by Lemma \ref{psystin3} has a unique solution, proving the invertibility of $\sigma_{\NH}(B)$. Note that we have not yet proven that $\det \H^{\ZJ}(\sigma) \neq 0$. This is true nevertheless but will be skipped here. \\
	\begin{figure}[h!]
		\begin{center}
		\begin{minipage}[h!]{0.9\linewidth}
			\centering
			\includegraphics[scale=0.25]{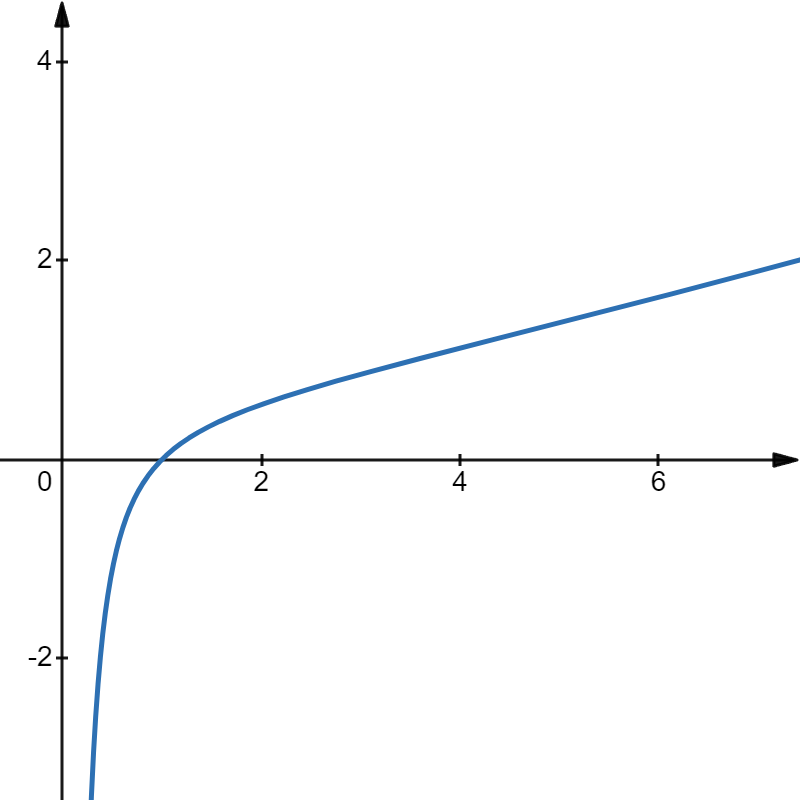}
			\put(-10,75){\footnotesize $t$}
			\put(-180,190){\footnotesize $f(t)$}
			\centering
			\qquad
			\includegraphics[scale=0.25]{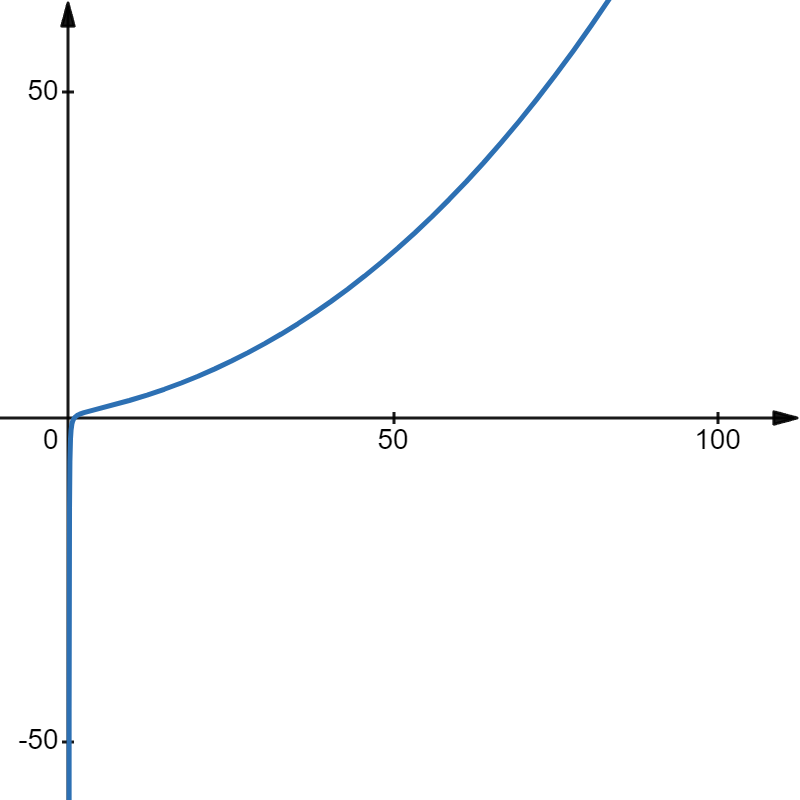}
			\put(-10,85){\footnotesize $t$}
			\put(-180,190){\footnotesize $f(t)$}
			\centering
			\caption{Illustration of the bijectivity of $f: \R^+ \to \R, \; f(t) = t^{-\frac12} \, (\log t) \, \mathrm{e}^{\frac14 \, (\log t)^2}$.}
			\label{fig16}
		\end{minipage}
	\end{center}
	\end{figure}

\noindent Let us also check positive definiteness for $\H^{\ZJ}(\sigma_{\NH})$. The Zaremba-Jaumann derivative of \eqref{eqappendixneohooke01} is given by
	\begin{equation}
	\label{eqappendixmajorsymmetry01}
	\begin{alignedat}{2}
	\frac{\DD^{\ZJ}}{\DD t}[\sigma_{\NH}] = \mu \, (\det B)^{-\frac56} &\, \left\{-\frac53 \, \tr(D) \, \dev_3 B + \dev_3(B \, D + D \, B)\right\} \\
	&\hspace*{-40pt}+ \kappa \, (\det B)^{-\frac12} \, \mathrm{e}^{\frac14 \, (\log \det B)^2} \, \tr(D) \, \left(\left(\log \det B - \frac12\right)^2 + \frac74\right) \, \id \equalscolon \H^{\ZJ}(\sigma_{\NH}).D \, ,
	\end{alignedat}
	\end{equation}
where we used the formula (cf.~\eqref{eqratetype5}) $\frac{\DD^{\ZJ}}{\DD t}[\sigma] = \DD_B\sigma(B).[D \, B + B \, D]$ for its calculation. \\
\\
In the upcoming calculation we prove that the induced tangent stiffness tensor $\H^{\ZJ}(\sigma_{\NH})$ is not positive definite. It suffices to consider
	\begin{equation}
	\begin{alignedat}{2}
	\langle -\frac53 \, \tr(D) \, &\dev_3 B + \dev_3(B \, D + D \, B) , D \rangle \\
	&= -\frac53 \, \tr(D) \left(B - \frac13 \, \tr(B)\right) \, \id + B \, D + D \, B - \frac13 \, \tr(B \, D + D \, B) \, \id , D \rangle \\
	&= -\frac53 \, \tr(D) \, \langle B,D \rangle + \frac59 \, \tr(B) \, (\tr(D))^2 + 2 \, \langle B \, D , D \rangle - \frac23 \, \tr(B \, D) \, \tr(D) \\
	&= \frac59 \, \tr(B) \, (\tr(D))^2 + 2 \, \langle B \, D , D \rangle - \frac73 \, \tr(B \, D) \, \tr(D),
	\end{alignedat}
	\end{equation}
since we will show next, that this expression is not bounded from below. For simplicity let us assume that $B$ and $D$ are only two-dimensional, as the three-dimensional case would follow similarly. Let $B = \diag(\alpha, \beta)$ and $D = \diag(a,b)$. Then we have
	\begin{align}
	\frac59 \, \tr(B) \, (\tr(D))^2 + 2 \, \langle B \, D , D \rangle - \frac73 \, \tr(B \, D) \, \tr(D) = \frac59 \, (\alpha + \beta) \, (a+b)^2 + 2 \, (\alpha \, a^2 + \beta \, b^2) - \frac73 \, (a+b)(a \, \alpha + b \, \beta)
	\end{align}
Next, we choose $a = 10^{-10}, \, b = 9 \, 10^{-10}$ and obtain
	\begin{equation}
	\begin{alignedat}{2}
	\frac59 \, (\alpha + \beta) &\, (a+b)^2 + 2 \, (\alpha \, a^2 + \beta \, b^2) - \frac73 \, (a+b)(a \, \alpha + b \, \beta) \\
	&= \frac59 \, (\alpha + \beta) \, 10^{-18} + 2 \, (\alpha \, 10^{-100} + \beta \, 81 \, 10^{-100}) - \frac73 \, 10^{-9} \, (10^{-10} \, \alpha + 9 \, 10^{-10} \, \beta).
	\end{alignedat}
	\end{equation}
Assuming further, that $\alpha>0$ is chosen small enough to be neglected, we are left with an expression in $\beta > 0$, given by
	\begin{align}
	\frac59 \, 10^{-18} \, \beta + \underbrace{162 \, 10^{-100}}_{<(4/9) \, 10^{-18}} \, \beta - 21 \, 10^{-19} \, \beta < 10^{-18} \, \beta - 2.1 \, 10^{-18} \, \beta \xrightarrow{\beta \to \infty} -\infty,
	\end{align}
so that by taking $\beta>0$ sufficiently large, we can guarantee, that $\langle \H^{\ZJ}(\sigma_{\NH}).D,D \rangle < 0$ for any choice of $\mu, \kappa > 0$.

This shows that invertibility of $B \mapsto \sigma(B)$ and positive definiteness of the corresponding induced stiffness tensor $\H^{\ZJ}(\sigma)$ are not related and positivity of $\H^{\ZJ}(\sigma)$ is even violated for the slightly compressible Neo-Hooke model \eqref{eqappendixneohooke01} which happened to be polyconvex and LH-elliptic.\footnote
{
Note that LH-ellipticity alone is only sufficient for injectivity of the Cauchy stress tensor along rank-one connected lines \cite{mihaineff1, mihaineff2, mihaineff3}.
}
\begin{rem}
Altmeyer et al.~\cite[eq.~(11)]{altmeyer2016} use their consistent hypo-elastic framework to provide a rate-formulation of a slightly compressible Mooney-Rivlin material with Lie-derivative $\mathcal{L}_{v_\varphi}$ (the Truesdell rate), i.e.~their constitutive law reads \cite[eq.~42]{altmeyer2016}
	\begin{align}
	\mathcal{L}_{v_\varphi}(\sigma) = \H^{\mathcal{L}_{v_{\varphi}}}(B).D \qquad \iff \qquad \frac{\DD^{\TR}}{\DD t}[\sigma] = \H^{\TR}(B).D,
	\end{align}
where $\H^{\mathcal{L}_{v_\varphi}}$ is derived from hyperelasticity by hand. No investigation of invertibility or positive definiteness of $\H^{\mathcal{L}_{v_\varphi}} = \H^{\TR}$ is undertaken. In \cite{fei1994}, starting from hyperelasticity [cf.~\cite{Panicaud2016, romenskii1974}], the rate-formulation for the Truesdell-rate in explicit form, also for the Mooney-Rivlin model, is presented, yielding
	\begin{align}
	\hspace*{-12pt} \frac{\DD^{\TR}}{\DD t}[\sigma] \colonequals \frac{\DD}{\DD t}[\sigma] - L \, \sigma - \sigma \, L^T + \sigma \, \tr (D) = \DD_B\sigma(B).[D \, B + B \, D] - (D \, \sigma + \sigma \, D) + \sigma \, \tr (D) \equalscolon \H^{\TR}(B).D \, .
	\end{align}
Here, we can observe that the Truesdell rate, which is \emph{non-corotational}, will not allow for a chain rule like formula (cf.~Section \ref{appendchainrule}). Indeed, with $\frac{\DD^{\TR}}{\DD t}[B] = B \, \tr(D)$ the chain rule would imply
	\begin{equation}
	\begin{alignedat}{2}
	\frac{\DD^{\TR}}{\DD t}[\sigma] = \DD_B\sigma(B).[D \, B + B \, D] - (D \, \sigma + \sigma \, D) + \sigma \, \tr(D) \overset{!}&{=} \DD_B \sigma(B).\!\left[\frac{\DD^{\TR}}{\DD t}[B]\right] = \DD_B \sigma(B).[B \, \tr(D)] \\
	\iff \qquad \DD_B\sigma(B).[D \, B + B \, D - B \, \tr(D)] &= D \, \sigma + \sigma \, D - \sigma \, \tr(D)
	\end{alignedat}
	\end{equation}
Inserting e.g.~$\sigma(B) = B^2$ and $D$ with $\tr(D) = 0$ would then yield with $\DD_B(B^2).H = B \, H + H \, B$
	\begin{equation}
	\begin{alignedat}{2}
	\DD_B \sigma(B).[D \, B + B \, D - B \, \tr(D)] &= B \, [D \, B + B \, D] + [D \, B + B \, D] \, B \\
	&\neq D \, B^2 + B^2 \ D = D \, \sigma + \sigma \, D - \sigma \, \tr(D) \, .
	\end{alignedat}
	\end{equation}
Thus a similar chain rule as in Section \ref{appendchainrule} does not apply to the Truesdell derivative. This implies that the possible positive definiteness of $\H^{\TR}(\sigma)$ is, in general, unrelated to the positive definiteness of $\DD_{\log B} \widehat \sigma(\log B)$, since
	\begin{equation}
	\begin{alignedat}{2}
	\frac{\DD^{\TR}}{\DD t}[\widehat \sigma(\log B)] = \DD_{\log B} \widehat \sigma(\log B). &\DD_B \log B. [B \, D + D \, B] \\
	& - [D \, \widehat \sigma(\log B) + \widehat \sigma(\log B) \, D] - \tr(D) \, \widehat \sigma(\log B) \equalscolon \H^{\TR}(B).D \, ,
	\end{alignedat}
	\end{equation}
which supports, again, our choice of only considering corotational rates $\dd \frac{\DD^{\circ}}{\DD t}$.
\end{rem}

\subsection{Monotonicity in $V$ versus monotonicity in $\log V$} \label{appmono001}
	 In a uniaxial situation, we can always consider 
	 \begin{align}
	 \widehat{\sigma}(\log \lambda)\colonequals\sigma(\lambda), \qquad \sigma\colon\mathbb{R}^+\to \mathbb{R}, \quad \widehat{\sigma}:\mathbb{R}\to \mathbb{R}, \quad \textnormal{differentiable}.
	 \end{align}
	 Then it is clear that $\lambda \mapsto \sigma(\lambda)$ is invertible if and only if $\log \lambda \mapsto \widehat \sigma(\log \lambda)$ is invertible and
	 \begin{align}
	 \lambda\to \sigma(\lambda) 
	 \ \ \textnormal{monotone}\ \ \iff \ \  \log\lambda\to \widehat{\sigma}(\log\lambda) \ \ \textnormal{monotone}, \qquad \textnormal{since} \quad \sigma'(\lambda) = \widehat \sigma'(\log \lambda) \, \frac{1}{\lambda}, \quad \lambda > 0.
	 \end{align}
	 However, in the matrix setting the latter simple correspondence is lost as is shown next. This shows that the statement in \cite[eq.~50]{jog2013conditions} is incorrect, cf.~the discussion in \cite{NeffMartin14}.
\subsubsection{Example: $\sigma$ monotone in $V$ does not imply $\widehat \sigma$ monotone in $\log V$} \label{appendixex1}
Consider the elastic energy (see also \cite{richter1948isotrope}) $\WW(F) = 2\,\mu\,\det V\,\bigl\{\tr(V)-4\bigr\}$, leading to the hyperelastic Cauchy stress
	\begin{equation}\label{sVp}
	\sigma(V)\ =\ 2\,\mu\, \{(V-\id) + \tr(V-\id)\,\id \}.
	\end{equation}
Obviously, this law is monotone in $V-\id$ since $\tr(\,)$ is a linear function and thus it is also monotone in $V$. \\
\\
Rewriting $\sigma(V)$ in terms of $\log V$ we obtain
	\begin{equation}
	\widehat\sigma(\log V) = \sigma(V)\ = V - \id + \tr(V-\id)\,\id\ =\ \exp(\log V) - \id + \tr\bigl(\exp(\log V)-\id\bigr)\,\id\,.
	\end{equation}
Hence, for $S\in\Sym(3)$ we define $ \widetilde\sigma\colon\Sym(3)\to\Sym(3)$ by
	\begin{equation}
	\widehat\sigma(S) = \exp(S)-\id + \tr\bigl(\exp(S)-\id\bigr)\,\id\,,
	\end{equation}
so that monotonicity of $\widehat \sigma$ in $\log V$ is now equivalent to monotonicity of $\widehat \sigma$ in $S$, i.e.
	\begin{equation}\label{eq_one}
	\iprod{ \widehat\sigma(S_1)- \widehat\sigma(S_2),S_1-S_2}_{\R^{3\times 3}} > 0\,.
	\end{equation}
Condition \eqref{eq_one} implies monotonicity in principal Cauchy stresses versus principal ($\log$)-strains, so that we have
	\begin{equation}\label{eq_two}{\scriptsize
	\iprod{\matr{  \widehat\sigma_1(s_1,s_2,s_3) & 0 & 0 \\ 0 & \!\!\!\!\! \widehat\sigma_2(s_1,s_2,s_3)\!\!\!\!\! & 0 \\ 0 & 0 &  \widehat\sigma_3(s_1,s_2,s_3) }-\matr{  \widehat\sigma_1(\overline s_1,\overline s_2,\overline s_3) & 0 & 0 \\ 0 & \!\!\!\!\! \widehat\sigma_2(\overline s_1,\overline s_2,\overline s_3)\!\!\!\!\! & 0 \\ 0 & 0 &  \widehat\sigma_3(\overline s_1,\overline s_2,\overline s_3) } , \matr{ s_1-\overline s_1 & 0 & 0 \\ 0 & \!\!\!\!\!s_2-\overline s_2\!\!\!\!\! & 0 \\ 0 & 0 & s_3-\overline s_3} }_{\R^{3\times 3}}\ >\ 0}\,.
	 \end{equation}
We may rewrite \eqref{eq_two} in the form
	\begin{equation}\label{eq_two_b}
	\iprod{\matr{ \widehat\sigma_1(s_1,s_2,s_3)\\ \widehat\sigma_2(s_1,s_2,s_3)\\ \widehat\sigma_3(s_1,s_2,s_3)} - \matr{ \widehat\sigma_1(\overline s_1,\overline s_2,\overline s_3)\\ \widehat\sigma_2(\overline s_1,\overline s_2,\overline s_3)\\ \widehat\sigma_3(\overline s_1,\overline s_2,\overline s_3)} , \matr{s_1-\overline s_1 \\ s_2-\overline s_2 \\ s_3-\overline s_3} }_{\R^3}\ >\ 0\,.
	\end{equation}
Here, in terms of principal Cauchy stresses we have by setting $x\colonequals \log \lambda_1, \; y\colonequals \log \lambda_2, \; z\colonequals \log \lambda_3$
	\begin{equation}
	\begin{pmatrix} \widehat \sigma_1(x,y,z) \\ \widehat \sigma_2(x,y,z) \\ \widehat \sigma_3(x,y,z) \end{pmatrix} = \matr{e^x-1\\e^y-1\\e^z-1} + \matr{e^x + e^y + e^z - 3 \\ e^x + e^y + e^z - 3 \\ e^x + e^y + e^z - 3} = \matr{2\,e^x+e^y+e^z -4 \\ 2\,e^y + e^x+e^z - 4 \\ 2\,e^z + e^x+e^y - 4}.
	\end{equation}
The Jacobian of $ \widehat\sigma$ is
	\begin{equation}
	\DD \widehat\sigma(x,y,z)\ =\ \matr{ 2\,e^x & e^y & e^z \\ e^x & 2\,e^y & e^z \\ e^x & e^y & 2\,e^z}\ \not\in \Sym(3)\,.
	\end{equation}
From \eqref{eq_two_b} we imply by application of the mean value theorem the inequality (with $\delta_i \colonequals s_i - \overline s_i, \, i=1,2,3$)
	\begin{equation}\label{eq_three}
	\iprod{\DD \widehat\sigma(x,y,z) . \!\! \matr{\delta_1\\\delta_2\\\delta_3} \!,\! \matr{\delta_1\\\delta_2\\\delta_3}}_{\R^3}\ >\ 0
	\qquad \iff \qquad
	\iprod{\bigl[ \sym \DD \widehat\sigma(x,y,z)\bigr] . \!\! \matr{\delta_1\\\delta_2\\\delta_3} \!,\! \matr{\delta_1\\\delta_2\\\delta_3}}_{\R^3}\ >\ 0\,,
	\end{equation}
since the skew-symmetric part of a matrix $X$ always fulfills $\langle \sk(X) . v, v \rangle = 0$ for all $v \in \R^3$. Hence, if $\widehat \sigma$ was monotone in $S$, it should hold that $\sym \DD \widehat\sigma(x,y,z)$ is positive definite, where
	\begin{equation}\label{eq_four}
	\sym \DD \widehat\sigma(x,y,z)\ =\ \matr{ 2\,e^x & \frac 12(e^y+e^x) & \frac 12(e^z+e^x) \\ \frac 12(e^y+e^x) & 2\,e^y & \frac 12(e^z+e^y) \\ \frac 12(e^x+e^z) & \frac 12(e^y+e^z) & 2\,e^z}\,.
	\end{equation}
However, according to the Sylvester criterion for positive definiteness of \eqref{eq_four} is
	\begin{align}
	2\,e^x &> 0,\notag \\
	\det\matr{2\,e^x & \frac 12(e^y+e^x) \\ \frac 12(e^y+e^x) & 2\,e^y }\ = 4\,e^x\,e^y - \frac 14[e^y\,e^y + 2\,e^y\,e^x + e^x\,e^x] &> 0,\label{eq_five}
	\\
	\det \sym\DD \widehat\sigma\ &> 0,\notag
	\end{align}
so that, by choosing $y=0$ and $x=\log 20$, we have that
\begin{align}
4\,e^x\,e^y - \frac 14[e^y\,e^y + 2\,e^y\,e^x + e^x\,e^x]=80-\frac{1}{4}[1+40+400]<80-\frac{1}{4}\,400<0
\end{align}
and therefore the positive definiteness of $\sym \DD \widetilde\sigma(x,y,z)$ is violated, which is equivalent to the statement that $\widehat \sigma(\log V)$ is not monotone in $\log V$.
\subsubsection{Example: $\widehat \sigma$ monotone in $\log V$ does not imply $\sigma$ monotone in $V$} \label{appendixex2}
Consider the constitutive Cauchy stress law \cite{Hencky1928}
	\begin{equation}
	\label{eqhenckynonmonotone1}
	\widehat \sigma(\log V)\ =\ 2\,\mu\,\log V + \lambda\,\tr(\log V)\, \id \, .
	\end{equation}
For $\mu>0$, $\lambda>0$, this  expression is clearly strictly monotone in $\log V$, i.e.
	\begin{align}
	\iprod{\widehat \sigma(\log V_1)-\widehat \sigma(\log V_2),\log V_1-\log V_2)}_{\mathbb{R}^{3\times 3}}\ >\ 0 \qquad \forall \, V_1, V_2 \in \Sym^{++}(3), \qquad V_1 \neq V_2 \,.
	\end{align}
However, $\sigma(V) = \widehat \sigma(\log V)$ is not monotone in $V$, i.e.
	\begin{equation}\label{eq_monotonicity_in_B}
	\iprod{\widehat \sigma(\log V_1)-\widehat \sigma(\log V_2),V_1-V_2)}_{\mathbb{R}^{3\times 3}}\ = \ \iprod{\sigma(V_1)-\sigma(V_2),V_1-V_2)}_{\mathbb{R}^{3\times 3}} \ >\ 0
	\end{equation}
\textbf{does not hold} for every $V_1, V_2 \in \Sym^{++}(3), \; V_1 \neq V_2$, if $\mu, \lambda > 0$. \\
\\
To see this, assume that $\mu$ is small $0<\mu\ll1$ and let $\lambda>0$ be arbitrary. Then we can, for now, neglect the first term in \eqref{eqhenckynonmonotone1} by smallness of $\mu$, so that \eqref{eq_monotonicity_in_B} corresponds approximately to 
	\begin{equation}\label{to_show}
	\iprod{\tr(\log V_1)\,\id - \tr(\log V_2)\,\id, V_1-V_2}_{\mathbb{R}^{3\times 3}}\ \geq\ 0\,.
	\end{equation}
By choosing $V_1 = \diag(\beta, \frac{1}{\beta}, 1), \; V_2 = \diag(\overline{\lambda}_1, \overline{\lambda}_2, 1)$ we get
	\begin{equation}
	\label{eq_monotonicity_concr}
	\begin{alignedat}{2}
	\langle\tr(\log V_1)\,\id &- \tr(\log V_2)\,\id, V_1-V_2\rangle_{\mathbb{R}^{3\times 3}}\ =\ \iprod{0 - \log(\overline\lambda_1\,\overline\lambda_2\,1)\,\id, V_1-V_2}_{\mathbb{R}^{3\times 3}} \\
	 \quad =\ &-\log(\overline\lambda_1\,\overline\lambda_2)\,\bigl[\tr(V_1)-\tr(V_2)\bigr]\ =\ -\log(\overline\lambda_1\,\overline\lambda_2)\,\bigl[\beta+\frac 1\beta + 1 - (\overline\lambda_1+\overline\lambda_2+1)\bigr] \\
	\quad =\ &-\log(\overline\lambda_1\,\overline\lambda_2)\,\bigl[\underbrace{\beta+\frac 1\beta}_{\geq 2}\ -\ (\overline\lambda_1+\overline\lambda_2)\bigr]\,.
	\end{alignedat}
	\end{equation}
Now we choose $\beta=3$, $\overline\lambda_1=1$, $\overline\lambda_2=2$ in \eqref{eq_monotonicity_concr} and obtain
	\begin{equation}
	\iprod{\tr(\log V_1)\,\id - \tr(\log V_2)\,\id, V_1-V_2}_{\mathbb{R}^{3\times 3}}\ =- \bigl[3+\frac 13 - 1 - 2\bigr]\log 2\, =\ -\frac 13\,\log 2\ <\ 0\,.
	\end{equation}
Furthermore, calculating the $\mu$-part of \eqref{eqhenckynonmonotone1}, which we neglected so far, yields
	\begin{equation}
	\begin{alignedat}{2}
	\langle \log V_1 - \log V_2, V_1-V_2\rangle_{\mathbb{R}^{3\times 3}}&=(\log \beta-\log \overline{\lambda}_1)(\beta- \overline{\lambda}_1)+(-\log \beta-\log \overline{\lambda}_2) \left(\frac{1}{\beta}- \overline{\lambda}_2\right) \\
	&=2\,\log 3-\left(\frac{1}{3}-2\right)\log 6=2\,\log 18-\frac{1}{3}\log 6.
	\end{alignedat}
	\end{equation}
Thus, for this choice, we have
	\begin{equation}
	\iprod{\widehat \sigma(\log V_1)-\widehat \sigma(\log V_2),V_1-V_2)}_{\mathbb{R}^{3\times 3}}=2 \, \mu \left(2\log 18-\frac{1}{3}\log 6\right)-\lambda\,\frac{1}{3}\log 2,
	\end{equation}
which is negative if $2\, \mu<\lambda\, \frac{\log \sqrt[3]{2}}{\log \frac{18^2}{\sqrt[3]{6}}}$, contradicting \eqref{eq_monotonicity_in_B}.
\subsection{Second order work condition versus corotational stability} \label{appsecondorderwork}
In linear elasticity, the stored elastic energy can be expressed as
	\begin{align}
	\mathcal{E}(t) = \int_{V_0} \WW^{\lin}(\varepsilon(t)) \dif V_0 = \int_{V_0} \frac12 \, \langle \C . \varepsilon(t), \varepsilon(t) \rangle \dif V_0 \qquad \textnormal{``work''}.
	\end{align}
Taking time derivatives yields
	\begin{align}
	\frac{\dif}{\dif t} \mathcal{E}(t) &= \int_{V_0} \langle \DD_{\varepsilon} \WW^{\lin}(\varepsilon(t)), \dot{\varepsilon}(t) \rangle \dif V_0 = \int_{V_0} \langle \sigma(t) , \dot{\varepsilon}(t) \rangle \dif V_0 = \mathcal{P}_{\textnormal{int}}^{\lin} \qquad \textnormal{``internal power''} \notag \\
	\underbrace{\frac{\dif^2}{\dif t^2} \mathcal{E}(t)}_{\substack{\textnormal{``second} \\ \textnormal{order work''}}} &= \int_{V_0} \langle \dot{\sigma}(t), \dot{\varepsilon}(t) \rangle + \langle \sigma(t), \ddot{\varepsilon}(t) \rangle \dif V_0 \overset{\sigma \in \Sym(3)}{=} \int_{V_0} \langle \dot{\sigma}(t) , \dot{\varepsilon}(t) \rangle + \langle \sigma(t) , \DD u_{tt} \rangle \dif V_0 \\
	&= \int_{V_0} \langle \dot{\sigma}(t), \dot{\varepsilon}(t) \rangle - \langle \underbrace{\textnormal{Div} \, \sigma(t)}_{\substack{= \; 0 \; \textnormal{in} \\ \textnormal{equilibrium}}} , u_{tt} \rangle \dif V_0 = \int_{V_0} \langle \dot{\sigma}(t) , \dot{\varepsilon}(t) \rangle \dif V_0 = \int_{V_0} \langle \C . \dot{\varepsilon}(t), \dot{\varepsilon}(t) \rangle \dif V_0 \ge c^{+} \, \norm{\dot{\varepsilon}(t)}^2_{V_0}. \notag
	\end{align}
Hence positive ``second order work'' in linear elasticity
	\begin{align}
	\int_{V_0} \langle \dot{\sigma}(t), \dot{\varepsilon}(t) \rangle \dif V_0 > 0
	\end{align}
is sufficient for having a stable (local) equilibrium and expresses nothing else than $\C \in \Sym^{++}_4(6)$ for the constitutive law $\sigma = \C . \varepsilon$. \\
\\
In nonlinear elasticity we have $\mathcal{E}(t) = \int_{V_0} \WW(F(t)) \dif V_0$ with (recall $\sigma = \frac{1}{J} \, S_1 \, F^T$ and $D = \sym L = \sym (\dot F \, F^{-1})$)
	\begin{equation}
	\label{eqPenergy}
	\begin{alignedat}{2}
	\frac{\dif}{\dif t} \mathcal{E}(t) = \int_{V_0} \langle \DD_F \WW(F(t)), \dot F(t) \rangle \dif V_0 &= \int_{V_0} \langle S_1(t), \dot F(t) \rangle \dif V_0 = \int_{V_0} \langle S_1 \, F^T, \dot F \, F^{-1} \rangle \dif V_0 \\
	&= \int_{V_0} \langle \sigma, L \rangle \, \underbrace{J \, \dif V_0}_{= \, \dif V_t} \overset{\sigma \in \Sym(3)}{=} \int_{V_t} \langle \sigma, D \rangle \dif V_t = \mathcal{P}_{\textnormal{int}}
	\end{alignedat}
	\end{equation}
and
	\begin{equation}
	\begin{alignedat}{2}
	\frac{\dif^2}{\dif t^2} \mathcal{E}(t) &= \int_{V_t} \langle \frac{\DD}{\DD t}[\sigma], D(t) \rangle + \langle \sigma(t), \dot{D}(t) \rangle \dif V_t = \int_{V_t} \langle \frac{\DD}{\DD t}[\sigma] , D(t) \rangle + \langle \sigma(t) , \DD_{\xi} v_{,t}(\xi,t) \rangle \dif V_t \\
	&= \int_{V_t} \underbrace{\langle \frac{\DD}{\DD t}[\sigma] , D(t) \rangle}_{\textnormal{not objective!}} - \langle \underbrace{\textnormal{Div}_{\xi} \, \sigma(t)}_{\substack{= \; 0 \; \textnormal{in spatial} \\ \textnormal{equilibrium}}}, v_{,t}(\xi,t) \rangle \dif V_t = \int_{V_t} \langle \frac{\DD}{\DD t}[\sigma], D(t) \rangle \dif V_t .
	\end{alignedat}
	\end{equation}
Thus, we observe the concordance
	\begin{align*}
	\textnormal{linear elasticity} & &&  & & \qquad \textnormal{nonlinear elasticity} \\
	\int_{V_0} \langle \sigma, \varepsilon \rangle \dif V_0 \quad & &&  & & \rightsquigarrow \quad \int_{V_0} \WW(F(t)) \dif V_0 &&\qquad \textnormal{``energy/work'' (objective)} \\
	\int_{V_0} \langle \sigma, \dot{\varepsilon} \rangle \dif V_0 \quad & &&  & &\rightsquigarrow \quad \int_{V_t} \langle \sigma, D \rangle \dif V_t = \mathcal{P}_{\textnormal{int}} &&\qquad \textnormal{``internal power'', ``rate of work'' (objective)} \\
	\int_{V_0} \langle \dot{\sigma}, \dot{\varepsilon} \rangle \dif V_0 \quad & &&  & &\rightsquigarrow \quad \int_{V_t} \langle \frac{\DD}{\DD t}[\sigma], D \rangle \dif V_t &&\qquad \textnormal{``second order work'' (not objective).} 
	\end{align*}
Therefore, the ``second order work'' condition $\frac{\dif^2}{\dif t^2} \mathcal{E}(t) > 0$ in equilibrium for finite strain can be written as
	\begin{align}
	\frac{\dif^2}{\dif t^2} \mathcal{E}(t) = \int_{V_t} \langle \frac{\DD}{\DD t} [\sigma(t)], D(t) \rangle \dif V_t > 0
	\end{align}
and ``almost'' looks like
	\begin{align}
	\int_{V_t} \underbrace{\langle \frac{\DD^{\circ}}{\DD t}[\sigma(t)], D(t) \rangle}_{\textnormal{objective!}} \dif V_t > 0.
	\end{align}
The local corotational stability requirement
	\begin{align}
	\langle \frac{\DD^{\circ}}{\DD t}[\sigma(t)], D(t) \rangle > 0 \qquad \forall \, D \in \Sym(3) \setminus \{0\}
	\end{align}
must therefore not be confused with $\frac{\dif^2}{\dif t^2} \mathcal{E}(t) > 0$, the positive second order work.
\subsection{Further calculus with the Zaremba-Jaumann rate}
Recall the formulas 
	\begin{align}
	\label{ZJrate01A}
	\boxed{\frac{\DD^{\ZJ}}{\DD t}[\sigma] \colonequals \frac{\DD}{\DD t}[\sigma] + \sigma \, W - W \, \sigma = Q \, \frac{\DD}{\DD t}[Q^T \, \sigma \, Q] \, Q^T  \quad \textnormal{for $Q(t) \in \OO(3)$ so that} \;  W =\dot{Q} \, Q^T\in \mathfrak{so}(3)}
	\end{align}
(where we set $Q\colonequals Q^W$ by abuse of notation for simplicity) and
	\begin{align}
	\label{eqnewZJappend}
	\boxed{\frac{\DD^{\ZJ}}{\DD t}[\sigma] = \frac{\DD}{\DD t}[\sigma] - W \, \sigma + \sigma \, W \overset{\eqref{eqratetype5}}{=} \DD_B\sigma(B).[D \, B + B \, D] = \H^{\ZJ}(B).D}
	\end{align}
for the derivation of the Zaremba-Jaumann derivative, which are valid for any isotropic tensor function \break $\sigma\colon \Sym^{++}(3) \to \Sym(3)$. \\
\\
The representation \eqref{eqnewZJappend} can be used to show the frame-indifference of the Zaremba-Jaumann rate in the following manner.
\begin{lem}[Frame-indifference of the Zaremba-Jaumann rate] \label{frameindZJ}
The Zaremba-Jaumann rate is frame-indifferent, i.e.~under the transformation $F(t) \mapsto Q(t) \, F(t)$ we have
	\begin{align}
	\label{eqframeindif2}
	\frac{\DD^{\ZJ}}{\DD t}[Q(t) \, \sigma(t) \, Q^T(t)] = Q(t) \, \frac{\DD^{\ZJ}}{\DD t} [\sigma] \, Q^T(t).
	\end{align}
\end{lem}
\begin{proof}
We denote the image of the transformation $F(t) \mapsto F^*(t) = Q(t) \, F(t)$ by a star, for example $\sigma^* = Q \, \sigma \, Q^T$. Then we first observe that for $L = \dot F \, F^{-1}$ we obtain
	\begin{align}
	L^* = \frac{\DD}{\DD t} F^* \, (F^*)^{-1} = \frac{\DD}{\DD t}[Q \, F] \, (F^{-1} \, Q^T) = (\dot Q \, F + Q \, \dot F) \, F^{-1} \, Q^T = \dot Q \, Q^T + Q \, L \, Q^T
	\end{align}
so that for the skew-symmetric part $W = \sk L$ of $L$ we obtain the identity 
	\begin{align}
	\label{identforw}
	W^* = \dot Q \, Q^T + Q \, W \, Q^T = \dot Q \, Q^T + Q \, W \, Q^T \qquad \iff \qquad W^* - \dot Q \, Q^T = Q \, W \, Q^T \, ,
	\end{align}
since $\sk(\,)$ is an isotropic tensor function. Then a direct computation shows
	\begin{align}
	\frac{\DD^{\ZJ}}{\DD t}[\sigma^*(t)] &= \frac{\DD^{\ZJ}}{\DD t}[Q(t) \, \sigma(t) \, Q^T(t)] = \frac{\DD}{\DD t}[Q \, \sigma \, Q^T] - W^* \, \sigma^* + \sigma^* \, W^* \notag \\
	&= \dot Q \, \sigma \, Q^T + Q \, \frac{\DD}{\DD t}[\sigma] \, Q^T + Q \, \sigma \, \dot Q^T - W \, \sigma^* + \sigma^* \, W \notag \\
	&=\dot Q \, Q^T \, \underbrace{Q \, \sigma \, Q^T}_{\sigma^*} + Q \, \frac{\DD}{\DD t}[\sigma] \, Q^T + \underbrace{Q \, \sigma \, Q^T}_{\sigma^*} \, \underbrace{Q \, \dot Q^T}_{= \; -\dot Q \, Q^T} - W^* \, \sigma^* + \sigma^* \, W^* \\
	&= Q \, \frac{\DD}{\DD t}[\sigma] \, Q^T - (W^* - \dot Q \, Q^T) \, \sigma^* + \sigma^* \, (W^* - \dot Q \, Q^T) \notag \\
	\overset{\eqref{identforw}}&{=} Q \, \frac{\DD}{\DD t}[\sigma] \, Q^T - Q \, W \, \sigma \, Q^T + Q \, \sigma \, W \, Q^T \overset{\eqref{ZJrate01A}}{=} Q \, \frac{\DD^{\ZJ}}{\DD t}[\sigma] \, Q^T \, . \notag \qedhere
	\end{align}
\end{proof}
\begin{cor}
A corotational rate $\frac{\DD^{\circ}}{\DD t}$ is objective if and only if the spin tensor $\Omega^{\circ}$ transforms according to $\Omega^{\circ} \mapsto \dot Q \, Q^T + Q \, \Omega^{\circ} \, Q^T$ under a Euclidean transformation $F \mapsto Q(t) \, F(t)$.
\end{cor}
\noindent Furthermore, by applying \eqref{eqnewZJappend} to the identity $\textnormal{id}: B \mapsto B$ we get
	\begin{align}
	\label{eqappendixZJcalc001}
	\frac{\DD^{\ZJ}}{\DD t}[B] = \DD_B B . [B \, D + D \, B] = \textnormal{id} . [B \, D + D \, B] = [B \, D + D \, B] = 2 \, \sym(B \, D),
	\end{align}
leading to the estimate
	\begin{align}
	\langle \frac{\DD^{\ZJ}}{\DD t}[B], D \rangle = \langle B \, D + D \, B , D \rangle = 2 \, \langle B \, D , D \rangle \ge 2 \, \lambda_{\min}(B) \, \norm{D}^2.
	\end{align}
Additionally, we note that
	\begin{align}
	\langle \frac{\DD^{\ZJ}}{\DD t}[\log B] , D \rangle \ge c^+ \, \norm{D}^2,
	\end{align}
which has been shown in Lemma \ref{dopeasslemma001}.  Another formula of interest is
	\begin{align}
	\label{eqappendixtrace1}
	\tr(\DD_B \log B . [B \, D + D \, B]) = \tr\left(\frac{\DD^{\ZJ}}{\DD t}[\log B]\right) = \tr(2 \, D),
	\end{align}
which is surprising, since we can show (see \eqref{eqboxed2D}), that in general
	\begin{align}
	\DD_B \log B . [B \, D + D \, B] \neq 2 \, D.
	\end{align}
To prove the validity of \eqref{eqappendixtrace1}, we use representation $\eqref{eqappendixowncalc001}_4$ for $\DD_B \log B . [B \, D + D \, B]$, which is given by
	\begin{align}
	\DD_B \log B . [B \, D + D \, B] = R \, \big[ \mathcal{L} \circ (R^T \, D \, R) \big] \, R^T.
	\end{align}
We readily check, that every diagonal element of $\mathcal{L} = X \, \mathcal{F} + \mathcal{F} \, X$, where
	\begin{align}
	X = \diag(\lambda_1^2, \lambda_2^2, \lambda_3^2) \qquad \textnormal{and} \qquad \mathcal{F}_{ij} =
		\left\{
		\begin{array}{ll}
		\frac{1}{\lambda_i^2} &\qquad \textnormal{if} \quad i = j, \\
		\frac{\log \lambda_i^2 - \log \lambda_j^2}{\lambda_i^2 - \lambda_j^2} &\qquad \textnormal{if} \quad i \neq j,
		\end{array}
		\right.
	\end{align}
has the value $2$, so that $\mathcal{L} \circ \id = 2 \, \id$. Then it easily follows
	\begin{equation}
	\begin{alignedat}{2}
	\tr(\DD_B \log B . [B \, D + D \, B]) &= \langle R \, \big[ \mathcal{L} \circ (R^T \, D \, R) \big] \, R^T, \id \rangle = \langle \mathcal{L} \circ (R^T \, D \, R) , \id \rangle \\
	&= \langle R^T \, D \, R , \mathcal{L} \circ \id \rangle = \langle R^T \, D \, R , 2 \, \id \rangle = \tr(2 \, D).
	\end{alignedat}
	\end{equation}
A similar phenomenon involving the logarithmic strain is described by Lemma \ref{lemA18} and compare to Appendix~\ref{appgoldenth}.
\footnotesize
\subsection{Supplementary material}
In this supplementary part we gather results that are more or less known, but helpful for a better understanding of the paper.
\subsubsection{Remarks about zero-grade hypo-elasticity} \label{appendixzerograde}
In a stress-free state of a hypo-elastic material, the Cauchy stress tensor $\sigma$ is equal to zero ($\sigma \equiv 0$). When loaded in this state, the following linear relationship\footnote
{
From Ludwig Prandtl \cite{Prandtl}: ``[the] elastic response depends only on the increment in the loading and not on prestresses'' - nur ``elastisch bestimmt'' (elastic determination). This requirement of Prandtl may suitably be interpreted as zero-grade hypoelasticity, i.e. $\H^{\ZJ}(\sigma). D = \C^{\iso}.D$ is \textbf{independent} of the current stress level $\sigma$. In order to be consistent with hyperelasticity, one has to use, however, the logarithmic rate and apply it to the Kirchhoff stress $\tau$, then $\frac{\DD^{\log}}{\DD t}[\tau] = \C^{\iso}.D$ integrates to
	\begin{align}
	\WW_{\textnormal{Hencky}}(F) = \mu \, \norm{\log V}^2 + \frac{\lambda}{2} \, \tr^2(\log V), \qquad \frac{\DD^{\log}}{\DD t}[\tau] = \C^{\iso}.D \qquad \textnormal{(cf. \cite{NeffGhibaLankeit}, \cite{NeffGhibaPoly})}.
	\end{align}
When using $\frac{\DD^{\log}}{\DD t}[\sigma] = \C^{\iso}.D$, then this constitutive prescription integrates to $\sigma = 2 \, \mu \, \log V + \lambda \, \tr(\log V) \, \id$, which is not hyperelastic. This was already known to Hencky.
}
is assumed:
	\begin{align}
	\label{zerograde}
	\frac{\DD^{\circ}}{\DD t}[\sigma] = \H^{\circ}(0).D = \C^{\iso}.D = 2\mu \, D + \lambda \, \tr(D) \, \id.
	\end{align}
This expression is independent of the choice of the objective time derivative $\frac{\DD^{\circ}}{\DD t}[\sigma]$ used for the stresses\footnote
{
In the paper \cite{romano2014} the authors discuss the philosophical issues connected with the use of the Zaremba-Jaumann derivative and hypo-elasticity with respect to integrability from scratch. Here, to the contrary, for the choice $\H^*(\sigma).D \colonequals \H^{\circ}(\sigma).D = \frac{\DD^{\circ}}{\DD t}[\sigma]$ no problems of appropriate choice of rates or conservativeness appear since we are always consistent with hyperelasticity or Cauchy-elasticity by definition of the induced tangent stiffness $\H^{\circ}(\sigma)$.
} (cf.~Section \ref{app13}).
Since there is an isotropic tensor function on the right side of equation \eqref{zerograde}, the stresses evolve from a stress-free state under small deformations initially resembling those in a linear isotropic material. The model \eqref{zerograde} is called hypo-elasticity of zero-grade ($\H(0) =$ const.). Regarding zero-grade hypo-elasticity, we have the following excerpt from Truesdell \cite[p.123]{truesdell55hypo}: \\
\\
``The classical linear theory of elasticity is defined by the constitutive equation
	\begin{align}
	\label{eqtruesdellpureelast01}
	\textnormal{small stress} \; = \; f \textnormal{(small strain from an unstressed state)}.
	\end{align}
Here, as henceforth, we are taking twice the shear modulus $\mu$ for the particular material as the unit of stress. The usual form of the classical finite strain theory extends \eqref{eqtruesdellpureelast01} by the more general constitutive equation
	\begin{align}
	\label{eqtruesdellpureelast02}
	\textnormal{stress} \; = \; f \textnormal{(strain from an unstressed state).}
	\end{align}
While the last few years have brought physical confirmation to the finite strain theory for rubber, there remain many physical materials which are linearly elastic under small enough strain but which in large strain behave in a fashion the finite strain theory is not intended to represent. Now the linear theory of elasticity embodied in \eqref{eqtruesdellpureelast02}, while in one sense more general, in another is far more restrictive than that employed in \eqref{eqtruesdellpureelast01}. It asserts that no matter how violent the distortion, the body responds to it only with reference to the unstressed state. The body has a memory for its initial state only, being utterly oblivious of all intermediate stages. This is true of the linear theory, too, but with the restriction that the initial state must be a very near one. It is much less to expect of a material that it remember where it has just been than where it was long ago. Thus we may prefer to generalize \eqref{eqtruesdellpureelast01} by
	\begin{align}
	\label{eqtruesdellpureelast03}
	\textnormal{stress increment} \; = \; f \textnormal{(small strain from the immediately preceding state).}
	\end{align}
When the immediately preceding state is unstressed, \eqref{eqtruesdellpureelast03} reduces to \eqref{eqtruesdellpureelast01}. Thus \eqref{eqtruesdellpureelast03} and \eqref{eqtruesdellpureelast02} represent two essentially different ideas of spring, both yielding as a common first approximation the classical linear theory for small strains \cite{Marsden83} from an unstressed state. \\
\\
When we come to realize \eqref{eqtruesdellpureelast03} mathematically, it is natural to begin with
	\begin{align}
	\textnormal{rate of stress} \; = \; f \textnormal{(rate of deformation)''}.
	\end{align}
Conversely, a material that reacts anisotropically rather than isotropically to small deformations from a stress-free state cannot be hypo-elastic in the above sense.
\subsubsection{Some objective derivatives}
Some possible objective derivatives $\frac{\DD^{\sharp}}{\DD t}[\sigma]$ of the Cauchy stress tensor $\sigma$, not necessarily corotational, but all invariant under a Euclidean transformation of the observer, are listed below:
	\begin{alignat}{2}
	\frac{\DD^{\ZJ}}{\DD t}[\sigma] &\colonequals && \; \frac{\DD}{\DD t}[\sigma] + \sigma \, W - W \, \sigma = Q \, \frac{\DD}{\DD t}[Q^T \, \sigma \, Q] \, Q^T \label{ZJrate01} \\
	& &&\; \textnormal{for $Q(t) \in \OO(3)$ so that} \; \dot{Q}=W\, Q \quad (\textnormal{corotational Zaremba-Jaumann derivative\footnotemark}), \notag \\
	\frac{\DD^{\Old}}{\DD t}[\sigma] &\colonequals &&\; \frac{\DD}{\DD t}[\sigma] - (L \, \sigma + \sigma \, L^T) \quad \textnormal{(non-corotational convective contravariant Oldroyd derivative \cite{oldroyd1950})}, \notag \\
	\frac{\DD^{\textnormal{Hencky}}}{\DD t}[\sigma] &\colonequals &&\; \frac{\DD}{\DD t}[\sigma] + \sigma \, W - W \, \sigma + \sigma \, \tr(D) \quad \textnormal{(non-corotational Biezino-Hencky \cite{biezeno1928} derivative,} \\
	& &&\hspace{4.5cm} \textnormal{sometimes also called Hill-rate (cf. \cite{korobeynikov2023})),} \notag \\
	\frac{\DD^{\TR}}{\DD t}[\sigma] &\colonequals &&\; \frac{\DD}{\DD t}[\sigma] - (L \, \sigma + \sigma \, L^T) + \sigma \, \tr(D) \quad \textnormal{(non-corotational Truesdell derivative \cite[eq.~3]{truesdellremarks}),} \notag \\
	\frac{\DD^{\GN}}{\DD t}[\sigma] &\colonequals &&\; \frac{\DD}{\DD t}[\sigma] + \sigma \, \Omega - \Omega \, \sigma = R \, \frac{\DD}{\DD t}[R^T \, \sigma \, R] \, R^T, \qquad R \in \OO(3), \qquad F = R \, U \notag \\
	& &&\; \textnormal{with the ``polar spin''} \, \Omega \colonequals \dot{R}(t) \, R^T(t) \in \mathfrak{so}(3) \quad \textnormal{(corotational Green-Naghdi\footnotemark derivative [cf. \cite{bellini2015}])}, \notag \\
	\frac{\DD^{\log}}{\DD t}[\sigma] &\colonequals &&\; \frac{\DD}{\DD t}[\sigma] + \sigma \, \Omega^{\log} + \Omega^{\log} \, \sigma \quad \textnormal{(corotational logarithmic derivative\footnotemark)}. \notag
	\end{alignat}
\addtocounter{footnote}{-2}\footnotetext
{
The Zaremba-Jaumann derivative is in fact a linear combination of two Lie-derivatives, see \cite[p.100, box 6.1]{Marsden83}.
}
\addtocounter{footnote}{+1}\footnotetext
{
$\Omega = \Omega^{\GN} = \dot{R}(t) \, R^T(t) \in \mathfrak{so}(3)$ is not a purely Eulerian quantity, since it depends on the choice of the reference configuration. This is true because $R(t)$ is determined by the polar decomposition of $F = \DD \varphi$, which involves the reference configuration.
}
\addtocounter{footnote}{+1}\footnotetext
{ \label{footnoteomega}
$\Omega^{\log}$ is given by
	\begin{align*}
	\Omega^{\log} \colonequals W + \sum_{i \neq j}^3 \left(\frac{1+\frac{\lambda_i}{\lambda_j}}{1 - \frac{\lambda_i}{\lambda_j}} + \frac{2}{\log\frac{\lambda_i}{\lambda_j}}\right) \, B_i \, D \, B_j = \underbrace{\Omega}_{= \dot{R} \, R^T} + \sum_{i \neq j}^3 \frac{2}{\log \frac{\lambda_i}{\lambda_j}} \, B_i \, D \, B_j \in \mathfrak{so}(3),
	\end{align*}
with $\lambda_i = \lambda_i(B)$, $B_i$ being the subordinate eigen-projections of $B$.
}

\noindent These objective time derivatives all satisfy frame-indifference in the sense that
	\begin{align}
	\frac{\DD^{\sharp}}{\DD t}[Q(t) \, \sigma(t) \, Q(t)] = Q(t) \, \frac{\DD^{\sharp}}{\DD t} [\sigma] \, Q^T(t), \qquad F \mapsto Q(t) \, F(t),
	\end{align}
while the material rate $\frac{\DD}{\DD t}[\sigma]$ does not satisfy this transformation law.
\subsubsection{From hyperelasticity to hypo-elasticity}
For isotropic hyperelastic materials, the Cauchy stress $\sigma$ can be expressed in terms of the invariants of the left Cauchy-Green deformation tensor B (or right Cauchy-Green deformation tensor). If the strain energy density function is $\WW(F) = \widehat{\WW}(I_1,I_2,I_3)$ then we can write e.g.
	\begin{align}
	\sigma(B) = 2 \sqrt{I_3} \, \frac{\partial \widehat{\WW}}{\partial I_3} \, \id + \frac{2}{\sqrt{I_3}} \left(\frac{\partial \widehat{\WW}}{\partial I_1} + I_1 \frac{\partial \widehat{\WW}}{\partial I_2}\right) \, B - \frac{2}{\sqrt{I_3}} \frac{\partial \widehat{\WW}}{\partial I_2} \, B^2 = \varphi_0 \, \id + \varphi_1 \, B + \varphi_2 \, B^2 \, ,
	\end{align}
where $I_k, \, k = 1,2,3$ denote the principal invariants of the left Cauchy-Green tensor $B\colonequals F \, F^T$. The coefficients $\varphi_0, \varphi_1, \varphi_2$ are scalar, isotropic functions of the principal invariants or other invariants of the left Cauchy-Green tensor $B$. Due to
	\begin{align}
	\frac{\DD}{\DD t}[B] = \frac{\dif}{\dif t}(F(x,t) \, F^T(x,t)) = \dot{F} \, F^{-1} \, F \, F^T + F \, F^T \, F^{-T} \, \dot{F}^T = L \, B + B \, L^T, \qquad L\colonequals \dot{F} \, F^{-1},
	\end{align}
we have
	\begin{equation}
	\begin{alignedat}{2}
	\dot{\varphi}_i &= \frac{\DD}{\DD t} \varphi_i = \langle \DD_B \varphi_i, \dot{B} \rangle = \langle \DD_B \varphi_i, (L\, B + B \, L^T) \rangle \\
	&= \langle \DD_B \varphi_i \, B, L \rangle + \langle B \, \DD_B \varphi_i, L^T \rangle \overset{\textnormal{isotropy}}{=} \langle \DD_B \varphi_i \, B, L \rangle + \langle \DD_B \varphi_i \, B, L^T \rangle = 2 \, \langle \DD_B \varphi_i \, B, D \rangle.
	\end{alignedat}
	\end{equation}
The derivative of the Cauchy stress tensor $\sigma$ with respect to time is therefore obtained as:
	\begin{equation}
	\begin{alignedat}{2}
	\frac{\DD}{\DD t} [\sigma] &= \dot{\varphi}_0 \, \id + \dot{\varphi}_1 \, B + \varphi_1 \, \dot{B} + \dot{\varphi_2} \, B^2 + \varphi_2 \, \dot{B} \, B + \varphi_2 \, B \, \dot{B} \\
	&= 2\, \langle \DD_B\varphi_0(B), D \rangle \, \id + 2 \, \langle \DD_B \varphi_1(B) \, B, D \rangle \, B + \varphi_1 \, (L\, B + B \, L^T) \\
	&\qquad + 2 \, \langle \DD_B \varphi_2(B) \, B, D \rangle \, B^2 + \varphi_2 \, (L \, B + B \, L^T) \, B + \varphi_2 \, B \, (L \, B + B \, L^T) \\
	&= 2 \, \underbrace{\{\langle \DD_B\varphi_0(B), D \rangle \, \id + \langle \DD_B \varphi_1(B) \, B, D \rangle \, B + \langle \DD_B \varphi_2(B) \, B, D \rangle \, B^2\}}_{\equalscolon A} \\
	&\qquad + \varphi_1 \, (L\, B + B \, L^T) + \varphi_2 \, (L \, B + B \, L^T) \, B + \varphi_2 \, B \, (L \, B + B \, L^T) \\
	&= 2A + \varphi_1 \, (L\, B + B \, L^T) + \varphi_2 \, (L \, B^2 + B \, L^T \, B + B \, L \, B + B^2 \, L^T) \\
	&= 2A + \varphi_1 \, (L\, B + B \, L^T) + \varphi_2 \, (L \, B^2 + B^2 \, L^T) + \varphi_2 \, (B \, L^T \, B + B \, L \, B) \\
	&= 2A + \underbrace{\varphi_1 \, (L\, B + B \, L^T) + \varphi_2 \, (L \, B^2 + B^2 \, L^T)}_{= L\, \sigma + \sigma \, L^T - 2\varphi_0 \, D} + 2\varphi_2 \, B \, D \, B \, ,
	\end{alignedat}
	\end{equation}
where we used that from $\sigma = \varphi_0 \, \id + \varphi_1 \, B + \varphi_2 \, B^2$, we also have
	\begin{equation}
	\begin{alignedat}{2}
	L\, \sigma + \sigma \, L^T &= \varphi_0 \, L + \varphi_1 \, L \, B + \varphi_2 \, L \, B^2 + \varphi_0 \, L^T + \varphi_1 \, B \, L^T + \varphi_2 \, B^2 \, L^T \\
	&= \varphi_0 \, (L + L^T) + \varphi_1 \, (L \, B + B \, L^T) + \varphi_2 \, (L \, B^2 + B^2 \, L^T) = 2 \, \varphi_0 \, D + \varphi_1 \, (L \, B + B \, L^T) + \varphi_2 \, (L \, B^2 + B^2 \, L^T)
	\end{alignedat}
	\end{equation}
and thus
	\begin{align}
	\varphi_1 \, (L \, B + B \, L^T) + \varphi_2 \, (L \, B^2 + B^2 \, L^T) = L\, \sigma + \sigma \, L^T - 2 \varphi_0 \, D.
	\end{align}
\noindent This implies that
	\begin{equation}
	\begin{alignedat}{2}
	\frac{\DD}{\DD t}[\sigma] - (L \, \sigma + \sigma \, L^T) &= 2 \, \{\langle \DD_B\varphi_0(B), D \rangle \, \id + \langle \DD_B \varphi_1(B) \, B, D \rangle \, B + \langle \DD_B \varphi_2(B) \, B, D \rangle \, B^2\} - 2 \varphi_0 \, D + 2 \varphi_2 \, B \, D \, B,
	\end{alignedat}
	\end{equation}
which means
	\begin{equation}
	\label{eqlongoldr}
	\begin{alignedat}{2}
	\frac{\DD^{\Old}}{\DD t}[\sigma] &=2 \, \{\langle \DD_B\varphi_0(B), D \rangle \, \id + \langle \DD_B \varphi_1(B) \, B, D \rangle \, B + \langle \DD_B \varphi_2(B) \, B, D \rangle \, B^2\} - 2 \varphi_0 \, D + 2 \varphi_2 \, B \, D \, B \equalscolon \H^{\Old}(B).D \, .
	\end{alignedat}
	\end{equation}
If the stress-strain relationship is invertible, i.e.~$\sigma = \mathcal{F}(B) \iff B = \mathcal{F}^{-1}(\sigma)$ we can write \newline
	$
	\frac{\DD^{\Old}}{\DD t}[\sigma] = \H^{\Old}(\mathcal{F}^{-1}(\sigma)).D \equalscolon \H^{\Old}(\sigma).D \, ,
	$
so that an isotropic, Cauchy-elastic material is also hypo-elastic in this case, as first demonstrated by Walter Noll \cite{Noll55}. Other procedures to obtain consistent rate equations from hyperelasticity are e.g.~presented in Altmeyer et al.~\cite{altmeyer2016}, Panicaud et al.~\cite{Panicaud2016} and Romenskii \cite{romenskii1974}. \\
In a final step, for the reader's convenience, we convert the Oldroyd rate further into the Zaremba-Jaumann rate.
\subsubsection{Converting the Oldroyd rate into the corotational Zaremba-Jaumann rate}
Recall that
	\begin{equation}
	\frac{\DD^{\Old}}{\DD t}[\sigma(t)] = \underbrace{\frac{\DD}{\DD t}[\sigma] - L\sigma - \sigma L^T}_{\textnormal{Oldroyd rate}} \implies \frac{\DD}{\DD t}[\sigma] = \frac{\DD^{\Old}}{\DD t} [\sigma] + L\sigma + \sigma L^T, \qquad \frac{\DD^{\ZJ}}{\DD t}[\sigma(t)] = \underbrace{\frac{\DD}{\DD t}[\sigma] + \sigma W - W \sigma}_{\textnormal{Zaremba-Jaumann rate}}.
	\end{equation}
and therefore with $L = D + W$
	\begin{equation}
	\begin{alignedat}{2}
	\frac{\DD^{\Old}}{\DD t}[\sigma(t)] &= \frac{\DD}{\DD t}[\sigma] - L\sigma - \sigma L^T = \frac{\DD}{\DD t}[\sigma] - (D+W)\sigma - \sigma(D+W)^T =\frac{\DD}{\DD t}[\sigma] - W\sigma - \sigma W^T - D\sigma - \sigma D \\
	&=\frac{\DD}{\DD t}[\sigma] + \sigma W - W\sigma - (D\sigma + \sigma D) =\frac{\DD^{\ZJ}}{\DD t}[\sigma] - (D\sigma + \sigma D).
	\end{alignedat}
	\end{equation}
Hence, from $\frac{\DD^{\Old}}{\DD t} [\sigma] = \H^{\Old}(\sigma) . D$ we get
	\begin{equation}
	\label{eqZarembaOld}
	\begin{alignedat}{2}
	\frac{\DD^{\ZJ}}{\DD t} [\sigma] - (D\sigma + \sigma D) = \H^{\Old}(\sigma) . D \quad \iff \quad \frac{\DD^{\ZJ}}{\DD t} [\sigma] &= \H^{\Old}(\sigma) . D + (D\sigma + \sigma D) \equalscolon \H^{\ZJ}(\sigma).D .
	\end{alignedat}
	\end{equation}
or equivalently, by using \eqref{eqlongoldr}
	\begin{align}
	\label{eqyoucantseeme}
	\H^{\ZJ}(\sigma) = 2 \, \{\langle \DD_B\varphi_0(B), D \rangle \, \id + \langle \DD_B \varphi_1(B) \, B, D \rangle \, B + \langle \DD_B \varphi_2(B) \, B, D \rangle \, B^2\} - 2 \varphi_0 \, D + 2 \varphi_2 \, B \, D \, B + D \, \sigma + \sigma \, D \, .
	\end{align}
Observe, that it is virtually impossible to derive any constitutive conclusion from \eqref{eqyoucantseeme} combined with $\langle \H^{\ZJ}(\sigma).D, D \rangle > 0$ $\forall \, D \in \Sym(3) \! \setminus \! \{0\}$.

It is also worth emphasizing that invertibility of e.g.~$\H^{\Old}(\sigma)$ at given $\sigma$ is in general independent of the invertibility of $\H^{\ZJ}(\sigma)$. By this we mean $\det \H^{\Old}(\sigma) \neq 0 \notiff \det \H^{\ZJ}(\sigma) \neq 0$. Similarly, the positive definiteness of $\H^{\Old}(\sigma)$ is not connected to the positive definiteness of $\H^{\ZJ}(\sigma)$.

\subsubsection{Physical inadequacy of zero-grade hypo-elasticity in simple shear for the Zaremba-Jaumann derivative} \label{appendixsimpleshear}
In a first example we recall that zero-grade hypo-elasticity results in oscillating Cauchy shear-stress when using the Zaremba-Jaumann derivative. This was first observed by Dienes \cite{dienes1979}. It is included here for the interested reader. The point to retain is that zero-grade hypo-elasticity is in general inconsistent but this inconsistency is not due to the Zaremba-Jaumann rate.
\begin{example}[Physical inadequacy of the zero-grade hypo-elasticity formulation]
Consider the Zaremba-Jaumann derivative 
	\begin{align}
	\frac{\DD^{\circ}}{\DD t}[\sigma] = \frac{\DD^{\ZJ}}{\DD t}[\sigma] = \frac{\DD}{\DD t}[\sigma] + \sigma \, W - W \, \sigma,
	\end{align}
with the assumption of zero-grade hypo-elasticity
	\begin{align}
	\H^*(\sigma).D = \C^{\iso}.D = 2 \, \mu \, D + \lambda \, \tr(D), \quad \mu,\lambda>0,\, \id,
	\end{align}
leading to the equation
	\begin{align}
	\label{eqZJmatrixcalc001}
	\frac{\DD^{\ZJ}}{\DD t}[\sigma] = \frac{\DD}{\DD t}[\sigma] + \sigma \, W - W \, \sigma = 2\mu \, D + \lambda \, \tr(D) \, \id = \C^{\iso}.D \, .
	\end{align}
Then we calculate for simple shear
	\begin{align*}
	F(t) =
	\begin{pmatrix}
	1 & \gamma \, t \\
	0 & 1
	\end{pmatrix},
	\qquad
	\dot{F}(t) =
	\begin{pmatrix}
	0 & \gamma \\
	0 & 0
	\end{pmatrix},
	\qquad
	F^{-1}(t) =
	\begin{pmatrix}
	1 & -\gamma \, t \\
	0 & 1
	\end{pmatrix},
	\end{align*}
	\begin{align*}
	L(t) = \dot{F} \, F^{-1} =
	\begin{pmatrix}
	0 & \gamma \\
	0 & 0
	\end{pmatrix},
	\qquad
	D(t) = \sym L(t) = \frac12
	\begin{pmatrix}
	0 & \gamma \\
	\gamma & 0
	\end{pmatrix},
	\qquad
	\WW(t) = \textnormal{skew} \, L(t) = \frac12
	\begin{pmatrix}
	0 & \gamma \\
	-\gamma & 0
	\end{pmatrix},
	\end{align*}
so that $\tr(D) = 0$. Now, we can determine $\sigma(t) = \begin{pmatrix} \sigma_{11} & \sigma_{12} \\ \sigma_{12} & \sigma_{22} \end{pmatrix}$ by solving the system of ordinary differential equations given by \eqref{eqZJmatrixcalc001}:
	\begin{equation}
	\begin{alignedat}{2}
	\frac{\DD}{\DD t}[\sigma] &= 2 \mu \, D + W \, \sigma - \sigma \, W \\
	\iff \qquad
	\begin{pmatrix}
	\sigma_{11}' & \sigma_{12}' \\
	\sigma_{12}' & \sigma_{22}'
	\end{pmatrix}
	&= \mu \, \gamma \,
	\begin{pmatrix}
	0 & 1 \\
	1 & 0
	\end{pmatrix}
	+ \frac12 \, \gamma \, \left(
	\begin{pmatrix}
	0 & 1 \\
	-1 & 0
	\end{pmatrix} \,
	\begin{pmatrix}
	\sigma_{11} & \sigma_{12} \\
	\sigma_{12} & \sigma_{22}
	\end{pmatrix}
	-
	\begin{pmatrix}
	\sigma_{11} & \sigma_{12} \\
	\sigma_{12} & \sigma_{22}
	\end{pmatrix} \,
	\begin{pmatrix}
	0 & 1 \\
	-1 & 0
	\end{pmatrix} \right) \\
	&= \mu \, \gamma \,
	\begin{pmatrix}
	0 & 1 \\
	1 & 0
	\end{pmatrix}
	+ \frac12 \, \gamma \, 
	\begin{pmatrix}
	2 \, \sigma_{12} & \sigma_{22} - \sigma_{11} \\
	\sigma_{22} - \sigma_{11} & -2 \, \sigma_{12}
	\end{pmatrix}.
	\end{alignedat}
	\end{equation}
This leads to the system of ordinary differential equations
	\begin{align}
	\label{eqODEsystem001}
	\sigma_{11}' = \gamma \, \sigma_{12}, \qquad \sigma_{22}' = -\gamma \, \sigma_{12}, \qquad \sigma_{12}' = \mu \, \gamma + \frac12 \, \gamma \, (\sigma_{22} - \sigma_{11}),
	\end{align}
with the initial conditions (stress-free initial state) $\sigma_{11}(0) = \sigma_{12}(0) = \sigma_{22}(0) = 0$. Adding the first two equations leads to
	\begin{align}
	(\sigma_{11} + \sigma_{22})' = 0 \qquad \implies \qquad \sigma_{11} + \sigma_{22} = \textnormal{const.} \qquad \overset{\sigma_{11}(0) = \sigma_{22}(0) = 0}{\implies} \qquad \sigma_{11}(t) = - \sigma_{22}(t).
	\end{align}
Inserting this into the third equation of \eqref{eqODEsystem001} and differentiating the result yields 
	\begin{align}
	\sigma_{12}'' = \gamma \, \sigma_{22}' = - \gamma^2 \, \sigma_{12} \qquad \implies \qquad \sigma_{12}'' + \gamma^2 \, \sigma_{12} = 0,
	\end{align}
which is well-known to have a trigonometric solution of the form $\sigma_{12}(t) = a \, \sin(\gamma \, t) + b \, \cos(\gamma \, t)$. Since $\sigma_{12}(0) = 0$, we also must have $b=0$, so that the general solution is given by $\sigma_{12}(t) = a \, \sin(\gamma \, t)$ for arbitrary $a \in \R$. This shows that the assumption of \textbf{zero-grade hypo-elasticity} together with the Zaremba-Jaumann derivative is \textbf{in general not sound}, since it produces oscillating Cauchy shear stress for monotone increasing simple shear.\footnotemark $\hfill{\square}$
\footnotetext
{
Several authors have nevertheless insisted on using  zero-grade hypo-elasticity by modifying the formulation. More precisely, they consider the Kirchhoff stress $\tau$ and propose to use an adapted corotational rate, the so called logarithmic rate \cite{xiao97, xiao98_2,xiao1999existence},  in order to obtain an integrable hypo-elastic formulation, leading exclusively to the Hencky-energy \cite{Neff_Osterbrink_Martin_Hencky13}. Contrary to this, in our framework, there is no immediate need to depart from the Zaremba-Jaumann rate since we start from hyperelasticity or Cauchy-elasticity and consider solely the induced tangent stiffness tensor $\H^{\ZJ}(\sigma)$ which is consistent by definition.
}
\end{example}
\subsubsection{Logarithmic monotonicity of the Cauchy stress $\sigma$ for volumetric functions} \label{appvolufunc}
In this section we derive the criterion for the monotonicity of the Cauchy stress $\sigma(V)$ in $\log V$ for a ``perfectly compressible fluid'', i.e.~we show that $\sigma(V)$ satisfies
	\begin{align}
	\label{eqmonotonicityinlog01}
	\langle \sigma(V_1) - \sigma(V_2), \log V_1 - \log V_2 \rangle > 0 \qquad \forall \, V_1 \neq V_2, \qquad V_1, V_2 \in \Sym^{++}(3),
	\end{align}
if the criterion is fulfilled. For such a perfect fluid, the energy function $\WW$ has the structure
	\begin{align}
	\WW(F) = \WW_{\vol}(\det F) = h(\det F)
	\end{align}
with a (smooth) function $h: \R^+ \to \R$. 
\begin{prop}
Let $\WW(F) = \WW_{\vol}(\det F) = h(\det F) = \widehat{h}(\log \det F)$. Then the corresponding Cauchy stress $\sigma_{\vol}$ satisfies the relation
	\begin{align}
	\sigma_{\textnormal{vol}} \; \textnormal{monotone in} \; \log V \quad \iff \quad \widehat{h}'' - \widehat{h}' \ge 0 \quad \iff \quad h \; \textnormal{convex in} \; \det F.
	\end{align}
\end{prop}
\begin{proof}
Since $F = V \, R$ with $R \in \OO(3)$, we have $\det F = \det V$ so that $\WW_{\vol}(\det F) = \WW_{\vol}(\det V)$. Introducing $\widehat{\WW}_{\vol}(\log \det V) \colonequals \WW_{\vol}(\det V)$ and $\widehat{h}(\log \det V) \colonequals h(\det V)$ respectively, we can derive $\sigma_{\vol}(V)$ by the Richter formula \cite{richter1948isotrope,vallee1978lois}
	\begin{align}
	\sigma(V) = \sigma_{\vol}(\det V) = \frac{1}{\det V} \, \tau(\det V) = \frac{1}{\det V} \, \DD_{\log V} \widehat{\WW}_{\vol}(\log \det V).
	\end{align}
Since $\tr(\log V) = \log \det V$ and $\DD_X \tr(X) = \id$, we can calculate the derivative $\DD_{\log V} \widehat{\WW}_{\vol}(\log \det V)$ explicitly by using $\widehat{\WW}_{\vol}(\log \det V) = \widehat{h}(\log \det V) = \widehat{h}(\tr(\log V))$, which yields
	\begin{align}
	\sigma_{\vol}(\det V) = \frac{1}{\det V} \, \widehat{h}'(\tr(\log V)) \, \id.
	\end{align}
In a final step we set $\widehat{\xi} \colonequals \tr(\log V)$, introduce $\widehat{\sigma}_{\vol}(\log \det V)\colonequals \sigma_{\vol}(\det V)$ and observe
	\begin{align}
	\label{eqmonotonicityinlog02}
	\widehat{\sigma}_{\vol}(\log \det V) = \widehat{\sigma}_{\vol}(\tr(\log V)) = \widehat{\sigma}_{\vol}(\widehat{\xi}) = \underbrace{\frac{1}{\mathrm{e}^{\tr(\log V)}} \, \widehat{h}'(\tr(\log V))}_{= \; \widehat{h}'(\widehat{\xi}) \, \mathrm{e}^{-\widehat{\xi}} \; \equalscolon \; g(\widehat \xi)} \, \id.
	\end{align}
With the help of \eqref{eqmonotonicityinlog02} we can reformulate the left hand side of \eqref{eqmonotonicityinlog01} (with $\widehat \xi_i = \tr(\log V_i)$, $i = 1, 2$)
	\begin{align}
	\langle \sigma(V_1) - \sigma(V_2), \log V_1 - \log V_2 \rangle = \langle ((g(\widehat \xi_1) - g(\widehat \xi_2)) \, \id, \log V_1 - \log V_2 \rangle = (g(\widehat \xi_1) - g(\widehat \xi_2)) \, (\widehat \xi_1 - \widehat \xi_2),
	\end{align}
showing that the monotonicity of $\sigma(V) = \sigma_{\vol}(\det V) = \widehat{\sigma}_{\vol}(\tr(\log V))$ in $\log V$ is equivalent to the monotonicity of the scalar function $g(\widehat \xi) = \widehat{h}'(\widehat{\xi}) \, \mathrm{e}^{-\widehat{\xi}}$ with respect to $\widehat{\xi} \in \R$. Calculating the derivative with respect to $\widehat{\xi}$ results in the condition
	\begin{align}
	\frac{\dif}{\dif \widehat{\xi}} \left( \widehat{h}'(\widehat{\xi}) \, \mathrm{e}^{-\widehat \xi} \right) = (\widehat h''(\widehat \xi) - \widehat h'(\widehat \xi)) \, \mathrm{e}^{-\widehat \xi} \ge 0 \qquad \iff \qquad \widehat h''(\widehat \xi) - \widehat h'(\widehat \xi) \ge 0.
	\end{align}
From this we can also derive a criterion for $h(\det V) = \widehat h(\log \det V) = \widehat h(\widehat \xi)$. Setting $x \colonequals \det V$, we obtain
	\begin{equation}
	\begin{alignedat}{2}
	h(x) &= \widehat h(\log x) \qquad \iff \qquad \frac{\dif}{\dif x} h(x) = h'(x) = \frac{\dif}{\dif x} \widehat h(\log x) = \widehat h'(\log x) \, \frac1x \\
	\iff \qquad h''(x) &= \frac{\dif}{\dif x}\left(\widehat h'(\log x) \, \frac1x\right) = \widehat h''(\log x) \, \frac{1}{x^2} - \widehat h' \, \frac{1}{x^2} = \frac{1}{x^2} \, (\widehat h''(\log x) - \widehat h'(\log x)).
	\end{alignedat}
	\end{equation}
Finally, inserting $x = \det V$ we see that
	\begin{align}
	\widehat h''(\widehat \xi) - \widehat h'(\widehat \xi) &\ge 0 \qquad \iff \qquad h''(\det V) \ge 0 \qedhere
	\end{align}
\end{proof}
\begin{rem}
Calculating the trace
	\begin{align}
	\tr(\sigma_{\vol}(\det V)) = \frac{3}{\det V} \, \widehat h'(\tr(\log V)) = 3 \, \widehat h'(\widehat \xi) \, \mathrm{e}^{-\widehat \xi},
	\end{align}
we additionally see that for the mean pressure $\frac13 \tr(\sigma_{\vol})$ the relation
	\begin{center}
	\fbox{\textnormal{$\frac13 \tr(\sigma_{\textnormal{vol}})$ invertible \quad $\overset{\substack{\DD_B \sigma_{\vol}(\id) \in \Sym^{++}_4(6)\\\left.\right.}}{\iff}$ \quad $\frac13 \tr(\sigma_{\textnormal{vol}})$ monotone \quad $\iff$ \quad $h$ convex in $\det V$}}
	\end{center}
holds true.
	\begin{figure}[h!]
		\begin{center}
		\begin{minipage}[h!]{0.8\linewidth}
			\centering
			\includegraphics[scale=0.5]{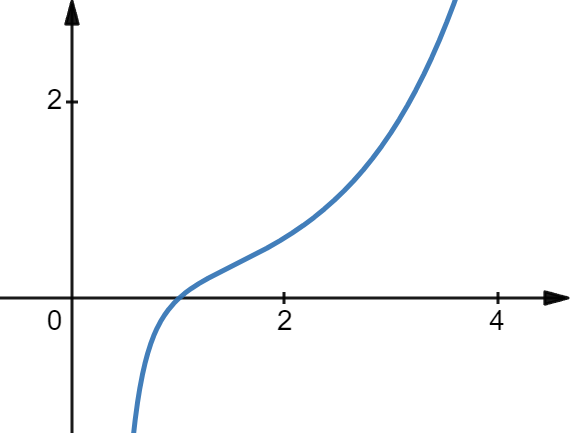}
			\put(-20,40){$\alpha$}
			\put(-145,40){$1$}
			\put(-180,160){$\frac{1}{3}\tr(\sigma(\alpha \,\id))$}
			\centering
			\caption{\footnotesize{Monotone volume-mean pressure $\alpha \mapsto \frac{1}{3}\tr(\sigma(\alpha \,\id))$. It is physically reasonable to expect that the mean pressure is an increasing function of the volumetric stretch for a stable elastic material. }}
			\label{graphic-dupa-k}
		\end{minipage}
	\end{center}
	\end{figure}
\end{rem}
\begin{rem}
The condition of Legendre-Hadamard-ellipticity and monotonicity of $\sigma_{\vol}$ in $\log V$ coincide for a perfectly compressible fluid, i.e.~in $\WW_{\vol}(F) = h(\det F)$ the function $h$ must be convex, ensuring also the polyconvexity \cite{Ball77} of $\WW_{\vol}$, see also Appendix~\ref{Appenhdet}. This was already established by Leblond \cite{leblond1992constitutive}. However, in general, polyconvexity and the corotational stability postulate (CSP) do not coincide as shown in \cite{NeffGhibaLankeit}, see also the example concerning the exponential Hencky energy herein \cite[Section 4]{NeffGhibaLankeit}.
\end{rem}
\subsubsection{Legendre-Hadamard-ellipticity for  functions of the type $F\mapsto h(\det F)$}
\label{Appenhdet}
This subsection is also taken from \cite{NeffGhibaLankeit}. \\
\\
We consider a function $h:\R\to\R$ and analyze when the function $F\mapsto h(\det F)$ is LH-elliptic as a function of $F\in \R^{3\times 3}$.  We recall that
	\begin{align}
	\DD(\det F).H=\det F\, \tr(H\, F^{-1})=\langle \Cof F,H\rangle.
	\end{align}
Using the first Frech\'{e}t-formal derivative, we compute
	\begin{align}
	\DD(h(\det F)).(H,H)=h^{\prime}(\det F)\, \langle \Cof F, H\rangle
	\end{align}
and the second derivative is given by
	\begin{align}
	\DD^2(h(\det F)).(H,H)&=h^{\prime\prime}(\det F)\, \langle \Cof F, H\rangle^2+h^{\prime}(\det F)\langle \DD(\Cof F).H, H\rangle\\
	&=h^{\prime\prime}(\det F)\, \langle \Cof F, H\rangle^2+h^{\prime}(\det F) \{\langle \langle \Cof F,H\rangle \, F^{-T},H\rangle+\det F\langle -F^{-T}H^T
	F^{-T},H\rangle\},\notag
	\\
	&=h^{\prime\prime}(\det F)\, \langle \Cof F, H\rangle^2+h^{\prime} (\det F)\, \det F \{\langle  F^{-T},H\rangle^2-\langle F^{-T}H^T
	F^{-T},H\rangle\}.\notag
	\end{align}
Hence, for $\xi,\eta\in \R^3$ we have
	\begin{align}
	\DD^2(h(&\det F)).((\xi\otimes\eta),(\xi\otimes\eta)) \\
	&=h^{\prime\prime}(\det F)\, \langle \Cof F, (\xi\otimes\eta)\rangle^2+h^{\prime} (\det F)\, \det F \{\langle  F^{-T},(\xi\otimes\eta)\rangle^2-\langle
	F^{-T}(\xi\otimes\eta)^T F^{-T},(\xi\otimes\eta)\rangle\}.\notag
	\end{align}
Surprisingly, we also have the simplification
	\begin{align}
	\langle  F^{-T},(\xi\otimes\eta)\rangle^2-&\langle F^{-T}(\xi\otimes\eta)^T F^{-T},(\xi\otimes\eta)\rangle=
	\langle \id, F^{-1}(\xi\otimes\eta)\rangle^2-\langle (\eta\otimes F^{-1}\xi) ,(F^{-1}\xi\otimes\eta)\rangle\notag\\
	&=
	\langle \id, F^{-1}(\xi\otimes\eta)\rangle^2-\langle (F^{-1}\xi\otimes\eta)^T ,(F^{-1}\xi\otimes\eta)\rangle
	\notag=
	\langle  F^{-1}\xi,\eta\rangle^2-\langle F^{-1}\xi,\eta \rangle^2=0,
	\end{align}
leading to
	\begin{align}
	\DD^2(h(\det F)).&(\xi\otimes\eta,\xi\otimes\eta)=h^{\prime\prime}(\det F)\, \langle \Cof F, (\xi\otimes\eta)\rangle^2.
	\end{align}
In conclusion, $F\mapsto h(\det F)$ is LH-elliptic  if and only if $t\mapsto h(t)$ is convex since $\langle \Cof F, (\xi\otimes\eta)\rangle^2$ is
positive. \\
\\
From \cite[page 213]{Dacorogna08} we additionally have for elastic fulids
\begin{prop}\label{propDacdet}
For $h:\mathbb{R}\rightarrow\mathbb{R}$, let $\WW:\mathbb{R}^{n\times n}\rightarrow\mathbb{R}$ be given by $\WW(F)=h(\det F).$ Then
    \begin{align}\hspace{-4mm}
      \WW \quad \textnormal{polyconvex}\quad  \iff \quad \WW \quad \textnormal{quasiconvex}\quad \iff \quad  \WW \quad \textnormal{rank one convex}\quad \iff \quad h \quad \textnormal{convex}.
    \end{align}
   \end{prop}
\subsubsection{Hypo-elasticity for the perfect elastic fluid} \label{apphypcomp}
We will demonstrate by a direct calculation that $\H^{\ZJ}(\sigma)$ is positive definite if and only if $h(x)$ is a convex function. \\
\\
Since $\det V = \sqrt{\det B}$, we have $\widetilde{\WW}(B) = h(\sqrt{\det B})$ leading to the Cauchy stress
	\begin{equation}
	\sigma(B) = \frac{2}{J} \, \DD_B \widetilde{\WW}(B) \, B = \frac{2}{\sqrt{\det B}} \, h'(\sqrt{\det B}) \, \frac{1}{2 \, \sqrt{\det B}} \, \det B \, B^{-1} \, B = h'(\sqrt{\det B}) \, \id.
	\end{equation}
Knowing $\sigma(B)$ we can calculate $\frac{\DD^{\ZJ}}{\DD t}[\sigma]$ according to formula \eqref{eqratetype5} from Section \ref{sechypoelastic} and we obtain
	\begin{equation}
	\begin{alignedat}{2}
	\H^{\ZJ}(\sigma).D = \frac{\DD^{\ZJ}}{\DD t}[\sigma] &= \DD_B \sigma(B).[B \, D + B \, D] = \langle \frac12 \, h''(\sqrt{\det B}) \, \sqrt{\det B} \, B^{-1}, B \, D + D \, B \rangle \, \id \\
	&= h''(\sqrt{\det B}) \, \sqrt{\det B} \, \tr(D) \, \id.
	\end{alignedat}
	\end{equation}
Equivalently, we can use \eqref{eqrateGN01} to calculate $\frac{\DD^{\GN}}{\DD t}[\sigma]$ from $\sigma(B)$ with $B = V^2$
	\begin{equation}
	\begin{alignedat}{2}
	\H^{\GN}(\sigma).D = \frac{\DD^{\GN}}{\DD t}[\sigma] &= 2 \, \DD_B \sigma(B).[V \, D \, V] = 2 \, \langle \frac12 \, h''(\sqrt{\det B}) \, \sqrt{\det B} \, B^{-1}, V \, D \, V \rangle \, \id \\
	&= h''(\sqrt{\det B}) \, \sqrt{\det B} \, \tr(D) \, \id,
	\end{alignedat}
	\end{equation}
showing that $\H^{\GN}(\sigma).D = \H^{\ZJ}(\sigma).D$ for the perfect elastic fluid and in fact the same representation is true for an arbitrary corotational rate $\H^{\circ}(\sigma).D$.
\begin{rem}
If, in general, we have the Cauchy stress $\sigma(B)$ given in the form $\sigma(B) = \alpha(B) \, \id$, where $\alpha: \Sym^{++}(3) \to \R$ is a scalar-valued isotropic function, we obtain (with $\frac{\DD}{\DD t}[B] = L \, B + B \, L^T$) for any corotational rate
	\begin{equation}
	\begin{alignedat}{2}
	\frac{\DD^{\circ}}{\DD t}[\sigma(B)] &= \frac{\DD}{\DD t}[\sigma(B)] = \langle \DD \alpha(B), \frac{\DD}{\DD t}[B] \rangle \, \id = \langle \DD \alpha(B) , (L \, B + B \, L^T) \rangle \, \id \\
	&= \langle \DD \alpha(B) \, B, L \rangle + \langle B \, \DD \alpha(B), L^T \rangle \, \id \overset{\alpha \, \textnormal{isotropic}}{=} \langle \DD \alpha(B) \, B, L \rangle + \langle \DD \alpha(B) \, B, L^T \rangle \, \id \\
	&= \langle \DD \alpha(B) \, B, L + L^T \rangle \, \id = 2 \, \langle \DD \alpha(B) \, B, D \rangle \, \id = \H^{\circ}(\sigma).D \, .
	\end{alignedat}
	\end{equation}
\end{rem}
\subsubsection{The Schur product}
We recall \cite{horn2013} that the Schur (or Hadamard) product of two matrices $A, B \in \C^{m \times n}$ is denoted by $(A \circ B) \in \C^{m \times n}$ and it is defined entrywise as $(A \circ B)_{ij} = A_{ij} \, B_{ij}$. This product will be used for the Daleckii-Krein formula in the upcoming section.
\begin{lem}[Properties of the Schur product]
We have the following properties:
	\begin{enumerate}
	\item symmetry: Let $A, B \in \R^{n \times n}$, then
		\begin{align}
		A \circ B = B \circ A, \qquad \textnormal{and} \qquad (A \circ B)^T = A^T \circ B^T \, .
		\end{align}
	\item distributive property: Let $A, B, C \in \R^{n \times n}$, then
		\begin{align}
		A \circ (B + C) = A \circ B + A \circ C, \qquad \textnormal{and} \qquad (A+B) \circ C = A \circ C + B \circ C \, .
		\end{align}
	\item positive definiteness: Let $A, B \in \Sym^+(n)$, then $A \circ B \in \Sym^+(n)$ and
		\begin{align}
		\lambda_{\min}(A \circ B) \ge \lambda_{\min}(B) \, \underset{1 \le i \le n}{\min} \, a_{ii} \, .
		\end{align}
	\item Let $\diag \in \R^{n \times n}$ be a diagonal matrix and $A, B \in \R^{n \times n}$, then
		\begin{equation}
		\begin{alignedat}{2}
		\diag \, (A \circ B) &= (\diag \, A) \circ B = A \circ (\diag B), \qquad (A \circ B) \, \diag &= (A \, \diag) \circ B = A \circ (B \, \diag) 
		\end{alignedat}
		\end{equation}
		and
		\begin{align}
		\diag \, (A \circ B) \, \diag &= (\diag \, A \, \diag) \circ B = A \circ (\diag \, B \, \diag), \\
		\diag \, (A \circ B) \, \diag &= (\diag \, A) \circ (B \, \diag) = (A \, \diag) \circ (\diag \, B) \, .
		\end{align}
	\item Let $A, B, \mathcal{L} \in \R^{n\times n}$, then
		\begin{align}
		\langle \mathcal{L} \circ A , A \rangle &= \sum_{i,j = 1}^n (\mathcal{L}_{ij} \, A_{ij}) \, A_{ij} \ge \sum_{i,j=1}^n \underset{i,j}{\min}(\mathcal{L}_{ij}) |A_{ij}|^2 = \underset{i,j}{\min}(\mathcal{L}_{ij}) \, \norm{A}^2, \\
		\langle \mathcal{L} \circ A , B \rangle &= \sum_{i,j=1}^n (\mathcal{L}_{ij} \, A_{ij}) \, B_{ij} = \sum_{i,j=1}^n A_{ij} \, (\mathcal{L}_{ij} \, B_{ij}) = \langle A, \mathcal{L} \circ B \rangle.
		\end{align}
	\item Let $R \in \OO(3)$ and $\mathcal{F}$ be symmetric $(\mathcal{F} \in \Sym(3))$. Then for all $\DD_1, \DD_2 \in \Sym(3)$ we have
		\begin{align}
		\langle R \, \bigg[\mathcal{F} &\circ [\diag(\lambda_1, \lambda_2, \lambda_3) \, \underbrace{R^T \, \DD_1 \, R}_{S_1} \, \diag(\lambda_1, \lambda_2, \lambda_3)]\bigg] \, R^T, \DD_2 \rangle \notag \\
		= \, &\langle \mathcal{F} \circ [\diag(\lambda_1, \lambda_2, \lambda_3) \, S_1 \, \diag(\lambda_1, \lambda_2, \lambda_3)], \underbrace{R^T \, \DD_2 \, R}_{S_2} \rangle \\
		= \, &\langle \mathcal{F} \circ [\diag(\lambda_1, \lambda_2, \lambda_3) \, S_1 \, \diag(\lambda_1, \lambda_2, \lambda_3)], S_2 \rangle = \langle \diag(\lambda_1, \lambda_2, \lambda_3) \, S_1 \, \diag(\lambda_1, \lambda_2, \lambda_3), \mathcal{F} \circ S_2 \rangle \notag \\
		= \, &\langle S_1, \diag(\lambda_1, \lambda_2, \lambda_3) \, (\mathcal{F} \circ S_2) \, \diag(\lambda_1, \lambda_2, \lambda_3) \rangle = \langle S_1, \mathcal{F} \circ [\diag(\lambda_1, \lambda_2, \lambda_3) \, S_2 \, \diag(\lambda_1, \lambda_2, \lambda_3)] \rangle \notag
		\end{align}
	\end{enumerate}
\end{lem}
\subsubsection{The Daleckii-Krein formula}
The upcoming formula can be used to explicitly determine the Fréchet derivative of a primary matrix function $f(X)$ (see Appendix~\ref{appendixnotation} for a definition) at $X = A$, applied to a perturbation $H$.
	\begin{thm} \label{Daleckii-KreinThm} \label{thmD-K}
	(\textbf{Daleckii-Krein} \cite{daleckii1965}). Let $A = Z \, D \, Z^{-1} \in \C^{n \times n}$ be a diagonalizable matrix, with $D$ diagonal, and let $f$ be continously differentiable on the spectrum of $A$. Then the Fréchet derivative of the primary matrix function $f(X)$ at $X = A$, applied to the perturbation $H$, is equal to
		\begin{align}
		\DD_X f(A) . H = Z \, (\mathcal{F} \circ (Z^{-1} \, H \, Z)) \, Z^{-1}
		\end{align}
	where the symbol $\circ$ denotes the Schur product and the matrix $\mathcal{F} \in \C^{n \times n}$ is defined as
		\begin{align}
		\mathcal{F}_{ij} = \left\{
			\begin{array}{ll}
			f'(\DD_{ii}) &\qquad \textnormal{if} \quad i = j, \\
			\frac{f(\DD_{ii}) - f(\DD_{jj})}{\DD_{ii} - \DD_{jj}} &\qquad \textnormal{if} \quad i \neq j.
			\end{array}
		\right.
		\end{align}
	\end{thm}
	\begin{cor}
	Let $V = R \, \diag(\lambda_1, \lambda_2, \lambda_3) \, R^T \in \Sym^{++}(3)$ with $R \in \OO(3)$ the orthogonal matrix whose columns are the eigenvectors of $V$. Then for the increment $H \in \Sym(3)$ we have the representation of the derivative
		\begin{align}
		\DD_V \log V . H = R \, [\mathcal{F} \circ (R^T \, H \, R)] \, R^T, \qquad \textnormal{with} \qquad \mathcal{F}_{ij} = \left\{
			\begin{array}{ll}
			\frac{1}{\lambda_i} &\qquad \textnormal{if} \quad i = j, \\
			\frac{\log \lambda_i - \log \lambda_j}{\lambda_i - \lambda_j} &\qquad \textnormal{if} \quad i \neq j.
			\end{array}
		\right.
		\end{align}
	\end{cor}
	\begin{cor} \label{appendixcor09}
	Let $B = R \, \diag(\lambda_1^2, \lambda_2^2, \lambda_3^2) \, R^T \in \Sym^{++}(3)$ with $R \in \OO(3)$ the orthogonal matrix whose columns are the eigenvectors of $B$. Then for the increment $H \in \Sym(3)$ we have the representation of the derivative
		\begin{align}
		\DD_B \log B . H = R \, [\mathcal{F} \circ (R^T \, H \, R)] \, R^T, \qquad \textnormal{with} \qquad \mathcal{F}_{ij} = \left\{
			\begin{array}{ll}
			\frac{1}{\lambda_i^2} &\qquad \textnormal{if} \quad i = j, \\
			\frac{\log \lambda_i^2 - \log \lambda_j^2}{\lambda_i^2 - \lambda_j^2} &\qquad \textnormal{if} \quad i \neq j.
			\end{array}
		\right.
		\end{align}
	\end{cor}
	\begin{cor}
	\label{dopeasscor001}
	Let $D \in \Sym(3), \, B = R \, \diag(\lambda_1^2, \lambda_2^2, \lambda_3^2) \, R^T \in \Sym^{++}(3)$ with $R \in \OO(3)$ the orthogonal matrix whose columns are the eigenvectors of $B$. Set $X\colonequals \diag(\lambda_1^2, \lambda_2^2, \lambda_3^2)$ and let $\mathcal{F}$ be given as in the previous Corollary, i.e.
		\begin{align}
		\mathcal{F}_{ij} = \left\{
			\begin{array}{ll}
			\frac{1}{\lambda_i^2} &\qquad \textnormal{if} \quad i = j, \\
			\frac{\log \lambda_i^2 - \log \lambda_j^2}{\lambda_i^2 - \lambda_j^2} &\qquad \textnormal{if} \quad i \neq j.
			\end{array}
		\right.
		\end{align}
	Then we have the following alternative representation for $\DD_B \log B . [B \, D + D \, B]:$
		\begin{equation}
		\label{eqappendixowncalc001}
		\begin{alignedat}{2}
		\DD_B \log B . [B \, D + D \, B] &= R \, \big[\mathcal{F} \circ \{R^T \, [B \, D + D \, B] \, R)\}\big] \, R^T \\
		&= R \, \big[ \mathcal{F} \circ \{ X \, \underbrace{R^T \, D \, R}_{S} + \underbrace{R^T \, D \, R}_S \, X\} \big] \, R^T = R \, \big[ \mathcal{F} \circ \{ X \, S + S \, X\} \big] \, R^T \\
		&= R \, \big[ \mathcal{F} \circ (X \, S) + \mathcal{F} \circ (S \, X) \big] \, R^T = R \, \big[ (X \, \mathcal{F}) \circ S + (\mathcal{F} \, X) \circ S \big] \, R^T \\
		&= R \, \big[ \underbrace{(X \, \mathcal{F} + \mathcal{F} \, X)}_{= \, \mathcal{L} \in \Sym(3)} \circ (R^T \, D \, R) \big] \, R^T = R \, \big[ \mathcal{L} \circ (R^T \, D \, R) \big] \, R^T.
		\end{alignedat}
		\end{equation}
	Furthermore, by the monotonicity of $\log(\,)$, every entry of $\mathcal{F}$ and thus also every entry of $\mathcal{L} = X \, \mathcal{F} + \mathcal{F} \, X$, where $X$ was given by $X= \diag(\lambda_1^2, \lambda_2^2, \lambda_3^2)$, is non-negative.
	\end{cor}
\noindent We may immediately use Corollary \ref{dopeasscor001} to prove
	\begin{lem}
	\label{dopeasslemma001}
	For $D \in \Sym(3), \, B \in \Sym^{++}(3)$ we have (while $\DD_B \log B. H \neq B^{-1} \, H$ in general)
		\begin{align}
		\langle \DD_B \log B . [B \, D + D \ B] , D \rangle > c^+(B) \, \norm{D}^2.
		\end{align}
	\end{lem}
	\begin{proof}
	From Corollary \ref{dopeasscor001} we have
		\begin{align}
		\langle \DD_B \log B . [B \, D + D \ B] , D \rangle &= \langle  R \, \big[ \mathcal{L} \circ (R^T \, D \, R) \big] \, R^T, D \rangle = \langle \mathcal{L} \circ (R^T \, D \, R) , (R^T \, D \, R) \rangle \\
		&\ge \underset{i,j}{\min}(\mathcal{L}_{ij}) \, \norm{R^T \, D \, R}^2 = c^+(\mathcal{L}) \, \norm{D}^2 \, . \notag \qedhere
		\end{align}
	\end{proof}
\noindent As a side-product of these considerations we also obtain
	\begin{lem}	[$\log$ is strongly Hilbert-monotone] \label{lemhilbertmon1}
	$\DD_B \log B \in \Sym^{++}_4(6)$, i.e.
		\begin{align}
		\langle \DD_B \log B . H , H \rangle > 0 \qquad \forall \, H \in \Sym(3) \setminus \{0\}.
		\end{align}
	\end{lem}
	\begin{proof} 
	Since by Corollary \ref{dopeasscor001} every entry of $\mathcal{F}$ is non-negative, we obtain, by using the Daleckii-Krein formula,
		\begin{align}
		\DD_B \log B . H &= R \, [\mathcal{F} \circ (R^T \, H \, R)] \, R^T \\
		\implies \quad \langle \DD_B \log B . H , H \rangle &= \langle \mathcal{F} \circ (R^T \, H \, R) , (R^T \, H \, R) \rangle \ge \underset{i,j}{\min}(\mathcal{F}_{i,j}) \, \norm{R^T \, H \, R}^2 = c^+ \, \norm{H}^2. \notag \qedhere
		\end{align}
	\end{proof}
	\begin{cor}
	The previous Lemma additionally shows that $\det \DD_B \log B > 0$ and $[\DD_B \log B]^{-1} \in \Sym^{++}_4(6)$.
	\end{cor}
	\begin{lem}
	$D \mapsto \DD_B \log B . [B \, D + D \, B]$ is self-adjoint (major symmetric), i.e. for $\DD_1, \DD_2 \in \Sym(3)$
	\begin{align}
	\langle \DD_B \log B . [B \, \DD_1 + \DD_1 \, B], \DD_2 \rangle = \langle \DD_B \log B . [B \, \DD_2 + \DD_2 \, B], \DD_1 \rangle \qquad \textnormal{holds.}
	\end{align}
	\end{lem}
	\begin{proof}
	We make use of Lemma \ref{dopeasslemma001}, which yields
		\begin{align}
		\langle \DD_B \log B . [B \, \DD_1 &+ \DD_1 \, B], \DD_2 \rangle \notag \\
		&= \langle R \, \big[ \mathcal{L} \circ \underbrace{(R^T \, \DD_1 \, R)}_{S_1} \big] \, R^T, \DD_2 \rangle = \langle \mathcal{L} \circ S_1 , \underbrace{R^T \, \DD_2 \, R}_{S_2} \rangle = \langle S_1 , \mathcal{L} \circ S_2 \rangle \overset{\mathcal{L}, S_1, S_2 \in \Sym(3)} = \langle \mathcal{L} \circ S_2 , S_1 \rangle \\
		&= \langle \mathcal{L} \circ (R^T \, \DD_2 \, R) , R^T \, \DD_1 \, R \rangle = \langle R \, \big[ \mathcal{L} \circ (R^T \, \DD_2 \, R) \big] \, R^T, \DD_1 \rangle = \langle \DD_B \log B . [B \, \DD_2 + \DD_2 \, B], \DD_1 \rangle \notag \qedhere
		\end{align}
	\end{proof}
	\begin{rem}
		Note that in general, $\DD_Bf(B)$ is self adjoint for any primary matrix function $f$, which includes the case $f=\log$, since any such function admits a potential \cite{NeffMartin14}. Moreover, for $f: \R^+ \to \R$ strongly monotone, $\DD_B f(B) \in \Sym^{++}_4(6)$.
	\end{rem}
	Let us choose $D = \alpha \, \id$ and calculate
	\begin{equation}
	\begin{alignedat}{2}
	\DD_B \log_B . [2 \, B \, \alpha \, \id] &= R \, \big[\mathcal{F} \circ \{2 \, \diag(\lambda_1^2, \lambda_2^2, \lambda_3^3) \, \alpha \, \id \} \big] \, R^T = 2 \, R \,
	\begin{pmatrix}
	\frac{1}{\lambda_1^2} \, \lambda_1^2 \, \alpha & 0 & 0 \\
	0 & \frac{1}{\lambda_2^2} \, \lambda_2^2 \, \alpha & 0 \\
	0 & 0 & \frac{1}{\lambda_3^2} \, \lambda_3^2 \, \alpha
	\end{pmatrix}
	\, R^T = 2 \, \alpha \, \id \, .
	\end{alignedat}
	\end{equation}
Hence
	\begin{align}
	\DD_B \log B . [B \, D + D \, B] \big\vert_{D = \alpha \, \id} = 2 \, D \big\vert_{D = \alpha \, \id} \qquad \textnormal{(where $B$ is arbitrary).}
	\end{align}
The latter result is clear since $B \, D = D \, B$ in this case ($\exp(B + D) = \exp(B) \, \exp(D)$ if $B \, D = D \, B)$ \\
$\DD_B \exp(B) . D = \exp(B) . D$ (Taylor expansion). \\
\\
Hence, the question arises, if in general
	\begin{align}
	\DD_B \log B . [B \, D + D \, B] = 2 \, D \, ?
	\end{align}
From \eqref{eqappendixowncalc001} we must have
	\begin{align}
	\mathcal{F} \circ (\diag(\lambda_1^2, \lambda_2^2, \lambda_3^2) \, R^T \, D \, R + R^T \, D \, R \, \diag(\lambda_1^2, \lambda_2^2, \lambda_3^2)) = 2 \, \underbrace{R^T \, D \, R}_{S \in \Sym(3)}
	\end{align}
so that, by setting $S = R^T \, D \, R$, we need to have
	\begin{align}
	\mathcal{F} \circ (\diag(\lambda_1^2, \lambda_2^2, \lambda_3^2) \, S + S \, \diag(\lambda_1^2, \lambda_2^2, \lambda_3^2)) \overset{?}{=} 2 \, S \qquad \forall \, S \in \Sym(3).
	\end{align}
Case $n=2$, $S = \begin{pmatrix} a & b \\ b & c \end{pmatrix}$. Then
	\begin{equation}
	\begin{alignedat}{2}
	\mathcal{F} \circ (\diag(\lambda_1^2, \lambda_2^2) \, S &+ S \, \diag(\lambda_1^2, \lambda_2^2)) \\
	&= 
	\begin{pmatrix}
	\frac{1}{\lambda_1^2} & \frac{2(\log \lambda_1 - \log \lambda_2)}{\lambda_1^2 - \lambda_2^2} \\
	\frac{2(\log \lambda_1 - \log \lambda_2)}{\lambda_1^2 - \lambda_2^2} & \frac{1}{\lambda_2^2}
	\end{pmatrix}
	\circ
	\begin{pmatrix}
	2 \, \lambda_1^2 \ a & (\lambda_1^2 + \lambda_2^2) \, b \\
	(\lambda_1^2 + \lambda_2^2) \, b & 2 \, \lambda_2^2 \, c
	\end{pmatrix} \\
	&=
	\begin{pmatrix}
	2a & \frac{2(\log \lambda_1 - \log \lambda_2)}{\lambda_1^2 - \lambda_2^2} \, (\lambda_1^2 + \lambda_2^2) \, b \\
	\frac{2(\log \lambda_1 - \log \lambda_2)}{\lambda_1^2 - \lambda_2^2} \, (\lambda_1^2 + \lambda_2^2) \, b & 2c
	\end{pmatrix}
	\neq
	\begin{pmatrix}
	2a & 2b \\
	2b & 2c
	\end{pmatrix}
	= 2S,
	\end{alignedat}
	\end{equation}
leading to the negative conclusion
	\begin{align}
	\label{eqboxed2D}
	\boxed{\DD_B \log B . [B \, D + D \, B] \neq 2 \, D \qquad \textnormal{in general.}}
	\end{align}
However, if $B \, D = D \, B$, we obtain
\begin{prop} \label{propA17}
Under the assumption $B \, D = D \, B$ we have
	\begin{align}
	\label{eqboxed3D}
	\boxed{B \, D = D \, B \qquad \implies \qquad \DD_B \log B.[B \, D + D \, B] = 2 \, D.}
	\end{align}
\end{prop}
\begin{proof}
We first prove, that
	\begin{align}
	\label{eqproofofa31}
	B \, H = H \, B \qquad \implies \qquad \DD_B \log B.H = B^{-1} \, H
	\end{align}
Therefore, observe that for $B \, H = H \, B$
	\begin{align}
	\label{eqfirstone1}
	\exp(B + H) = \exp(B) \, \exp(H) = \exp(B) \, [\id + H + \mathcal{O}(H^2)] = \exp(B) + \exp(B) \, H + \mathcal{O}(H^2)
	\end{align}
and also by expansion of $\exp(\,)$
	\begin{align}
	\label{eqsecondone1}
	\exp(B + H) = \exp(B) + \DD \exp(B).H + \mathcal{O}(H^2)
	\end{align}
so that by comparison of linear (in $H$) terms of \eqref{eqfirstone1} and \eqref{eqsecondone1} we obtain
	\begin{align}
	\DD \exp(B).H = \exp(B) \, H, \qquad \textnormal{if} \quad B \, H = H \, B.
	\end{align}
Applying the chain rule to $\exp(\log B) = B$ then shows for $B \, H = H \, B$
	\begin{equation}
	\begin{alignedat}{2}
	\exp(\log B) &= B \qquad \implies \qquad \DD \exp(\log B).[\DD_B \log B.H] = H \\
	\implies \qquad \underbrace{\exp(\log B)}_{B} \, [\DD_B \log B.H] &= H \qquad \implies \qquad \DD_B \log B.H = B^{-1} \, H.
	\end{alignedat}
	\end{equation}
Now assume that $D \, B = B \, D$. Then for $H = D \, B + B \, D$ we also obtain 
	\begin{align}
	B \, H = B \, (D \, B + B \, D) = (D \, B + B \, D) \, B = H \, B
	\end{align}
and thus we can use the rule $\DD_B \log B.H = B^{-1} \, H$ for $H = D \, B + B \, D$ to obtain
	\begin{align}
	\DD_B \log B.[B \, D + D \, B] &= B^{-1} \, (B \, D + D \, B) = 2 \, D, \qquad \textnormal{if} \qquad B \, D = D \, B. \qedhere
	\end{align}
\end{proof}
\subsubsection{Richter's and Vall\'{e}e's formula} \label{Appensansour}
From \cite{NeffGhibaLankeit} we have the following
\begin{lem}{\rm (Vall\'{e}e's formula\footnote{In \cite{vallee2008dual} Vall\'{e}e et al.~have given a proof without using a Taylor expansion.} (see
also  \cite{vallee1978lois,vallee2008dual,Kochkin1986,sansour1997theory, moreau1979lois}))} \label{lemA18} \\
Let $\Psi: {\rm Sym}^{++}(3) \rightarrow \R$ and $\widehat \WW\colon\Sym(3) \to \R$ be both differentiable isotropic scalar valued functions. Then for the function $\Psi(B) = \widehat \WW(\log B)$ we have still the chain rule
	\begin{align}
	\DD_B \Psi(B) = \DD_{\log B} \widehat \WW(\log B) \, B^{-1} \qquad \textnormal{while} \qquad \DD_B \log B.H \neq B^{-1} \, H \, .
	\end{align}
\end{lem}
\begin{proof}
For a simple proof using Taylor-expansion, see \cite{NeffGhibaLankeit}.
\end{proof}
\noindent Yet another similar result is given by the
\subsubsection{Golden-Thompson inequality} \label{appgoldenth}
The Golden-Thompson inequality (cf.~Golden \cite{golden65} and Thompson \cite{thompson65}, see also Tao \cite{Tao}) implies that for symmetric matrices $B, D \in \Sym(3)$ the inequality
	\begin{align}
	\label{eqgoldenthompson}
	\tr(\exp(B+D)) \le \tr(\exp(B) \, \exp(D))
	\end{align}
holds. It is motivated by the fact, that for scalar values $b, d \in \R$, we obviously have $\mathrm{e}^{b + d} = \mathrm{e}^{b} \, \mathrm{e}^d$. It is well known, that for \emph{commuting matrices} $B, D$ the same equality holds, i.e.
	\begin{align}
	\label{eqappendixcommute}
	\exp(B + D) = \exp(B) \, \exp(D), \qquad \textnormal{if} \qquad B \, D = D \, B.
	\end{align}
\emph{However}, if $B$ and $D$ do not commute, the equation \eqref{eqappendixcommute} does no longer hold. In fact, Petz \cite{Petz1994} proved that for two symmetric matrices $B, D \in \Sym(3)$, it holds
	\begin{align}
	\tr(\exp(B+D)) = \tr(\exp(B) \, \exp(D)) \qquad \iff \qquad B \, D = D \, B.
	\end{align}

\end{appendix}

\end{document}